\documentclass{amsart}
\usepackage{tikz,hyperref,amsbsy}

\usepackage{enumerate}
\newtheorem{theorem}{Theorem}[section]
\newtheorem{proposition}[theorem]{Proposition}
\newtheorem{lemma}[theorem]{Lemma}
\theoremstyle{definition}
\newtheorem{example}[theorem]{Example}
\numberwithin{equation}{section}

\newcommand{\ba}{\mathbf{a}}
\newcommand{\bb}{\mathbf{b}}
\newcommand{\bu}{\mathbf{u}}
\newcommand{\bv}{\mathbf{v}}
\newcommand{\bsigma}{\boldsymbol{\sigma}}
\newcommand{\btau}{\boldsymbol{\tau}}
\newcommand{\supn}{\sup\nolimits_0}
\newcommand{\infn}{\inf\nolimits_1}

\begin{document}
\title{Unique double base expansions}
\author[V. Komornik]{Vilmos Komornik}
\address{D\'{e}partement de math\'{e}matique,
Universit\'{e} de Strasbourg,
7 rue Ren\'{e} Descartes,
67084 Strasbourg Cedex, France}
\email{komornik@math.unistra.fr}
\author[W. Steiner]{Wolfgang Steiner}
\address{Universit\'{e} Paris Cit\'{e}, CNRS, IRIF, F--75006 Paris, France}
\email{steiner@irif.fr}
\author[Y. Zou]{Yuru Zou}
\address{College of Mathematics and Statistics,
Shenzhen University,
Shenzhen 518060,
People's Republic of China}
\email{yuruzou@szu.edu.cn}

\begin{abstract}
For two real bases $q_0, q_1 > 1$, we consider expansions of real numbers of the form $\sum_{k=1}^{\infty} i_k/(q_{i_1}q_{i_2}\cdots q_{i_k})$ with $i_k \in \{0,1\}$, which we call $(q_0,q_1)$-expansions. 
A sequence $(i_k)$ is called a unique $(q_0,q_1)$-expansion if all other sequences have different values as $(q_0,q_1)$-expansions, and the set of unique $(q_0,q_1)$-expansions is denoted by $U_{q_0,q_1}$. 
In the special case $q_0 = q_1 = q$, the set $U_{q,q}$ is trivial if $q$ is below the golden ratio and uncountable if $q$ is above the Komornik--Loreti constant. 
The curve separating pairs of bases $(q_0, q_1)$ with trivial $U_{q_0,q_1}$ from those with non-trivial $U_{q_0,q_1}$ is the graph of a function $\mathcal{G}(q_0)$ that we call generalized golden ratio. 
Similarly, the curve separating pairs $(q_0, q_1)$ with countable $U_{q_0,q_1}$ from those with uncountable $U_{q_0,q_1}$ is the graph of a function $\mathcal{K}(q_0)$ that we call generalized Komornik--Loreti constant. 
We show that the two curves are symmetric in $q_0$ and $q_1$, that $\mathcal{G}$ and $\mathcal{K}$ are continuous, strictly decreasing, hence almost everywhere differentiable on $(1,\infty)$, and that the Hausdorff dimension of the set of $q_0$ satisfying $\mathcal{G}(q_0)=\mathcal{K}(q_0)$ is zero.
We give formulas for $\mathcal{G}(q_0)$ and $\mathcal{K}(q_0)$ for all $q_0 > 1$, using characterizations of when a binary subshift avoiding a lexicographic interval is trivial, countable, uncountable with zero entropy and uncountable with positive entropy respectively. 
Our characterizations in terms of $S$-adic sequences including Sturmian and the Thue--Morse sequences are simpler than those of Labarca and Moreira (2006) and Glendinning and Sidorov (2015), and are relevant also for other open dynamical systems.


\end{abstract}

\thanks{This work was supported by the Agence Nationale de la Recherche through the projects ANR-18-CE40-0007 and ANR-22-CE40-0011, by the National Natural Science Foundation of China (NSFC) \#11871348 and \#61972265, by Shenzhen Basis Research Project \#JCYJ20180305125521534,  by Natural Science Foundation of Guangdong Province of China \#2023A1515011691, and by the grants CAPES: No.\ 88881.520205/2020-01 and MATH AMSUD: 21-MATH-03. Data sharing not applicable to this article as no datasets were generated or analysed during the current study.}
\maketitle

\section{Introduction}
Non-integer base expansions of real numbers
\begin{equation*}
x = \sum_{k=1}^{\infty}\frac{i_k}{q^k},\quad (i_k)\in\{0,1,\ldots,m\}^\infty,
\end{equation*}
where $q>1$ is a given non-integer real number and $m$ is a given positive integer,
have been intensively studied ever since their introduction by R\'{e}nyi~\cite{Ren1957}.
While most integer base expansions are unique, in non-integer bases a number has typically infinitely many expansions \cite{EggVanEyn1966,ErdJooKom1990,ErdHorJoo1991,Sid2003,Bak2014}.
However, the sets of numbers $\mathcal{U}_q(m)$ with \emph{unique expansions}, also called \emph{univoque sets}, have many interesting  properties \cite{KomLor1998,DevKom2009,All2017,AllKon2021a,AllKon2021b,ZouLiLuKom2021,DevKomLor2022}.
In particular, $\mathcal{U}_q(m)$ is the \emph{survivor set} of a dynamical system with a hole, and this kind of \emph{open dynamical systems} have received a lot of attention in recent years \cite{Alc2014,GleSid2015,Cla2016,BakKon2020,KalKonLanLi2020}, involving connections to several mathematical fields such as fractal geometry, ergodic theory, symbolic dynamics and number theory.

Let us recall a  remarkable  theorem of Glendinning and Sidorov~\cite{GleSid2001} concerning the two-digit case $m=1$, where $\varphi\approx 1.618$ denotes the \emph{golden ratio}, and $q_{KL}\approx 1.787$ denotes the \emph{Komornik--Loreti constant}, i.e., the smallest base $q>1$ in which $x=1$ has a unique expansion:
\begin{itemize}
\item if $1<q\le\varphi$, then $\mathcal{U}_q(1)$ has two elements;
\item if $\varphi<q<q_{KL}$, then $\mathcal{U}_q(1)$ is countably infinite;
\item if $q=q_{KL}$, then $\mathcal{U}_q(1)$ is uncountable but of zero Hausdorff dimension;
\item if $q>q_{KL}$, then $\mathcal{U}_q(1)$ has positive Hausdorff dimension.
\end{itemize}
This theorem was generalized for $m>1$ \cite{Dev2009,DevKom2009,Bak2014,DevKomLor2022}, and the Hausdorff dimension of $\mathcal{U}_q(m)$ as well as the topological entropy of the underlying set of digit sequenes $U_q(m)$ have been the subject of a large number of research articles \cite{KomKonLi2017,AlcBarBakKon2019,AllKon2019,AllBakKon2019,KalKonLiLu2019,KonLiLuWanXu2020,AllKon2021a,AllKon2021b}. 
In particular, we know that the topological entropy of $U_q(m)$ is constant (as a function of~$q$) on infinitely many disjoint intervals \cite{KomKonLi2017,AlcBarBakKon2019}, the first of these entropy plateaus being $(1, q_{KL}]$.

The main tools in these investigations are the lexicographic characterizations of unique and related expansions that generalize a classical theorem of Parry \cite{Par1960}.
They have been extended by Pedicini~\cite{Ped2005} to more general alphabets $\{d_0,d_1,\ldots, d_m\}$ with \emph{real digits}.
Based on this theorem, the threshold separating bases with trivial and non-trivial univoque sets, called \emph{generalized golden ratio}, was determined in~\cite{KomLaiPed2011} (and later in~\cite{BakSte2017}) for all ternary alphabets $\{d_0, d_1, d_2\}$.
The determination of the next threshold, which separates bases with countable and uncountable univoque sets and is called \emph{generalized  Komornik--Loreti constant}, seems to be a difficult problem for ternary alphabets, but a closely related question was solved partially in~\cite{KomPed2017} and then completely in~\cite{Ste2020}.

Recently, some of the preceding results were generalized to \emph{multiple-base} expansions of the form
\begin{equation*}
x = \sum_{k=1}^{\infty}\frac{i_k}{q_{i_1}q_{i_2}\cdots q_{i_k}},
\quad (i_k)\in\{0,1,\ldots,m\}^\infty,
\end{equation*}
where $m$ is a positive integer and $q_0,q_1,\ldots, q_m>1$ are given bases \cite{Li2021,Neu2021,Neu2024}, and furthermore to expansions of the form
\begin{equation*}
x = \sum_{k=1}^{\infty}\frac{d_{i_k}}{q_{i_1}q_{i_2}\cdots q_{i_k}},
\quad (i_k)\in\{0,1,\ldots,m\}^\infty,
\end{equation*}
where $\mathcal{S} = \{(d_0,q_0), (d_1,q_1),\ldots,(d_m,q_m)\}$  is a given finite \emph{alphabet--base system} of pairs of real numbers with $q_0,q_1,\ldots,q_m>1$ \cite{KomLuZou2022}; this contains all preceding expansions as special cases.

The purpose of this paper is to extend results on the cardinality of univoque sets to alphabet--base systems.
We restrict to the binary case $\mathcal{S} = \{(d_0,q_0), (d_1, q_1)\}$, as we have seen above that the ternary case is already difficult when all bases are equal.
We first show that this 4-dimensional problem can be reduced to 2~dimensions because the system $\{(d_0,q_0), (d_1, q_1)\}$ is isomorphic to $\{(0,q_0), (1, q_1)\}$ (except in the degenerate case $\frac{d_0}{q_0-1} = \frac{d_1}{q_1-1}$, where $x = \frac{d_0}{q_0-1}$ for all expansions).
Now, we for each fixed $q_0 > 1$, two critical values $\mathcal{G}(q_0), \mathcal{K}(q_0) > 1$, called generalized golden ratio and generalized Komornik--Loreti constant, which separate pairs of bases $(q_0,q_1)$ according to whether the univoque set is trivial, uncountable or in-between. 
We give various properties of the functions $\mathcal{G}$ and 
$\mathcal{K}$ and formulas for $\mathcal{G}(q_0)$ and $\mathcal{K}(q_0)$ for all $q_0 > 1$ in Theorems~\ref{t:1}--\ref{t:3} below. 

To prove our main results, we study the critical values of the survivor set in dynamical systems with a hole.
More precisely, we characterize in Theorem~\ref{t:lex} below when a binary subshift never hitting a lexicographic interval is trivial, countable, uncountable with zero entropy and uncountable with positive entropy  respectively. 
This improves and simplifies results of Labarca and Moreira~\cite{LabMor2006}, which were used as tools for understanding the properties of Lorenz maps. 
Our characterizations are also simpler than those of Glendinning and Sidorov \cite{Sid2014,GleSid2015}, who investigated the critical values of the symmetric and asymmetric holes for the doubling map.
Note that positive entropy of binary subshifts with a hole plays also a crucial role in the characterization of critical itineraries of maps with constant slope and one discontinuity, also called intermediate $\beta$-shifts~\cite{BarSteVin2014}.
Based on our results, the first entropy plateau is also given by the generalized Komornik--Loreti constant. 
The investigation of further fractal and dynamical properties of the univoque set is left as an open problem.

\section{Statement of the main results} \label{s:intro0}
Fix a finite \emph{alphabet--base system} $\mathcal{S}:=\{(d_0,q_0),(d_1,q_1),\ldots, (d_m, q_m)\}$, and set
\begin{equation*}
\pi_{\mathcal{S}}((i_k)) := \sum_{k=1}^{\infty} \frac{d_{i_k}}{q_{i_1}q_{i_2}\cdots q_{k_k}}, \quad (i_k)\in\{0,1,\ldots,m\}^\infty.
\end{equation*}
If $\pi_{\mathcal{S}}((i_k))=x$, then $(i_k)$ is called an \emph{$\mathcal{S}$-expansion} of the real number $x$.
The system $\mathcal{S}$ is called \emph{regular} if
\begin{align*}
&\pi_{\mathcal{S}}(\overline{0}) < \pi_{\mathcal{S}}(1\overline{0}) < \cdots < \pi_{\mathcal{S}}(m\overline{0}), \\ &\pi_{\mathcal{S}}(0\overline{m}) < \pi_{\mathcal{S}}(1\overline{m}) < \cdots < \pi_{\mathcal{S}}(\overline{m}), \\
&\pi_{\mathcal{S}}((i{+}1)\overline{0}) \le \pi_{\mathcal{S}}(i\,\overline{m})  \text{ for all}\ 0 \le i < m.
\end{align*}
Here and in the following, $\overline{c}$ denotes the infinite repetition of a digit (or a finite sequence of digits)~$c$.
Regular systems~$\mathcal{S}$ were studied in~\cite{KomLuZou2022}, where it was shown that $x$ has an $\mathcal{S}$-expansion if and only if $x\in [\pi_{\mathcal{S}}(\overline{0}), \pi_{\mathcal{S}}(\overline{m})]$.
Let
\begin{equation*}
U_{\mathcal{S}} := \big\{\bu \in \{0,1,\ldots, m\}^\infty \,:\, \pi_{\mathcal{S}}(\bu) \ne \pi_{\mathcal{S}}(\bv) \ \text{for all} \ \bv \ne \bu\big\}
\end{equation*}
be the set of \emph{unique $\mathcal{S}$-expansions}.
The points $\pi_{\mathcal{S}}(\overline{0})$ and $\pi_{\mathcal{S}}(\overline{m})$ trivially have unique $\mathcal{S}$-expansions, and we call $U_{\mathcal{S}}$ \emph{trivial} if $U_{\mathcal{S}} = \{\overline{0}, \overline{m}\}$.
The set $U_{\mathcal{S}}$ is shift-invariant, and we denote its \emph{(topological) entropy} by
\begin{equation*}
h(U_{\mathcal{S}}) := \lim_{n\rightarrow\infty}\frac{1}{n} \log A_n(U_{\mathcal{S}}),
\end{equation*}
where $A_n(U_{\mathcal{S}})$ is the number of different blocks of $n$ letters appearing in the sequences of~$U_{\mathcal{S}}$.\footnote{This is the usual definition of topological entropy for subshifts. The set $U_{\mathcal{S}}$ need not be closed but taking the closure adds only countably many points, hence the entropy defined by Bowen~\cite{Bow1973} is equal to the entropy of the closure.}

In this paper, we study regular systems $\mathcal{S} = \{(d_0,q_0),(d_1,q_1)\}$ with $d_0,d_1 \in \mathbb{R}$, $q_0,q_1 > 1$; the case of larger systems is more complex.
We show in Lemma~\ref{l:01} that the structure of the $\mathcal{S}$-expansions is isomorphic to those in the alphabet-base system $\{(0,q_0),(1,q_1)\}$.
Hence in the following we can restrict ourselves without loss of generality to the case where $d_0 = 0$ and $d_1 = 1$.
For simplicity, we write $\pi_{q_0,q_1}$ and $U_{q_0,q_1}$ instead of $\pi_{\{(0,q_0),(1,q_1)\}}$ and $U_{\{(0,q_0),(1,q_1)\}}$, i.e.,
\begin{equation*}
\begin{aligned}
& \pi_{q_0,q_1}(i_1i_2\cdots) = \sum_{k=1}^\infty \frac{i_k}{q_{i_1}q_{i_2}\cdots q_{i_k}}, \\ & U_{q_0, q_1} = \big\{\bu \in \{0,1\}^\infty \,:\, \pi_{q_0,q_1}(\bu) \ne \pi_{q_0,q_1}(\bv) \ \text{for all} \ \bv \ne \bu\big\}.
\end{aligned}
\end{equation*}
The goal of this paper is to characterize
\begin{equation*}
\begin{aligned}
\mathcal{G}(q_0) & := \inf\{q_1>1 \,:\, U_{q_0, q_1} \ne \{\overline{0},\overline{1}\}\}, \\
\mathcal{K}(q_0) & := \inf\{q_1>1 \,:\, U_{q_0, q_1} \text{ is uncountable}\}.
\end{aligned}
\end{equation*}
We first state results that do not require further notation, before going into details.

\begin{theorem} \label{t:1} \mbox{}

\begin{enumerate}[\upshape (i)]
\itemsep.5ex
\item \label{i:t11}
The functions $\mathcal{G}$ and $\mathcal{K}$ are continuous, strictly decreasing on $(1,\infty)$, and hence almost everywhere differentiable.
\item \label{i:t12}
For all $q_0>1$, we have
\begin{equation*}
\begin{aligned}
\mathcal{G}(\mathcal{G}(q_0)) & = q_0, & & \max\bigg\{\frac{1}{q_0+1}, \frac{1}{\mathcal{G}(q_0)+1}\bigg\} \le (q_0-1) \big(\mathcal{G}(q_0)-1\big) \le \frac{1}{2}, \\
\mathcal{K}(\mathcal{K}(q_0)) & = q_0, & & \frac{1}{2} \le (q_0-1) \big(\mathcal{K}(q_0)-1\big) < \min\bigg\{\frac{q_0}{q_0+1}, \frac{\mathcal{K}(q_0)}{\mathcal{K}(q_0)+1}\bigg\}. \\
\end{aligned}
\end{equation*}
\item \label{i:t13}
For $q_0 > 1$, we have
\begin{equation*}
\begin{gathered}
(q_0-1) \big(\mathcal{G}(q_0)-1\big) = \frac{1}{2} \ \Longleftrightarrow \ \mathcal{G}(q_0) = \mathcal{K}(q_0) \ \Longleftrightarrow \ (q_0-1) \big(\mathcal{K}(q_0)-1\big) = \frac{1}{2}, \\
\hspace{.25em} (q_0-1)(\mathcal{G}(q_0)-1) = \left\{\hspace{-.5em}\begin{array}{cl} \frac{1}{q_0+1} & \hspace{-.5em}\text{if and only if $q_0^{k} = q_0{+}1$ for some integer $k \ge 2$,}\\
\frac{1}{\mathcal{G}(q_0)+1} & \hspace{-.5em}\text{if and only if $q_0 = \frac{q_1^2}{q_1^2-1}$ with $q_1>1$ such that} \\ & \quad \text{$q_1^k=q_1{+}1$ for some integer $k \ge 2$.}
\end{array}\right.
\end{gathered}
\end{equation*}
\item \label{i:14_2} 
The Hausdorff dimension of $\{q_0 > 1 \,:\, \mathcal{G}(q_0) = \mathcal{K}(q_0)\}$ is zero.
\item \label{i:t14}
For all $q_0 > 1$, $q_1>\mathcal{G}(q_0)$, the set $U_{q_0,q_1}$ is infinite.
\item \label{i:t15}
For all $q_0 > 1$, $q_1>\mathcal{K}(q_0)$, the shift-invariant set $U_{q_0,q_1}$ has positive entropy, and $\pi_{q_0, q_1}(U_{q_0, q_1})$ has positive Hausdorff dimension.
\end{enumerate}
\end{theorem}

Our main result is a precise description of $\mathcal{G}$ and~$\mathcal{K}$ in terms of equations $q_1\, \pi_{q_0,q_1}(\bu) = 1$ and $q_0\, \tilde{\pi}_{q_0,q_1}(\bv) = 1$ for certain expansions $\bu, \bv \in \{0,1\}^\infty$, where
\begin{equation*}
\tilde{\pi}_{q_0,q_1}(i_1i_2\cdots) := \sum_{k=1}^\infty \frac{1-i_k}{q_{i_1}q_{i_2}\cdots q_{i_k}} = \pi_{q_1,q_0}((1{-}i_1)(1{-}i_2)\cdots).
\end{equation*}
If the equation $q_1\, \pi_{q_0,q_1}(\bu) = 1$ has a unique solution $q_1 > 1$, then we denote this solution by~$g_{\bu}(q_0)$.
If the equation $q_0\, \tilde{\pi}_{q_0,q_1}(\bv) = 1$ has a unique solution $q_1 > 1$, then we denote this solution by~$\tilde{g}_{\bv}(q_0)$.
If the equation $g_{\bu}(q_0) = \tilde{g}_{\bv}(q_0)$ has a unique solution $q_0 > 1$, then we denote this solution by~$\mu_{\bu,\bv}$.
In Lemmas~\ref{l:functiong}, \ref{l:functiongt} and~\ref{l:mu}, we will show that $g_{\bu}(q_0)$, $\tilde{g}_{\bv}(q_0)$ and $\mu_{\bu,\bv}$ are well defined for all $\bu, \bv, q_0$ that are relevant for us.

The involved words $\bu, \bv$ are defined by the substitutions (or morphisms)
\begin{equation*}
\begin{aligned}
L:\ & 0 \mapsto 0, & \quad M:\ & 0 \mapsto 01, & \quad R:\ & 0 \mapsto 01, \\
& 1 \mapsto 10, & & 1 \mapsto 10, & & 1 \mapsto 1,
\end{aligned}
\end{equation*}
which act on finite and infinite words by $\sigma(i_1i_2\cdots) = \sigma(i_1) \sigma(i_2) \cdots$.
We use the notation $S^*$ for the monoid generated by a set of substitutions~$S$ (with the composition as product).
For a sequence of substitutions $\bsigma = (\sigma_n)_{n\ge1} \in \{L,M,R\}^\infty$ and any $\bu \in \{0,1\}^\infty$, the \emph{limit word}
\begin{equation*}
\bsigma(\bu) := \lim_{n\to\infty} \sigma_1\sigma_2 \cdots \sigma_n(\bu)
\end{equation*}
exists because $\sigma(i)$ starts with~$i$ for all $\sigma \in \{L,M,R\}$, $i \in \{0,1\}$.
A~sequence of substitutions $(\sigma_n)_{n\ge1}$ is \emph{primitive} if for each $n \ge 1$ there exists $k \ge n$ such that the image by $\sigma_n \sigma_{n+1} \cdots \sigma_k$ of each letter contains all letters; for $\bsigma \in \{L,M,R\}^\infty$, this means that $\bsigma$ does not end with $\overline{L}$ or~$\overline{R}$.
Note that the limit words of~$\overline{M}$ are the Thue--Morse word and its reflection by $0 \leftrightarrow 1$; limit words of primitive sequences in $\{L,R\}^\infty$ are \emph{Sturmian} words, and limit words of primitive sequences in $\{L,M,R\}^\infty$ are called \emph{Thue--Morse--Sturmian words} according to~\cite{Ste2020}.\footnote{The substitution $L$ in~\cite{Ste2020} is rotationally conjugate to our substitution~$L$, thus the limit words have the same dynamical properties.}

\begin{theorem} \label{t:2}
The map $\mathcal{G}$ is given on $(1,\infty)$ by
\begin{equation}\label{e:ggr}
\mathcal{G}(q_0) = \begin{cases}g_{\sigma(\overline{0})}(q_0) & \text{if}\ q_0 \in [\mu_{\sigma(\overline{0}),\sigma(1\overline{0})}, \mu_{\sigma(\overline{0}),\sigma(\overline{1})}],\, \sigma \in \{L,R\}^*M, \\[.5ex]
\tilde{g}_{\sigma(\overline{1})}(q_0) & \text{if}\ q_0 \in [\mu_{\sigma(\overline{0}),\sigma(\overline{1})}, \mu_{\sigma(0\overline{1}),\sigma(\overline{1})}],\, \sigma \in \{L, R\}^*M, \\[.5ex]
g_{\bsigma(\overline{0})}(q_0) & \text{if}\ q_0 = \mu_{\bsigma(\overline{0}),\bsigma(\overline{1})},\, \bsigma \in \{L,R\}^\infty\ \text{primitive}. \end{cases}
\end{equation}
The map $\mathcal{K}$ is given on $(1, \infty)$ by
\begin{equation}\label{e:klc}
\mathcal{K}(q_0) =
\begin{cases}\tilde{g}_{\sigma(1\overline{0})}(q_0) & \text{if}\ q_0 \in [\mu_{\sigma(\overline{0}),\sigma(1\overline{0})}, \mu_{\sigma(01\overline{0}),\sigma(1\overline{0})}],\, \sigma \in \{L,M,R\}^*M, \\[.5ex]
g_{\sigma(0\overline{1})}(q_0) & \text{if}\ q_0 \in [\mu_{\sigma(0\overline{1}),\sigma(10\overline{1})}, \mu_{\sigma(0\overline{1}),\sigma(\overline{1})}],\, \sigma \in \{L,M,R\}^*M, \\[.5ex]
g_{\bsigma(\overline{0})}(q_0) & \text{if}\ q_0 = \mu_{\bsigma(\overline{0}),\bsigma(\overline{1})},\, \bsigma \in \{L,M,R\}^\infty\ \text{primitive}. \end{cases}
\end{equation}
\end{theorem}

Note that $\mathcal{G}(\mu_{\bsigma(\overline{0}),\bsigma(\overline{1})}) = g_{\bsigma(\overline{0})}(\mu_{\bsigma(\overline{0}),\bsigma(\overline{1})})$ holds for non-primitive $\bsigma \in \{L,R\}^\infty \setminus \{\overline{L}, \overline{R}\}$ too because $L \overline{R}(\overline{0}) = M(\overline{0}) = R \overline{L}(\overline{0})$ and $L \overline{R}(\overline{1}) = M(\overline{1}) = R \overline{L}(\overline{1})$. 
However, we need $\bsigma \in \{L,M,R\}^\infty$ to be primitive for $\mathcal{K}(\mu_{\bsigma(\overline{0}),\bsigma(\overline{1})}) = g_{\bsigma(\overline{0})}(\mu_{\bsigma(\overline{0}),\bsigma(\overline{1})})$.

\begin{theorem} \label{t:3} \mbox{}
\begin{enumerate}[\upshape (i)]
\itemsep.5ex
\item \label{i:t31}
We have $\mathcal{G}(q_0) \,{=}\, \mathcal{K}(q_0)$ if and only if $q_0 \,{=}\, \mu_{\bsigma(\overline{0}),\bsigma(\overline{1})}$ for a primitive $\bsigma \,{\in}\, \{L,R\}^\infty$ or $q_0 \in \{\mu_{\sigma(\overline{0}),\sigma(1\overline{0})}, \mu_{\sigma(0\overline{1}),\sigma(\overline{1})}\}$ for some $\sigma \in \{L, R\}^*M$; this is also equivalent to $(q_0{-}1)(\mathcal{G}(q_0){-}1) = \frac{1}{2}$ as well as to $(q_0{-}1)(\mathcal{K}(q_0){-}1) = \frac{1}{2}$.
In particular, we have $\mathcal{G}(q_0) < \mathcal{K}(q_0)$ for all $q_0 > 1$ except a set of zero Hausdorff dimension.
\item \label{i:t32}
$U_{q_0,\mathcal{G}(q_0)} \ne \{\overline{0}, \overline{1}\}$ if and only if $q_0 = \mu_{\bsigma(\overline{0}),\bsigma(\overline{1})}$ for a primitive $\bsigma \in \{L,R\}^\infty$; the set of $q_0 > 1$ with this property has zero Hausdorff dimension.
\item \label{i:t33}
$U_{q_0,\mathcal{K}(q_0)}$ is trivial if and only if $q_0 \in \bigcup_{\sigma\in\{L, R\}^*M} \{\mu_{\sigma(\overline{0}),\sigma(1\overline{0})}, \mu_{\sigma(0\overline{1}),\sigma(\overline{1})}\}$, uncountable with zero entropy if and only if $q_0 = \mu_{\bsigma(\overline{0}),\bsigma(\overline{1})}$ for some primitive $\bsigma \in \{L,M,R\}^\infty$, countably infinite otherwise.
\end{enumerate}
\end{theorem}

We conjecture that the set of $q_0 > 1$ where $U_{q_0,\mathcal{K}(q_0)}$ is not countably infinite has zero Hausdorff dimension (and thus zero Lebesgue measure); see the Open Problem~(\ref{op:1}) in Section~\ref{sec:op}.

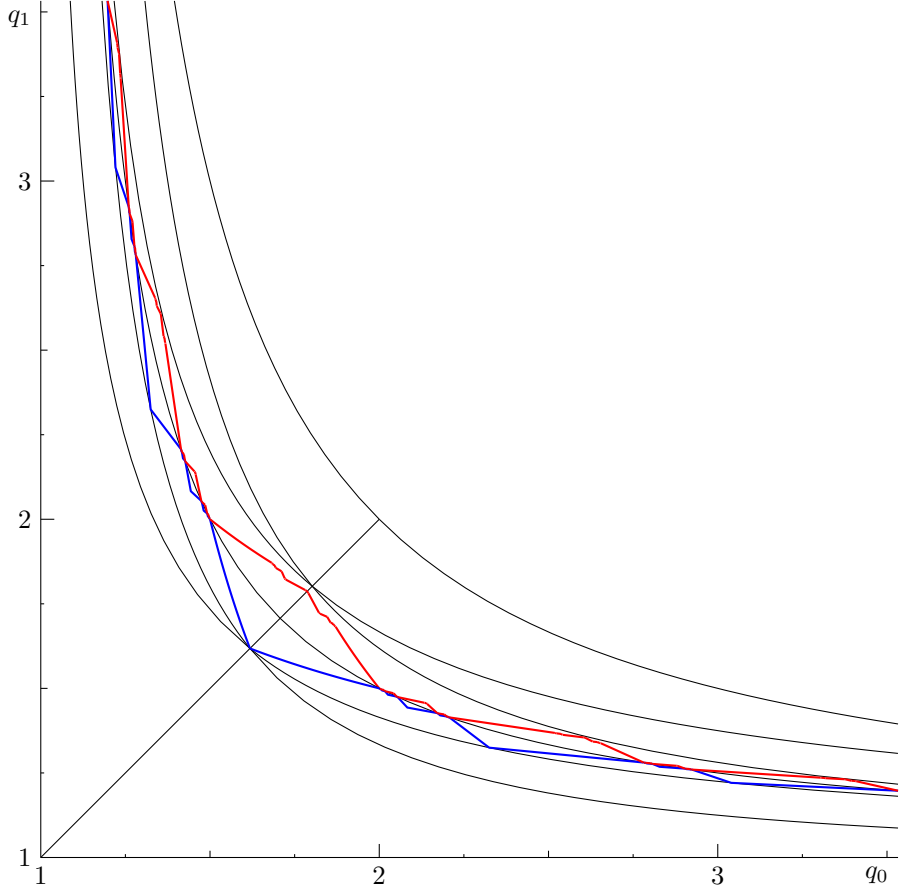
\begin{figure}[ht]
\centering
\begin{tikzpicture}[scale=4.5,join=bevel]
\draw(1,3.5326)node[below left]{$q_1$}--(1,1)node[below]{1}node[left]{1}--(3.5326,1)node[below left]{$q_0$};
\foreach \i in {2,3} \draw(1.04,\i)--(1,\i)node[left]{\i} (\i,1.04)--(\i,1)node[below]{\i};
\foreach \i in {1.5,2.5,3.5} \draw(1.02,\i)--(1,\i) (\i,1.02)--(\i,1);
\foreach \i in {1.25,1.75,...,3.25} \draw(1.01,\i)--(1,\i) (\i,1.01)--(\i,1);
\draw(1,1)--(2,2);

\draw[domain=1.18104:1.618] plot (\x,{1/(\x^2-1)+1});
\draw[domain=1.618:3.5326] plot ({1/(\x^2-1)+1},\x);
\draw[domain=1.70711:3.5326] plot ({.5/(\x-1)+1},\x);
\draw[domain=1.21673:1.80194] plot (\x,{1+\x/(\x^2-1)});
\draw[domain=1.80194:3.5326] plot ({1+\x/(\x^2-1)},\x);
\draw[domain=2:3.5326] plot ({\x/(\x-1)},\x);

\draw[thick,blue](1.1974,3.5326)--(1.2207,3.0399)--(1.2599,2.9237); 
\draw[thick,blue](1.26,2.9231)--(1.2601,2.9209)--(1.2604,2.9203); 
\draw[thick,blue](1.2604,2.9203)--(1.2613,2.9077)--(1.2626,2.904); 
\draw[thick,blue](1.2627,2.9034)--(1.2682,2.8282)--(1.2767,2.8071); 
\draw[thick,blue](1.2768,2.8064)--(1.2781,2.7908)--(1.2799,2.7864); 
\draw[thick,blue](1.2799,2.7863)--(1.2802,2.7829)--(1.2806,2.782); 
\draw[thick,blue](1.2808,2.7808)--(1.3247,2.3247)--(1.4142,2.2071); 
\draw[thick,blue](1.4143,2.2067)--(1.4145,2.2057)--(1.4148,2.2054); 
\draw[thick,blue](1.4148,2.2054)--(1.4155,2.201)--(1.4169,2.1994); 
\draw[thick,blue](1.4169,2.1993)--(1.4202,2.1799)--(1.4262,2.1731); 
\draw[thick,blue](1.4262,2.173)--(1.4265,2.1716)--(1.4269,2.1712); 
\draw[thick,blue](1.427,2.171)--(1.4433,2.083)--(1.4743,2.0542); 
\draw[thick,blue](1.4744,2.054)--(1.4748,2.0521)--(1.4755,2.0514); 
\draw[thick,blue](1.4756,2.0514)--(1.4813,2.0256)--(1.4917,2.0169); 
\draw[thick,blue](1.4918,2.0166)--(1.4938,2.0083)--(1.4972,2.0055); 
\draw[thick,blue](1.4973,2.0055)--(1.4979,2.0028)--(1.4991,2.0018); 
\draw[thick,blue](1.4991,2.0018)--(1.4993,2.0009)--(1.4997,2.0006); 
\draw[thick,blue,domain=1.5:1.61804] plot (\x,{1/(\x-1)}); 

\draw[thick,red](1.1974,3.5326)--(1.2255,3.409) (1.236,3.3033)--(1.2599,2.9237); 
\draw[thick,red](1.2256,3.4086)--(1.2257,3.4082) (1.2257,3.4082)--(1.2257,3.4067); 
\draw[thick,red](1.2257,3.4067)--(1.2264,3.4039) (1.2264,3.4039)--(1.2269,3.3946); 
\draw[thick,red](1.2269,3.3943)--(1.2312,3.3769) (1.2313,3.3749)--(1.2345,3.3178); 
\draw[thick,red](1.2312,3.3768)--(1.2313,3.3764) (1.2313,3.3764)--(1.2313,3.375); 
\draw[thick,red](1.2345,3.3174)--(1.2352,3.3147) (1.2352,3.3147)--(1.2358,3.3056); 
\draw[thick,red](1.2358,3.3056)--(1.2359,3.3051) (1.2359,3.3051)--(1.236,3.3037); 
\draw[thick,red](1.26,2.9231)--(1.2602,2.9225) (1.2602,2.9224)--(1.2604,2.9203); 
\draw[thick,red](1.2604,2.9203)--(1.2617,2.9166) (1.2617,2.9165)--(1.2626,2.904); 
\draw[thick,red](1.2627,2.9034)--(1.2706,2.8819) (1.2711,2.8784)--(1.2767,2.8071); 
\draw[thick,red](1.2706,2.8818)--(1.2709,2.881) (1.2709,2.881)--(1.2711,2.8785); 
\draw[thick,red](1.2768,2.8064)--(1.2786,2.8019) (1.2786,2.8018)--(1.2799,2.7864); 
\draw[thick,red](1.2799,2.7863)--(1.2803,2.7854) (1.2803,2.7854)--(1.2806,2.782); 
\draw[thick,red](1.2808,2.7808)--(1.3377,2.6547) (1.3683,2.5188)--(1.4142,2.2071); 
\draw[thick,red](1.3379,2.6536)--(1.3384,2.6526) (1.3384,2.6526)--(1.3387,2.6498); 
\draw[thick,red](1.3387,2.6497)--(1.3411,2.645) (1.3411,2.6448)--(1.3426,2.6313); 
\draw[thick,red](1.3428,2.6304)--(1.3541,2.608) (1.355,2.6041)--(1.3627,2.5419); 
\draw[thick,red](1.3541,2.6078)--(1.3546,2.6069) (1.3546,2.6069)--(1.3549,2.6042); 
\draw[thick,red](1.3629,2.5411)--(1.3654,2.5366) (1.3654,2.5364)--(1.3671,2.5235); 
\draw[thick,red](1.3671,2.5235)--(1.3677,2.5225) (1.3677,2.5225)--(1.368,2.5198); 
\draw[thick,red](1.4143,2.2067)--(1.4146,2.2064) (1.4146,2.2064)--(1.4148,2.2054); 
\draw[thick,red](1.4148,2.2054)--(1.4161,2.2038) (1.4161,2.2038)--(1.4169,2.1994); 
\draw[thick,red](1.4169,2.1993)--(1.4228,2.1924) (1.4229,2.192)--(1.4262,2.1731); 
\draw[thick,red](1.4262,2.173)--(1.4267,2.1725) (1.4267,2.1725)--(1.4269,2.1712); 
\draw[thick,red](1.427,2.171)--(1.4542,2.1413) (1.4576,2.1328)--(1.4743,2.0542); 
\draw[thick,red](1.4544,2.1407)--(1.4563,2.1388) (1.4563,2.1387)--(1.4574,2.1334); 
\draw[thick,red](1.4744,2.054)--(1.4751,2.0533) (1.4751,2.0533)--(1.4755,2.0514); 
\draw[thick,red](1.4756,2.0514)--(1.4854,2.0426) (1.4858,2.0417)--(1.4917,2.0169); 
\draw[thick,red](1.4918,2.0166)--(1.4952,2.0138) (1.4952,2.0137)--(1.4972,2.0055); 
\draw[thick,red](1.4973,2.0055)--(1.4984,2.0046) (1.4984,2.0046)--(1.4991,2.0018); 
\draw[thick,red,domain=1.8712:2] plot (\x,{(2*\x-1)/(\x^2-\x)}); 
\draw[thick,red](1.6827,1.8707)--(1.6832,1.8704) (1.6832,1.8704)--(1.6835,1.8697); 
\draw[thick,red](1.6835,1.8697)--(1.6851,1.8687) (1.6851,1.8687)--(1.686,1.8666); 
\draw[thick,red](1.686,1.8665)--(1.6912,1.8633) (1.6913,1.8632)--(1.694,1.8567); 
\draw[thick,red](1.6942,1.8563)--(1.711,1.8464) (1.7118,1.8454)--(1.7207,1.8251); 
\draw[thick,red](1.711,1.8463)--(1.7115,1.846) (1.7115,1.846)--(1.7118,1.8454); 
\draw[thick,red](1.7208,1.825)--(1.7225,1.824) (1.7226,1.824)--(1.7235,1.822); 
\draw[thick,red](1.781,1.7908)--(1.7816,1.7905) (1.7816,1.7905)--(1.7818,1.79); 
\draw[thick,red](1.7238,1.8217)--(1.781,1.7909); 
\draw[thick,red](1.7818,1.79)--(1.7872,1.7873); 
\draw[thick,red](1.7872,1.7873)--(1.7872,1.7872); 

\begin{scope}[cm={0,1,1,0,(0,0)}]
\draw[domain=1.18104:1.618] plot (\x,{1/(\x^2-1)+1});
\draw[domain=1.618:3.5326] plot ({1/(\x^2-1)+1},\x);
\draw[domain=1.70711:3.5326] plot ({.5/(\x-1)+1},\x);
\draw[domain=1.21673:1.80194] plot (\x,{1+\x/(\x^2-1)});
\draw[domain=1.80194:3.5326] plot ({1+\x/(\x^2-1)},\x);
\draw[domain=2:3.5326] plot ({\x/(\x-1)},\x);

\draw[thick,blue](1.1974,3.5326)--(1.2207,3.0399)--(1.2599,2.9237); 
\draw[thick,blue](1.26,2.9231)--(1.2601,2.9209)--(1.2604,2.9203); 
\draw[thick,blue](1.2604,2.9203)--(1.2613,2.9077)--(1.2626,2.904); 
\draw[thick,blue](1.2627,2.9034)--(1.2682,2.8282)--(1.2767,2.8071); 
\draw[thick,blue](1.2768,2.8064)--(1.2781,2.7908)--(1.2799,2.7864); 
\draw[thick,blue](1.2799,2.7863)--(1.2802,2.7829)--(1.2806,2.782); 
\draw[thick,blue](1.2808,2.7808)--(1.3247,2.3247)--(1.4142,2.2071); 
\draw[thick,blue](1.4143,2.2067)--(1.4145,2.2057)--(1.4148,2.2054); 
\draw[thick,blue](1.4148,2.2054)--(1.4155,2.201)--(1.4169,2.1994); 
\draw[thick,blue](1.4169,2.1993)--(1.4202,2.1799)--(1.4262,2.1731); 
\draw[thick,blue](1.4262,2.173)--(1.4265,2.1716)--(1.4269,2.1712); 
\draw[thick,blue](1.427,2.171)--(1.4433,2.083)--(1.4743,2.0542); 
\draw[thick,blue](1.4744,2.054)--(1.4748,2.0521)--(1.4755,2.0514); 
\draw[thick,blue](1.4756,2.0514)--(1.4813,2.0256)--(1.4917,2.0169); 
\draw[thick,blue](1.4918,2.0166)--(1.4938,2.0083)--(1.4972,2.0055); 
\draw[thick,blue](1.4973,2.0055)--(1.4979,2.0028)--(1.4991,2.0018); 
\draw[thick,blue](1.4991,2.0018)--(1.4993,2.0009)--(1.4997,2.0006); 
\draw[thick,blue,domain=1.5:1.61804] plot (\x,{1/(\x-1)}); 

\draw[thick,red](1.1974,3.5326)--(1.2255,3.409) (1.236,3.3033)--(1.2599,2.9237); 
\draw[thick,red](1.2256,3.4086)--(1.2257,3.4082) (1.2257,3.4082)--(1.2257,3.4067); 
\draw[thick,red](1.2257,3.4067)--(1.2264,3.4039) (1.2264,3.4039)--(1.2269,3.3946); 
\draw[thick,red](1.2269,3.3943)--(1.2312,3.3769) (1.2313,3.3749)--(1.2345,3.3178); 
\draw[thick,red](1.2312,3.3768)--(1.2313,3.3764) (1.2313,3.3764)--(1.2313,3.375); 
\draw[thick,red](1.2345,3.3174)--(1.2352,3.3147) (1.2352,3.3147)--(1.2358,3.3056); 
\draw[thick,red](1.2358,3.3056)--(1.2359,3.3051) (1.2359,3.3051)--(1.236,3.3037); 
\draw[thick,red](1.26,2.9231)--(1.2602,2.9225) (1.2602,2.9224)--(1.2604,2.9203); 
\draw[thick,red](1.2604,2.9203)--(1.2617,2.9166) (1.2617,2.9165)--(1.2626,2.904); 
\draw[thick,red](1.2627,2.9034)--(1.2706,2.8819) (1.2711,2.8784)--(1.2767,2.8071); 
\draw[thick,red](1.2706,2.8818)--(1.2709,2.881) (1.2709,2.881)--(1.2711,2.8785); 
\draw[thick,red](1.2768,2.8064)--(1.2786,2.8019) (1.2786,2.8018)--(1.2799,2.7864); 
\draw[thick,red](1.2799,2.7863)--(1.2803,2.7854) (1.2803,2.7854)--(1.2806,2.782); 
\draw[thick,red](1.2808,2.7808)--(1.3377,2.6547) (1.3683,2.5188)--(1.4142,2.2071); 
\draw[thick,red](1.3379,2.6536)--(1.3384,2.6526) (1.3384,2.6526)--(1.3387,2.6498); 
\draw[thick,red](1.3387,2.6497)--(1.3411,2.645) (1.3411,2.6448)--(1.3426,2.6313); 
\draw[thick,red](1.3428,2.6304)--(1.3541,2.608) (1.355,2.6041)--(1.3627,2.5419); 
\draw[thick,red](1.3541,2.6078)--(1.3546,2.6069) (1.3546,2.6069)--(1.3549,2.6042); 
\draw[thick,red](1.3629,2.5411)--(1.3654,2.5366) (1.3654,2.5364)--(1.3671,2.5235); 
\draw[thick,red](1.3671,2.5235)--(1.3677,2.5225) (1.3677,2.5225)--(1.368,2.5198); 
\draw[thick,red](1.4143,2.2067)--(1.4146,2.2064) (1.4146,2.2064)--(1.4148,2.2054); 
\draw[thick,red](1.4148,2.2054)--(1.4161,2.2038) (1.4161,2.2038)--(1.4169,2.1994); 
\draw[thick,red](1.4169,2.1993)--(1.4228,2.1924) (1.4229,2.192)--(1.4262,2.1731); 
\draw[thick,red](1.4262,2.173)--(1.4267,2.1725) (1.4267,2.1725)--(1.4269,2.1712); 
\draw[thick,red](1.427,2.171)--(1.4542,2.1413) (1.4576,2.1328)--(1.4743,2.0542); 
\draw[thick,red](1.4544,2.1407)--(1.4563,2.1388) (1.4563,2.1387)--(1.4574,2.1334); 
\draw[thick,red](1.4744,2.054)--(1.4751,2.0533) (1.4751,2.0533)--(1.4755,2.0514); 
\draw[thick,red](1.4756,2.0514)--(1.4854,2.0426) (1.4858,2.0417)--(1.4917,2.0169); 
\draw[thick,red](1.4918,2.0166)--(1.4952,2.0138) (1.4952,2.0137)--(1.4972,2.0055); 
\draw[thick,red](1.4973,2.0055)--(1.4984,2.0046) (1.4984,2.0046)--(1.4991,2.0018); 
\draw[thick,red,domain=1.8712:2] plot (\x,{(2*\x-1)/(\x^2-\x)}); 
\draw[thick,red](1.6827,1.8707)--(1.6832,1.8704) (1.6832,1.8704)--(1.6835,1.8697); 
\draw[thick,red](1.6835,1.8697)--(1.6851,1.8687) (1.6851,1.8687)--(1.686,1.8666); 
\draw[thick,red](1.686,1.8665)--(1.6912,1.8633) (1.6913,1.8632)--(1.694,1.8567); 
\draw[thick,red](1.6942,1.8563)--(1.711,1.8464) (1.7118,1.8454)--(1.7207,1.8251); 
\draw[thick,red](1.711,1.8463)--(1.7115,1.846) (1.7115,1.846)--(1.7118,1.8454); 
\draw[thick,red](1.7208,1.825)--(1.7225,1.824) (1.7226,1.824)--(1.7235,1.822); 
\draw[thick,red](1.781,1.7908)--(1.7816,1.7905) (1.7816,1.7905)--(1.7818,1.79); 
\draw[thick,red](1.7238,1.8217)--(1.781,1.7909); 
\draw[thick,red](1.7818,1.79)--(1.7872,1.7873); 
\draw[thick,red](1.7872,1.7873)--(1.7872,1.7872); 
\end{scope}
\end{tikzpicture}
\caption{The functions $\mathcal{G}(q_0)$ (blue), $\mathcal{K}(q_0)$ (red), and the curves $(q_0{-}1)(q_1{-}1) = \frac{1}{q_0+1}, \frac{1}{q_1+1}, \frac{1}{2}, \frac{q_1}{q_1+1}, \frac{q_0}{q_0+1}, 1$.} \label{f:GL}
\end{figure}

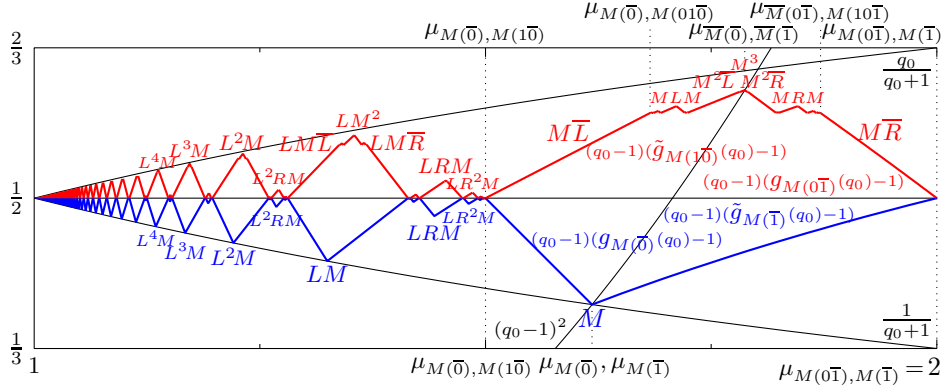
\begin{figure}[ht]
\begin{tikzpicture}[scale=12,join=bevel]
\draw(2,.6567)--(2,.6667)--(1,.6667)node[left]{$\frac{2}{3}$}--(1,.3333)node[left]{$\frac{1}{3}$}node[below]{\small 1}--(2,.3333)--(2,.3433);
\draw(1.25,.3383)--(1.25,.3333) (1.25,.6617)--(1.25,.6667) (1.5,.3433)--(1.5,.3333) (1.5,.6567)--(1.5,.6667) (1.75,.3383)--(1.75,.3333) (1.75,.6617)--(1.75,.6667) (1,.5)node[left]{$\frac{1}{2}$}--(2,.5);
\draw[domain=1:2] plot (\x,{1/(\x+1)});
\node[above] at (1.9667,.3333){$\frac{1}{q_0+1}$};
\draw[domain=1:1.7] plot (\x,{\x/(\x+1)});
\draw[domain=1.7:2] plot (\x,{\x/(\x+1)});
\node[below] at (1.9667,.6667){$\frac{q_0}{q_0+1}$};
\draw[domain=1.57735:1.8167] plot (\x,{(\x-1)^2});
\node[above] at (1.55,.3333){$\scriptstyle (q_0-1)^2$};

\draw[dotted](1.5,.5)--(1.5,.3333);
\node[below] at (1.485,.3333){\small$\mu_{M(\overline{0}),M(1\overline{0})}$};
\draw[dotted](1.618,.382)--(1.618,.3333);
\node[below] at (1.633,.3333){\small$\mu_{M(\overline{0})},\mu_{M(\overline{1})}$};
\draw[dotted](2,.5)--(2,.3333);
\node[below] at (1.915,.3333){{\small$\mu_{M(0\overline{1}),M(\overline{1})}\,{=}\,$}$2$};
\fill[blue](1,.5)--(1.0404,.4901)--(1.0393,.5);
\draw[thick,blue](1.0393,.5)--(1.0404,.4901)--(1.0416,.5); 
\draw[thick,blue](1.0417,.5)--(1.0429,.4895)--(1.0443,.5); 
\draw[thick,blue](1.0444,.5)--(1.0458,.4888)--(1.0473,.5); 
\draw[thick,blue](1.0474,.5)--(1.049,.488)--(1.0508,.5); 
\draw[thick,blue](1.0509,.5)--(1.0527,.4872)--(1.0548,.5); 
\draw[thick,blue](1.055,.5)--(1.0571,.4861)--(1.0595,.5); 
\draw[thick,blue](1.0597,.5)--(1.0622,.4849)--(1.065,.5); 
\draw[thick,blue](1.0654,.5)--(1.0683,.4835)--(1.0718,.5); 
\draw[thick,blue](1.0723,.5)--(1.0758,.4817)--(1.0801,.5); 
\draw[thick,blue](1.0807,.5)--(1.0851,.4796)--(1.0905,.5); 
\draw[thick,blue](1.0915,.5)--(1.097,.4769)--(1.1041,.5); 
\draw[thick,blue](1.1055,.5)--(1.1128,.4733)--(1.1225,.5); 
\draw[thick,blue](1.1227,.5)--(1.1235,.4972)--(1.1245,.5); 
\draw[thick,blue](1.1248,.5)--(1.1347,.4684)node[below=-.5ex]{\tiny$L^4\!M$}--(1.1487,.5); 
\draw[thick,blue](1.1491,.5)--(1.1505,.4961)--(1.1523,.5); 
\draw[thick,blue](1.1901,.5)--(1.1927,.4942)--(1.1962,.5); 
\draw[thick,blue](1.1528,.5)--(1.1673,.4614)node[below]{\scriptsize$L^3\!M$}--(1.1892,.5); 
\draw[thick,blue](1.1974,.5)--(1.2207,.4503)node[below=-.5ex]{\footnotesize$L^2\!M$}--(1.2599,.5); 
\draw[thick,blue](1.2604,.5)--(1.2613,.4984)--(1.2626,.5); 
\draw[thick,blue](1.2627,.5)--(1.2682,.4903)node[below=-.5ex]{\tiny$L^2\!RM$}--(1.2767,.5); 
\draw[thick,blue](1.2768,.5)--(1.2781,.4979)--(1.2799,.5); 
\draw[thick,blue](1.2808,.5)--(1.3247,.4302)node[below=-.5ex]{\small$LM$}--(1.4142,.5); 
\draw[thick,blue](1.4148,.5)--(1.4155,.4991)--(1.4169,.5); 
\draw[thick,blue](1.4169,.5)--(1.4202,.4958)--(1.4262,.5); 
\draw[thick,blue](1.427,.5)--(1.4433,.4801)node[below]{\footnotesize$LRM$}--(1.4743,.5); 
\draw[thick,blue](1.4756,.5)--(1.4813,.4937)node[below=-.5ex]{\tiny$LR^2\!M$}--(1.4917,.5); 
\draw[thick,blue](1.4918,.5)--(1.4938,.4979)--(1.4972,.5); 
\draw[thick,blue](1.4973,.5)--(1.4979,.4993)--(1.4991,.5); 
\draw[thick,blue](1.5,.5)--(1.618,.382)node[below=-.5ex]{$M$}; 
\node[blue,right] at (1.54,.45){\small${\scriptstyle(q_0-1)(}g_{M(\overline{0})}{\scriptstyle(q_0)-1)}$};
\draw[thick,blue,domain=1.61804:2] plot (\x,{1-1/\x}); 
\node[blue,left] at (1.92,.48){\small${\scriptstyle(q_0-1)(}\tilde{g}_{M(\overline{1})}{\scriptstyle(q_0)-1)}$};

\draw[dotted](1.5,.5)--(1.5,.6667)node[above=-.5ex]{\small$\mu_{M(\overline{0}),M(1\overline{0})}$};
\draw[dotted](1.6823,.5944)--(1.6823,.69)node[above=-.5ex]{\small$\mu_{M(\overline{0}),M(01\overline{0})}$};
\draw[dotted](1.7872,.6197)--(1.7872,.6667)node[above=-.5ex]{\small$\mu_{\overline{M}(\overline{0}),\overline{M}(\overline{1})}$};
\draw[dotted](1.8712,.594)--(1.8712,.69)node[above=-.5ex]{\small$\mu_{\overline{M}(0\overline{1}),M(10\overline{1})}$};
\draw[dotted](2,.5)--(2,.6667);
\node[above=-.5ex] at (1.94,.6667){\small$\mu_{M(0\overline{1}),M(\overline{1})}$};
\fill[red](1,.5)--(1.0393,.5)--(1.0405,.5095);
\draw[thick,red](1.0393,.5)--(1.0405,.5095) (1.0405,.5096)--(1.0416,.5); 
\draw[thick,red](1.0417,.5)--(1.043,.5101) (1.0431,.5101)--(1.0443,.5); 
\draw[thick,red](1.0444,.5)--(1.0458,.5107) (1.0459,.5107)--(1.0473,.5); 
\draw[thick,red](1.0474,.5)--(1.0491,.5114) (1.0492,.5115)--(1.0508,.5); 
\draw[thick,red](1.0509,.5)--(1.0528,.5123) (1.053,.5123)--(1.0548,.5); 
\draw[thick,red](1.055,.5)--(1.0572,.5132) (1.0574,.5132)--(1.0595,.5); 
\draw[thick,red](1.0597,.5)--(1.0623,.5143) (1.0626,.5143)--(1.065,.5); 
\draw[thick,red](1.0654,.5)--(1.0685,.5156) (1.0689,.5156)--(1.0718,.5); 
\draw[thick,red](1.0723,.5)--(1.076,.5171) (1.0765,.5172)--(1.0801,.5); 
\draw[thick,red](1.0807,.5)--(1.0854,.519) (1.0862,.5191)--(1.0905,.5); 
\draw[thick,red](1.0915,.5)--(1.0975,.5214) (1.0985,.5215)--(1.1041,.5); 
\draw[thick,red](1.1055,.5)--(1.1136,.5245) (1.1152,.5245)--(1.1225,.5); 
\draw[thick,red](1.1227,.5)--(1.1237,.5027) (1.1237,.5027)--(1.1245,.5); 
\draw[thick,red](1.1248,.5)--(1.136,.5286) (1.1386,.5287)--(1.1487,.5); 
\node[red,above=-.5ex] at (1.137,.529){\tiny$L^4\!M$};
\draw[thick,red](1.1363,.5286)--(1.1374,.5314) (1.1374,.5314)--(1.1384,.5287); 
\draw[thick,red](1.1491,.5)--(1.1508,.5038) (1.1509,.5038)--(1.1523,.5); 
\draw[thick,red](1.1701,.5344)--(1.1722,.5384) (1.1722,.5384)--(1.1739,.5345); 
\draw[thick,red](1.1528,.5)--(1.1696,.5344) (1.1744,.5346)--(1.1892,.5); 
\node[red,above] at (1.172,.535){\scriptsize$L^3\!M$};
\draw[thick,red](1.1901,.5)--(1.1935,.5056) (1.1937,.5056)--(1.1962,.5); 
\draw[thick,red](1.1974,.5)--(1.2255,.5433) (1.236,.5436)--(1.2599,.5); 
\node[red,above] at (1.23,.543){\footnotesize$L^2\!M$};
\draw[thick,red](1.2257,.5433)--(1.2264,.5443) (1.2264,.5443)--(1.2269,.5433); 
\draw[thick,red](1.2269,.5433)--(1.2312,.5494) (1.2313,.5494)--(1.2345,.5435); 
\draw[thick,red](1.2345,.5435)--(1.2352,.5445) (1.2352,.5445)--(1.2358,.5436); 
\draw[thick,red](1.2604,.5)--(1.2617,.5016) (1.2617,.5016)--(1.2626,.5); 
\draw[thick,red](1.2627,.5)--(1.2706,.5093) (1.2711,.5093)--(1.2767,.5); 
\node[red,above=-.25ex] at (1.271,.509){\tiny$L^2\!RM$};
\draw[thick,red](1.2768,.5)--(1.2786,.502) (1.2786,.502)--(1.2799,.5); 
\draw[thick,red](1.2808,.5)--(1.3377,.5588) (1.3683,.5593)--(1.4142,.5); 
\node[red,above=1ex] at (1.3,.529){\footnotesize$LM\overline{L}$};
\node[red,above=1ex] at (1.401,.53){\footnotesize$LM\overline{R}$};
\draw[thick,red](1.3379,.5588)--(1.3384,.5593) (1.3384,.5593)--(1.3387,.5588); 
\draw[thick,red](1.3387,.5588)--(1.3411,.561) (1.3411,.561)--(1.3426,.5589); 
\draw[thick,red](1.3428,.5589)--(1.3541,.5694) (1.355,.5694)--(1.3627,.5592); 
\node[red,above] at (1.359,.569){\scriptsize$LM^2$};
\draw[thick,red](1.3541,.5694)--(1.3546,.5698) (1.3546,.5698)--(1.3549,.5694); 
\draw[thick,red](1.3629,.5592)--(1.3654,.5615) (1.3654,.5615)--(1.3671,.5593); 
\draw[thick,red](1.3671,.5593)--(1.3677,.5598) (1.3677,.5598)--(1.368,.5593); 
\draw[thick,red](1.4148,.5)--(1.4161,.5009) (1.4161,.5009)--(1.4169,.5); 
\draw[thick,red](1.4169,.5)--(1.4228,.5041) (1.4229,.5041)--(1.4262,.5); 
\draw[thick,red](1.427,.5)--(1.4542,.5184) (1.4576,.5184)--(1.4743,.5); 
\node[red,above] at (1.456,.518){\footnotesize$LRM$};
\draw[thick,red](1.4544,.5184)--(1.4563,.5196) (1.4563,.5196)--(1.4574,.5184); 
\draw[thick,red](1.4756,.5)--(1.4854,.5061) (1.4858,.5061)--(1.4917,.5); 
\node[red,above=-.5ex] at (1.486,.506){\tiny$LR^2\!M$};
\draw[thick,red](1.4918,.5)--(1.4952,.502) (1.4952,.502)--(1.4972,.5); 
\draw[thick,red](1.4973,.5)--(1.4984,.5007) (1.4984,.5007)--(1.4991,.5); 
\draw[thick,red] (1.5,.5)--(1.6823,.5944); 
\node[red,above=.5ex] at (1.591,.547){\small$M\overline{L}$};
\node[red,right] at (1.6,.55){\small${\scriptstyle(q_0-1)(}\tilde{g}_{M(1\overline{0})}{\scriptstyle(q_0)-1)}$};
\draw[thick,red](1.6827,.5944)--(1.6832,.5947) (1.6832,.5947)--(1.6835,.5944); 
\draw[thick,red](1.6835,.5944)--(1.6851,.5952) (1.6851,.5952)--(1.686,.5944); 
\draw[thick,red](1.686,.5944)--(1.6912,.5967) (1.6913,.5967)--(1.694,.5945); 
\draw[thick,red](1.6942,.5945)--(1.711,.6018) (1.7118,.6018)--(1.7207,.5947); 
\node[red,above=-.5ex] at (1.711,.602){\tiny$MLM$};
\draw[thick,red](1.711,.6018)--(1.7115,.602) (1.7115,.602)--(1.7118,.6018); 
\draw[thick,red](1.7208,.5947)--(1.7225,.5954) (1.7226,.5954)--(1.7235,.5947); 
\draw[thick,red](1.7238,.5947)--(1.781,.6177) (1.7909,.6177)--(1.8217,.5947); 
\node[red,above=.5ex] at (1.752,.606){\scriptsize$M^2\!\overline{L}$};
\node[red,above=.5ex] at (1.806,.606){\scriptsize$M^2\!\overline{R}$};
\draw[thick,red](1.781,.6177)--(1.7816,.6179) (1.7816,.6179)--(1.7818,.6177); 
\draw[thick,red](1.7818,.6177)--(1.7872,.6197) (1.7873,.6197)--(1.79,.6177); 
\node[red,above=1ex] at (1.787,.619){\tiny$M^3$};
\draw[thick,red](1.7872,.6197)--(1.7872,.6197) (1.7872,.6197)--(1.7873,.6197); 
\draw[thick,red](1.79,.6177)--(1.7905,.6179) (1.7905,.6179)--(1.7908,.6177); 
\draw[thick,red](1.822,.5947)--(1.824,.5954) (1.824,.5954)--(1.825,.5947); 
\draw[thick,red](1.8251,.5947)--(1.8454,.6018) (1.8464,.6018)--(1.8563,.5945); 
\node[red,above=-.5ex] at (1.846,.602){\tiny$MRM$};
\draw[thick,red](1.8454,.6018)--(1.846,.602) (1.846,.602)--(1.8463,.6018); 
\draw[thick,red](1.8567,.5945)--(1.8632,.5967) (1.8633,.5967)--(1.8665,.5944); 
\draw[thick,red](1.8666,.5944)--(1.8687,.5952) (1.8687,.5952)--(1.8697,.5944); 
\draw[thick,red](1.8697,.5944)--(1.8704,.5947) (1.8704,.5947)--(1.8707,.5944); 
\draw[thick,red,domain=1.8712:2] plot (\x,{3-\x-1/\x}); 
\node[red,above=.5ex] at (1.936,.548){\small$M\overline{R}$};
\node[red,left] at (1.975,.515){\small${\scriptstyle(q_0-1)(}g_{M(0\overline{1})}{\scriptstyle(q_0)-1)}$};
\end{tikzpicture}
\caption{The maps $(q_0{-}1)(\mathcal{G}(q_0){-}1)$ (blue), $(q_0{-}1)(\mathcal{K}(q_0){-}1)$ (red).} \label{f:GL2}
\end{figure}

The functions $\mathcal{G}$ and $\mathcal{K}$ are drawn in Figure~\ref{f:GL}, the functions $(q_0{-}1)(\mathcal{G}(q_0){-}1)$  and $(q_0{-}1)(\mathcal{K}(q_0){-}1)$ are drawn in Figure ~\ref{f:GL2}, and the calculation of $\mathcal{G}(q_0)$ and $\mathcal{K}(q_0)$ is worked out in the following example for the case $\sigma = L^k M$, $k \ge 0$.
In particular, we have $\mathcal{G}(\frac{1+\sqrt{5}}{2}) = \frac{1+\sqrt{5}}{2}$, and $q_0 = \frac{1+\sqrt{5}}{2}$ is the unique value with $\mathcal{G}(q_0) = q_0$ since
$\mathcal{G}$ is strictly decreasing by Theorem~\ref{t:1}.
Since $\overline{M}(\overline{0})$ and $\overline{M}(\overline{1})$ are the Thue-Morse words, we have $g_{\overline{M}(\overline{0})}(q_{KL}) = q_{KL} = \tilde{g}_{\overline{M}(\overline{1})}(q_{KL})$ for the Komornik--Loreti constant~$q_{KL}$, hence $\mu_{\overline{M}(\overline{0}), \overline{M}(\overline{1})} = q_{KL} = \mathcal{K}(q_{KL})$, with no other~$q_0$ satisfying $\mathcal{K}(q_0) = q_0$.
Consequently, Theorem~\ref{t:2} can be seen as a generalization of the theorem of Glendinning and Sidorov~\cite{GleSid2001}.

\begin{example} \label{ex:Lk}
Let $\sigma = L^k M$, $k \ge 0$.
Then
\begin{equation*}
\sigma(0) = 010^k, \qquad \sigma(1) = 100^k.
\end{equation*}
In order to determine $g_{\bu}(q_0)$ for $\bu \in \{\sigma(\overline{0}), \sigma(0\overline{1}), \sigma(01\overline{0})\}$, we use that
\begin{equation*}
\begin{aligned}
& q_0\, \big(q_1\, \pi_{q_0,q_1}(01i_1\cdots i_r\,\overline{i_{r+1}\cdots i_{r+p}}) - 1\big) = 1 - q_0 + \pi_{q_0,q_1}(i_1\cdots i_r\,\overline{i_{r+1}\cdots i_{r+p}}) \\
& = 1 - q_0 + \sum_{n=1}^r \frac{i_n\,q_{i_{n+1}}\cdots q_{i_r}}{q_{i_1}\cdots q_{i_r}} + \sum_{n=1}^p \frac{i_{r+n}\,q_{i_{r+n+1}}\cdots q_{i_{r+p}}}{q_{i_1}\cdots q_{i_r}\, (q_{i_{r+1}}\cdots q_{i_p}-1)}.
\end{aligned}
\end{equation*}
Since $\sigma(\overline{0}) = 01\overline{0^k01}$, $\sigma(0\overline{1}) = 01\overline{0^k10}$, $\sigma(01\overline{0}) = 010^k10\overline{0^k01}$, we obtain that
\begin{equation*}
\begin{aligned}
g_{\sigma(\overline{0})}(q_0) = q_1 \ & \Longleftrightarrow \ (1-q_0) (q_0^{k+1} q_1-1) + 1 = 0, \\
g_{\sigma(0\overline{1})}(q_0) = q_1 \ & \Longleftrightarrow \ (1-q_0) (q_0^{k+1} q_1-1) + q_0 = 0, \\
 g_{\sigma(01\overline{0})}(q_0) = q_1 \ & \Longleftrightarrow \ \big((1-q_0) q_0^{k+1} q_1 + q_0\big) (q_0^{k+1} q_1-1) + 1 = 0, \\
\end{aligned}
\end{equation*}
i.e.,
\begin{equation*}
g_{\sigma(\overline{0})}(q_0) = \frac{1}{q_0^k(q_0{-}1)},\ g_{\sigma(0\overline{1})}(q_0)  = \frac{2q_0-1}{q_0^{k+1}(q_0{-}1)}, \
g_{\sigma(01\overline{0})}(q_0) = \frac{2q_0{-}1{+}\sqrt{4q_0{-}3}}{2\,q_0^{k+1}(q_0-1)}.
\end{equation*}
For $\tilde{g}_{\bv}(q_0)$, $\bv \in \{\sigma(\overline{1}), \sigma(1\overline{0}), \sigma(10\overline{1})\}$, we use similarly that
\begin{equation*}
\begin{aligned}
& q_1\, \big(q_0\, \tilde{\pi}_{q_0,q_1}(10i_1\cdots i_r\,\overline{i_{r+1}\cdots i_{r+p}}) - 1\big) = 1 - q_1 + \tilde{\pi}_{q_0,q_1}(i_1\cdots i_r\,\overline{i_{r+1}\cdots i_{r+p}}) \\
& = 1 - q_1 + \sum_{n=1}^r \frac{(1-i_n)\,q_{i_{n+1}}\cdots q_{i_r}}{q_{i_1}\cdots q_{i_r}} + \sum_{n=1}^p \frac{(1-i_{r+n})\,q_{i_{r+n+1}}\cdots q_{i_{r+p}}}{q_{i_1}\cdots q_{i_r}\, (q_{i_{r+1}}\cdots q_{i_p}-1)}.
\end{aligned}
\end{equation*}
We have $\sigma(\overline{1}) = 10\overline{0^k10}$, $\sigma(1\overline{0}) = 10\overline{0^k01}$, $\sigma(10\overline{1}) = 100^k01\overline{0^k10}$, thus
\begin{equation*}
\begin{aligned}
\tilde{g}_{\sigma(\overline{1})}(q_0) = q_1 & \Longleftrightarrow \ (1-q_1) (q_0^{k+1} q_1-1) + q_0 q_1 \frac{q_0^k-1}{q_0-1}+ 1 = 0, \\
\tilde{g}_{\sigma(1\overline{0})}(q_0) = q_1 & \Longleftrightarrow \ (1-q_1) (q_0^{k+1} q_1-1) + q_1 \frac{q_0^{k+1}-1}{q_0-1} = 0, \\
\tilde{g}_{\sigma(10\overline{1})}(q_0) = q_1 & \Longleftrightarrow \Big((1{-}q_1) q_0^{k+1} q_1 + q_1 \frac{q_0^{k+1}{-}1}{q_0{-}1}\Big) (q_0^{k+1} q_1{-}1) + q_0 q_1 \frac{q_0^k{-}1}{q_0{-}1} + 1= 0.
\end{aligned}
\end{equation*}
This gives that
\begin{equation*}
\begin{aligned}
\tilde{g}_{\sigma(\overline{1})}(q_0) & = \frac{q_0^{k+2}-1}{q_0^{k+1}(q_0-1)}, \\
\tilde{g}_{\sigma(1\overline{0})}(q_0) & = \frac{1}{2q_0^{k+1}}\, \Bigg(\frac{q_0^{k+2}-1}{q_0-1}+1+\sqrt{\bigg(\frac{q_0^{k+3}-1}{q_0-1}-q_0+3\bigg)\frac{q_0^{k+1}-1}{q_0-1}}\Bigg),
\end{aligned}
\end{equation*}
and a formula for $\tilde{g}_{\sigma(10\overline{1})}(q_0)$ with cubic roots that we do not need for the calculation of $\mu_{\sigma(0\overline{1}),\sigma(10\overline{1})}$.
Evaluating equations of the form $g_{\bu}(q_0) = \tilde{g}_{\bv}(q_0)$, we obtain that
\begin{gather*}
\mu_{\sigma(\overline{0}),\sigma(\overline{1})}^{k+2} =  \mu_{\sigma(\overline{0}),\sigma(\overline{1})} + 1, \qquad 2\,\mu_{\sigma(\overline{0}),\sigma(1\overline{0})}^{k+1} = \mu_{\sigma(\overline{0}),\sigma(1\overline{0})}^k + 2, \qquad \mu_{\sigma(0\overline{1}),\sigma(\overline{1})}^{k+1} = 2, \\
\mu_{\sigma(01\overline{0}),\sigma(1\overline{0})}^{2k+3} = \mu_{\sigma(01\overline{0}),\sigma(1\overline{0})}^{2k+2} + 2\, \mu_{\sigma(01\overline{0}),\sigma(1\overline{0})}^{k+2} - 3\, \mu_{\sigma(01\overline{0}),\sigma(1\overline{0})}^{k+1} - \mu_{\sigma(01\overline{0}),\sigma(1\overline{0})} + 3, \\
3\, \mu_{\sigma(0\overline{1}),\sigma(10\overline{1})}^{k+3} = 2\, \mu_{\sigma(0\overline{1}),\sigma(10\overline{1})}^{k+2} + 6\, \mu_{\sigma(0\overline{1}),\sigma(10\overline{1})}^2 - 5\, \mu_{\sigma(0\overline{1}),\sigma(10\overline{1})} + 1.
\end{gather*}
In particular, for $k=0$, Theorem~\ref{t:2} gives that
\begin{align*}
\mathcal{G}(q_0) & =
\begin{cases}\frac{1}{q_0-1} & \text{if}\ q_0 \in [\frac{3}{2},\frac{1+\sqrt{5}}{2}], \\[.5ex]
\frac{q_0+1}{q_0} & \text{if}\ q_0 \in [\frac{1+\sqrt{5}}{2},2],\end{cases} \\ \mathcal{K}(q_0) & =
\begin{cases}\frac{q_0+2+\sqrt{q_0^2+4}}{2q_0}  & \text{if}\ q_0 \in [\frac{3}{2},1.6823278], \\[.5ex]
\frac{2q_0-1}{q_0(q_0-1)} & \text{if}\ q_0 \in [1.8711568,2].\end{cases}
\end{align*}
\end{example}

\smallskip
An essential ingredient of the proofs of Theorems~\ref{t:1}--\ref{t:3} is the characterization of pairs $\ba\in 0\{0,1\}^\infty$, $\bb \in 1\{0,1\}^\infty$ for which the lexicographically defined subshift
\begin{equation*}
\Omega_{\ba,\bb} := \big\{i_1i_2 \cdots \in \{0,1\}^\infty \,:\, i_n i_{n+1}\cdots \le \ba\ \text{or}\ i_n i_{n+1}\cdots \ge \bb \ \text{for all} \ n \ge 1\big\},
\end{equation*}
i.e., the union of shift-orbits avoiding the open interval $(\ba,\bb)$, is trivial, countable or uncountable (with zero or positive entropy).

\begin{theorem} \label{t:lex}
Let $\ba \in 0\{0,1\}^\infty$, $\bb \in  1\{0,1\}^\infty$. Then we have the following.
\begin{enumerate}[\upshape (i)]
\itemsep.5ex
\item \label{i:lex1}
$\Omega_{\ba,\bb} \ne \{\overline{0}, \overline{1}\}$ if and only if $\ba \ge \sigma(\overline{0})$, $\bb \le \sigma(\overline{1})$ for some $\sigma \in \{L,R\}^*M$, or $\ba = \bsigma(\overline{0})$, $\bb = \bsigma(\overline{1})$ for some primitive $\bsigma \in \{L,R\}^\infty$, or $\ba = 0\overline{1}$, or $\bb = 1\overline{0}$.
\item \label{i:lex2}
$\Omega_{\ba,\bb} = \{\overline{0}, \overline{1}\}$ if and only if $\ba < \sigma(\overline{0})$, $\bb \ge \sigma(1\overline{0})$, or $\ba \le \sigma(0\overline{1})$, $\bb > \sigma(\overline{1})$ for some $\sigma \in \{L,R\}^*M$.
\item \label{i:lex3}
$\Omega_{\ba,\bb}$ is uncountable with positive entropy if and only if $\ba \ge \sigma(\overline{0})$, $\bb < \sigma(1\overline{0})$, or $\ba > \sigma(0\overline{1})$, $\bb \le \sigma(\overline{1})$ for some $\sigma \in \{L,M,R\}^* M$. 
\item \label{i:lex4}
$\Omega_{\ba,\bb}$ is uncountable with zero entropy if and only if $\ba = \bsigma(\overline{0})$, $\bb = \bsigma(\overline{1})$ for some primitive $\bsigma \in \{L,M,R\}^\infty$.
\item \label{i:lex5}
$\Omega_{\ba,\bb}$ is countable if and only if $\ba \,{\le}\, \sigma(01\overline{0})$, $\bb \,{\ge}\, \sigma(1\overline{0})$, or $\ba \,{\le}\, \sigma(0\overline{1})$, $\bb \,{\ge}\, \sigma(10\overline{1})$ for some $\sigma \in \{L,M,R\}^*$. 
\end{enumerate}
\end{theorem}

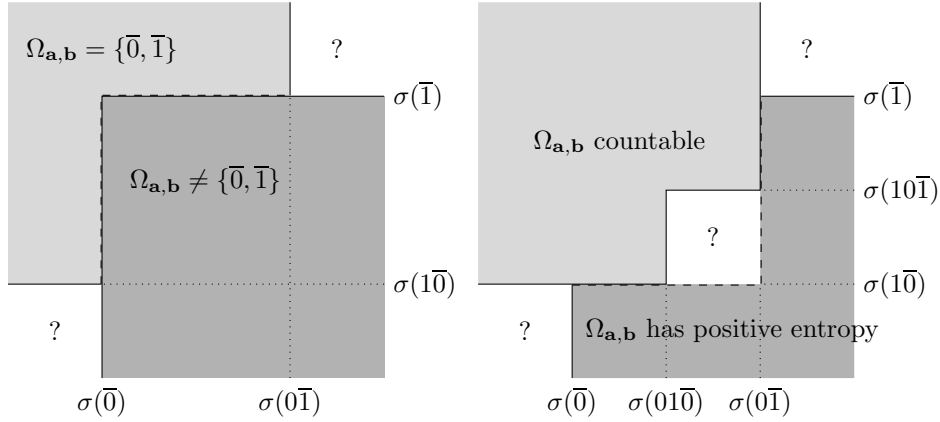
\begin{figure}[ht]
\centerline{\begin{tikzpicture}[scale=1.25]
\fill[black!15](0,1)--(1,1)--(1,3)--(3,3)--(3,4)--(0,4);
\fill[black!30](1,0)--(4,0)--(4,3)--(1,3);
\draw(0,1)--(.99,1) (3,3.01)--(3,4);
\draw[dashed](.99,1)--(.99,3.01)--(3,3.01);
\draw(1,0)node[below]{$\sigma(\overline{0})$}--(1,3)--(4,3)node[right]{$\sigma(\overline{1})$};
\draw[dotted](3,0)node[below]{$\sigma(0\overline{1})$}--(3,3);
\draw[dotted](4,1)node[right]{$\sigma(1\overline{0})$}--(1,1);
\node at (1,3.5){$\Omega_{\ba,\bb} = \{\overline{0}, \overline{1}\}$};
\node at (2.1,2.1){$\Omega_{\ba,\bb} \ne \{\overline{0}, \overline{1}\}$};
\node at (.5,.5){?};
\node at (3.5,3.5){?};
\begin{scope}[shift={(5,0)}]
\fill[black!15](0,1)--(2,1)--(2,2)--(3,2)--(3,4)--(0,4);
\fill[black!30](1,0)--(4,0)--(4,3)--(3,3)--(3,1)--(1,1);
\draw(0,1)--(2,1)--(2,2)--(3,2)--(3,4);
\draw(1,0)node[below]{$\sigma(\overline{0})$}--(1,.99) (3.01,3)--(4,3)node[right]{$\sigma(\overline{1})$};
\draw[dashed](1,.99)--(3.01,.99)--(3.01,3);
\draw[dotted](2,0)node[below]{$\sigma(01\overline{0})$}--(2,1);
\draw[dotted](3,0)node[below]{$\sigma(0\overline{1})$}--(3,1);
\draw[dotted](4,1)node[right]{$\sigma(1\overline{0})$}--(3,1);
\draw[dotted](4,2)node[right]{$\sigma(10\overline{1})$}--(3,2);
\node at (2.7,.5){$\Omega_{\ba,\bb}$ has positive entropy};
\node at (1.5,2.5){$\Omega_{\ba,\bb}$ countable};
\node at (.5,.5){?};
\node at (2.5,1.5){?};
\node at (3.5,3.5){?};
\end{scope}
\end{tikzpicture}}
\caption{The cardinality of $\Omega_{\ba,\bb}$ according to Theorem~\ref{t:lex}, for $\sigma = \tilde{\sigma} M$ with $\tilde{\sigma} \in \{L,R\}^*$ (left) and $\tilde{\sigma} \in \{L,M,R\}^*$ (right). 
In the regions with question marks, we have to consider substitutions starting with $\tilde{\sigma} L$ (in the lower left corners), with $\tilde{\sigma} R$ (in the upper right corners) and with $\tilde{\sigma} M$ (in the middle of the right picture).}
\end{figure}

This also improves and simplifies results of Labarca and Moreira~\cite{LabMor2006} for subshifts of the form
\begin{equation*}
\Sigma_{\ba,\bb} := \big\{i_1i_2 \cdots \in \{0,1\}^\infty \,:\, \ba \le i_n i_{n+1}\cdots \le \bb \ \text{for all}\ n \ge 1\big\},
\end{equation*}
i.e., the union of shift-orbits staying in the closed interval $[\ba,\bb]$.

\begin{theorem} \label{t:LM}
We have $\Omega_{0\bb,1\ba} = 0^* \Sigma_{\ba,\bb} \cup 1^* \Sigma_{\ba,\bb} \cup \{\overline{0}, \overline{1}\}$ and, hence, $\Sigma_{\ba,\bb} = \emptyset$ if and only if $\Omega_{0\bb,1\ba} = \{\overline{0}, \overline{1}\}$, $\Sigma_{\ba,\bb}$ is countable if and only if $\Omega_{0\bb,1\ba}$ is countable, $h(\Sigma_{\ba,\bb}) = h(\Omega_{0\bb,1\ba})$.
\end{theorem}

Here, $c^* = \{c^k : k \ge 0\}$ is the set of words containing only the letter $c \in \{0,1\}$.
Note that Theorem~\ref{t:lex} gives a semi-algorithm for deciding triviality, countability and zero entropy of~$\Omega_{\ba,\bb}$, which terminates for all $\ba \in 0\{0,1\}^\infty$, $\bb \in 1\{0,1\}^\infty$ except for $\ba = \bsigma(\overline{0})$, $\bb = \bsigma(\overline{1})$ with primitive $\bsigma \in \{L,R\}^\infty$ resp.\ $\bsigma \in \{L,M,R\}^\infty$.
Glendinning and Sidorov~\cite{GleSid2015} have similar results to Theorem~\ref{t:lex}; their set of substitutions, which are defined by the lexicographically smallest and largest cyclically balanced words with rational frequencies, is equal to $\{L,R\}^*M$. 

\smallskip
The rest of the paper is organised as follows.
In Sections~\ref{s:lex} and~\ref{s:g}, we study some relevant properties of the set~$U_{q_0, q_1}$, the functions~$g$ and~$\tilde{g}$, and prove Theorems~\ref{t:lex} and~\ref{t:LM}.
In Section~\ref{sec:proof-main-results}, we prove Theorems~\ref{t:1}, \ref{t:2} and~\ref{t:3}.
We end the paper by raising some open problems.

\section{Lexicographic world} \label{s:lex}
In this section, we first show that an alphabet-base system $\{(d_0,q_0),(d_1,q_1)\}$ with $d_1(q_0-1) \ne d_0(q_1-1)$ is isomorphic to $\{(0,q_0),(1,q_1)\}$.
Therefore, we call the expansions in the system $\{(0,q_0),(1,q_1)\}$ simply \emph{$(q_0,q_1)$-expansions}.
For this result, we do not require $q_0, q_1$ to be real numbers, but only that $|q_0|, |q_1| > 1$.
Set
\begin{equation*}
T_{d,q}:\, \mathbb{C} \to \mathbb{C}, \qquad x \mapsto q x - d \qquad (d,q \in \mathbb{C}).
\end{equation*}

\begin{lemma} \label{l:01}
Let $d_0,d_1,q_0,q_1 \in \mathbb{C}$ with $|q_0|,|q_1| > 1$, $i_1 i_2 \cdots \in \{0,1\}^\infty$.
Then
\begin{equation} \label{e:di}
\sum_{k=1}^\infty \frac{d_{i_k}}{q_{i_1}q_{i_2}\cdots q_{i_k}} = \frac{d_0}{q_0-1} + \bigg(d_1-d_0\frac{q_1-1}{q_0-1}\bigg)\, \pi_{q_0,q_1}(i_1i_2\cdots).
\end{equation}
In particular, we have
\begin{equation} \label{e:q1q0}
(q_1-1)\, \pi_{q_0,q_1}(i_1i_2\cdots)  + (q_0-1)\, \pi_{q_1,q_0}\big((1{-}i_1)(1{-}i_2)\cdots) = 1,
\end{equation}
hence $i_1i_2\cdots U_{q_0,q_1}$ if and only if $(1{-}i_1)(1{-}i_2)\cdots \in U_{q_1,q_0}$.
\end{lemma}

\begin{proof}
Let
\begin{equation*}
\varphi:\, \mathbb{C} \to \mathbb{C}, \quad x \mapsto \frac{d_0}{q_0-1} + \bigg(d_1-d_0\frac{q_1-1}{q_0-1}\bigg)\, x.
\end{equation*}
Then
\begin{equation*}
\begin{aligned}
T_{d_0,q_0} \circ \varphi(x) & = q_0\, \bigg(\frac{d_0}{q_0-1} + \bigg(d_1-d_0\frac{q_1-1}{q_0-1}\bigg)\, x\bigg) - d_0 \\
& = \frac{d_0}{q_0-1} + \bigg(d_1-d_0\frac{q_1-1}{q_0-1}\bigg)\, q_0x = \varphi \circ T_{0,q_0}(x), \\
T_{d_1,q_1} \circ \varphi(x) & = q_1\, \bigg(\bigg(d_1-d_0\frac{q_1-1}{q_0-1}\bigg)\, x + \frac{d_0}{q_0-1}\bigg) - d_1 \\
& = \bigg(d_1-d_0\frac{q_1-1}{q_0-1}\bigg)\, (q_1 x - 1) + \frac{d_0}{q_0-1} = \varphi \circ T_{1,q_1}(x).
\end{aligned}
\end{equation*}
For $n \ge 0$, let $x_n := \pi_{q_0,q_1}(i_{n+1}i_{n+2} \cdots)$.
Then
\begin{equation*}
\begin{aligned}
\varphi(x_0) & = T_{d_{i_1},q_{i_1}}^{-1} \circ \cdots \circ T_{d_{i_n},q_{i_n}}^{-1} \circ \varphi \circ T_{i_n,q_{i_n}} \\
& =  T_{d_{i_1},q_{i_1}}^{-1} \circ \cdots \circ T_{d_{i_n},q_{i_n}}^{-1} \circ \varphi(x_n) \\
& = \sum_{k=1}^n \frac{d_{i_k}}{q_{i_1}q_{i_2}\cdots q_{i_k}} + \frac{\varphi(x_n)}{q_{i_1}q_{i_2}\cdots q_{i_n}}
\end{aligned}
\end{equation*}
for all $n \ge 1$.
Since $|q_0|,|q_1|>1$ and $\varphi(x_n)$ is bounded, this implies that $\varphi(x_0) = \sum_{k=1}^\infty d_{i_k}/(q_{i_1}q_{i_2}\cdots q_{i_k})$, which proves \eqref{e:di}.

Using \eqref{e:di} with $d_0 = 1$, $d_1 = 0$, we obtain that
\begin{equation*}
\pi_{q_1,q_0}\big((1{-}i_1)(1{-}i_2)\cdots\big) = \sum_{k=1}^\infty \frac{1-i_k}{q_{i_1}\cdots q_{i_k}} = \frac{1}{q_0-1} - \frac{q_1-1}{q_0-1}\, \pi_{q_0,q_1}(i_1i_2\cdots),
\end{equation*}
i.e., \eqref{e:q1q0} holds.
\end{proof}

Now, we return to real bases $q_0,q_1>1$.
The action of $T_{0,q_0}$ and $T_{1,q_1}$ on the interval $[0,\frac{1}{q_1-1}]$ is depicted in Figure~\ref{f:T}.
Each number in this interval has a $(q_0,q_1)$-expansion, i.e., $\{(0,q_0),(1,q_1)\}$ is regular, if and only if $\frac{1}{q_1} \le \frac{1}{q_0(q_1-1)}$, which is equivalent to $q_0 + q_1 \ge q_0 q_1$.

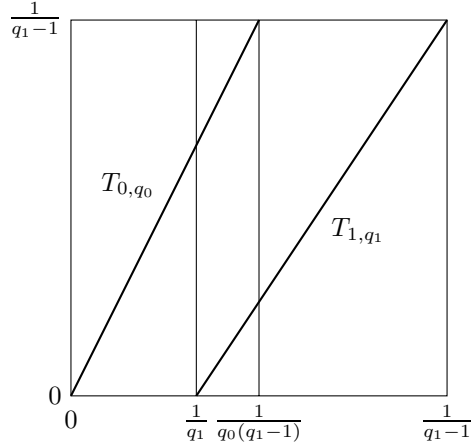
\begin{figure}[ht]
\begin{tikzpicture}[scale=2.5]
\draw(0,0)node[left]{$0$}node[below]{$\vphantom{\frac{1}{q_1-1}}0$}--(2,0)node[below]{$\frac{1}{q_1-1}$}--(2,2)--(0,2)node[left]{$\frac{1}{q_1-1}$}--cycle (.6667,0)node[below]{$\frac{1}{q_1}$}--(.6667,2) (1,0)node[below]{$\frac{1}{q_0(q_1-1)}$}--(1,2);
\draw[thick](0,0)--node[above left]{$T_{0,q_0}$}(1,2) (.6667,0)--node[below right]{$T_{1,q_1}$}(2,2);
\end{tikzpicture}
\caption{The maps $T_{0,q_0}$ and $T_{1,q_1}$; here, $q_0=2$ and $q_1=3/2$.} \label{f:T}
\end{figure}

We have $i_1 i_2 \cdots \in U_{q_0,q_1}$ if and only if
\begin{equation} \label{e:Uint}
\pi_{q_0,q_1}(i_{n+1}i_{n+2}\cdots) \not\in [\tfrac{1}{q_1}, \tfrac{1}{q_0(q_1-1)}] \quad \text{for all}\ n \ge 0.
\end{equation}
In particular, we have $U_{q_0,q_1} = \{0,1\}^\infty$ if $q_0 + q_1 < q_0 q_1$.
Therefore, we assume that $q_0 + q_1 \ge q_0 q_1$ in the following.

Denote the \emph{quasi-greedy} $(q_0,q_1)$-expansion of~$\frac{1}{q_1}$ by~$\ba_{q_0,q_1}$ and the \emph{quasi-lazy} $(q_0,q_1)$-expansion of $\frac{1}{q_0(q_1-1)}$ by~$\bb_{q_0,q_1}$.
(An expansion of a number is quasi-greedy if it is the lexicographically largest expansion not ending with~$\overline{0}$; it is quasi-lazy if it is the lexicographically smallest expansion not ending with~$\overline{1}$.)
Since $\frac{1}{q_1} = \pi_{q_0,q_1}(1\overline{0})$ and $\frac{1}{q_0(q_1-1)} = \pi_{q_0,q_1}(0\overline{1})$, $\ba_{q_0,q_1}$ starts with~$01$ and $\bb_{q_0,q_1}$ starts with~$10$.
We recall the following properties of quasi-greedy and quasi-lazy expansions from~\cite{KomLuZou2022}.

\begin{lemma} \label{l:qgql}
For $0 \le x < y \le \frac{1}{q_1-1}$, the quasi-greedy $(q_0,q_1)$-expansion of $x$ is lexicographically smaller than that of~$y$ and the quasi-lazy $(q_0,q_1)$-expansion of $x$ is lexicographically smaller than that of~$y$.
\end{lemma}

From \eqref{e:Uint} and Lemma~\ref{l:qgql} (see also \cite[Corollary~1.3]{KomLuZou2022}), we get that
\begin{equation*}
U_{q_0,q_1} = \big\{i_1i_2 \cdots \in \{0,1\}^\infty \,:\, i_n i_{n+1}\cdots < \ba_{q_0,q_1}\ \text{or}\ i_n i_{n+1}\cdots > \bb_{q_0,q_1} \ \forall n \ge 1\big\}.
\end{equation*}
The following lemmas show that we can consider the closed set
\begin{equation*}
V_{q_0,q_1} := \big\{i_1i_2 \cdots \in \{0,1\}^\infty \,:\, i_n i_{n+1}\cdots \le \ba_{q_0,q_1}\ \text{or}\ i_n i_{n+1}\cdots \ge \bb_{q_0,q_1} \ \forall n \ge 1\big\}
\end{equation*}
instead of $U_{q_0,q_1}$.

\begin{lemma} \label{l:ab}
For $q'_0 \ge q_0$, $q'_1 \ge q_1$, with $q'_0+q'_1 \ge q'_0 q'_1 > q_0 q_1$, we have
\begin{equation*}
\ba_{q_0,q_1} < \ba_{q'_0,q'_1}  \quad \text{and} \quad \bb_{q'_0,q'_1} < \bb_{q_0,q_1},
\end{equation*}
thus $U_{q_0,q_1} \subseteq V_{q_0,q_1}  \subseteq U_{q'_0,q'_1} \subseteq V_{q'_0,q'_1}$.
\end{lemma}

\begin{proof}
Since
$q'_1\, \pi_{q'_0,q'_1}(\ba_{q_0,q_1}) < q_1\, \pi_{q_0,q_1}(\ba_{q_0,q_1}) = 1 = q'_1\, \pi_{q'_0,q'_1}(\ba_{q'_0,q'_1})$,
the quasi-greedy $(q'_0,q'_1)$-expansion of $\pi_{q'_0,q'_1}(\ba_{q_0,q_1})$ is smaller than $\ba_{q'_0,q'_1}$ by Lemma~\ref{l:qgql}, hence \mbox{$\ba_{q_0,q_1} < \ba_{q'_0,q'_1}$}.
Symetrically, we have $q'_0\, \tilde{\pi}_{q'_0,q'_1}(\bb_{q_0, q_1}) < q_0\, \tilde{\pi}_{q_0,q_1}(\bb_{q_0,q_1}) = 1 = q'_0\, \tilde{\pi}_{q'_0,q'_1}(\bb_{q'_0, q'_1})$, hence by \eqref{e:q1q0} the quasi-lazy expansion of $\pi_{q'_0,q'_1}(\bb_{q_0,q_1})$ is larger than $\bb_{q'_0,q'_1}$, thus $\bb_{q_0,q_1} > \bb_{q'_0,q'_1}$.
\end{proof}

\begin{lemma} \label{l:UV} \mbox{}

\begin{enumerate}[\upshape (i)]
\itemsep.5ex
\item \label{i:UU1}
$U_{q_0, q_1}$ is infinite if and only if $U_{q_0, q_1} \ne \{\overline{0},\overline{1}\}$.
\item \label{i:UU2}
$U_{q_0, q_1}$ is uncountable if and only if $V_{q_0, q_1}$ is uncountable.
\item \label{i:UU3}
$U_{q_0, q_1}$ and $V_{q_0, q_1}$ have the same entropy.
\end{enumerate}
\end{lemma}

\begin{proof}
(\ref{i:UU1}) If $\bu \in U_{q_0, q_1}$ starts with~$0$, then we have $0^k\, \bu \in U_{q_0, q_1}$ for all $k \ge 0$.
If moreover $\bu \ne \overline{0}$, then all sequences $0^k\, \bu$ are different.
Therefore, $U_{q_0, q_1} \ne \{\overline{0},\overline{1}\}$ implies that $U_{q_0, q_1}$ is infinite.
The converse is trivial.

\noindent
(\ref{i:UU2}) We have $U_{q_0, q_1} \subseteq V_{q_0, q_1}$ and all elements of $V_{q_0, q_1} \setminus U_{q_0, q_1}$ end with $\ba_{q_0,q_1}$ or~$\bb_{q_0,q_1}$, hence this difference is countable.

\noindent
(\ref{i:UU3}) The set $V_{q_0,q_1}$ differs from $U_{q_0,q_1}$ only by a countable set, which does not affect the entropy according to Bowen's definition~\cite{Bow1973}, i.e., $h(U_{q_0,q_1}) = h(V_{q_0,q_1})$.
See \cite[Proposition~2.6]{AllKon2019} for a similar statement, with a more complicated proof.
\end{proof}

Since $V_{q_0,q_1} = \Omega_{\ba_{q_0,q_1},\bb_{q_0,q_1}}$, its cardinality is given by Theorem~\ref{t:lex}, which we prove in the rest of the section.
For $\ba \in 0\{0,1\}^\infty$, $\bb \in 1\{0,1\}^\infty$, we can write
\begin{equation*}
\Omega_{\ba,\bb} = \big\{\bu \in \{0,1\}^\infty \setminus \{\overline{0}, \overline{1}\} \,:\, \supn(\bu) \le \ba,\, \infn(\bu) \ge \bb\big\} \cup \{\overline{0}, \overline{1}\} ,
\end{equation*}
with
\begin{equation*}
\begin{aligned}
\supn(i_1i_2\cdots) & := \sup\{i_ni_{n+1} \cdots \,:\, n \ge 1,\, i_n = 0\}, \\
\infn(i_1i_2\cdots) & := \inf\{i_ni_{n+1} \cdots \,:\, n \ge 1,\, i_n = 1\}.
\end{aligned}
\end{equation*}
In other words, $\supn(\bu)$ is the maximal element of $X_{\bu}$ starting with~$0$ and $\infn(\bu)$ is the minimal element of $X_{\bu}$ starting with~$1$, where
\begin{equation*}
X_{i_1i_2\cdots} := \mathrm{closure}(\{i_n i_{n+1} \cdots : n \ge 1\})
\end{equation*}
is the \emph{subshift generated by the word} $i_1i_2\cdots \in \{0,1\}^\infty$.
We have
\begin{equation} \label{e:infsup}
\begin{aligned}
\supn(\sigma(\overline{0})) & = \sigma(\overline{0}), & \supn(\sigma(0\overline{1})) & = \sigma(0\overline{1}), & \supn(\bsigma(\overline{0})) & = \bsigma(\overline{0}), \\
\infn(\sigma(\overline{1})) & = \sigma(\overline{1}), & \infn(\sigma(1\overline{0})) & = \sigma(1\overline{0}), & \infn(\bsigma(\overline{1})) & = \bsigma(\overline{1}),
\end{aligned}
\end{equation}
for all $\sigma \in \{L,M,R\}^*$, $\bsigma \in \{L,M,R\}^\infty$; more precisely, the following lemma holds.

\begin{lemma} \label{l:infsup}
For all $\sigma \in \{L,M,R\}^*$, $\bsigma \in \{L,M,R\}^\infty$, $\bu \in \{0,1\}^\infty$, we have
\begin{equation*}
\begin{aligned}
\supn(\sigma(0\bu)) & = \sigma(\supn(0\bu)), & \infn(\sigma(1\bu)) & = \sigma(\infn(1\bu)), \\
\supn(\bsigma(0\bu)) & = \bsigma(\supn(0\bu)), & \infn(\bsigma(1\bu)) & = \bsigma(\infn(1\bu)).
\end{aligned}
\end{equation*}
\end{lemma}

\begin{proof}
We first show that $\sigma \in \{L,M,R\}^*$ and $\supn$ commute on $0\{0,1\}^\infty$.
It suffices to consider $\sigma \in \{L,M,R\}$.
Since the only occurrence of $0$ in $R(0)$ and $R(1)$ is at the beginning of $R(0)$, we have $\supn(R(0\bu)) = R(0\bv)$ for some $\bv \in X_{0\bu}$, and $0\bv = \supn(\bu)$ because $R$ is order-preserving.
For $\sigma \in \{L,M\}$, the letter $0$ also occurs at the end of~$\sigma(1)$.
Suppose that $\supn(\sigma(\bu)) = 0\sigma(\bv)$ with $1\bv \in X_{0\bu}$, and let $k \ge 1$ be such that $01^k\bv \in X_{0\bu}$.
Since $\sigma(01^k\bv) \ge 0\sigma(\bv)$, we have $\sigma(01^k\bv) = \supn(\sigma(\bu))$, thus $01^k\bv = \supn(\bu)$, i.e., $L$ and $M$ also commute with $\supn$ on $0\{0,1\}^\infty$.

The proof that $\sigma \in \{L,M,R\}^*$ and $\infn$ commute on $1\{0,1\}^\infty$ is symmetric; it suffices to exchange $0$ and $1$, $L$ and $R$, $\supn$ and $\infn$.

For $\bsigma = (\sigma_n)_{n\ge1} \in \{L,M,R\}^\infty$, $\bu \in \{0,1\}^\infty$, we have already proved that
\begin{equation*}
\supn(\bsigma(0\bu)) = \supn(\sigma_{[1,n)}(\sigma_{[n,\infty)}(0\bu))) = \sigma_{[1,n)}(\supn(\sigma_{[n,\infty)}(0\bu)))
\end{equation*}
for all $n \ge 1$.
If $\bsigma$ is primitive or ends with~$\overline{R}$, then the length of $\sigma_{[1,n)}(0)$ is unbounded, thus $\supn(\bsigma(0\bu)) = \bsigma(\overline{0}) = \bsigma(\supn(0\bu))$.
Otherwise, $\bsigma$ ends with~$\overline{L}$.
Since $\supn(\overline{L}(\overline{0})) = \overline{0} = \overline{L}(\supn(\overline{0}))$ and $\supn(\overline{L}(0\bu)) = 01\overline{0} = \overline{L}(\supn(0\bu))$ for $\bu \ne \overline{0}$, $\overline{L}$ commutes with $\supn$ on $0\{0,1\}^\infty$, hence $\bsigma$ commutes with $\supn$ on $0\{0,1\}^\infty$ for all $\bsigma \in \{L,M,R\}^\infty$.
Again, the case of $\infn$ and $\bu \in 1\{0,1\}^\infty$ is symmetric.
\end{proof}

We have the following relation between $\Omega_{\sigma(\ba),\sigma(\bb)}$ and $\sigma(\Omega_{\ba,\bb})$.

\begin{lemma} \label{l:Omegasigma}
If $\bu \in \Omega_{\sigma(\ba),\sigma(\bb)} \setminus \{\overline{0}, \overline{1}\}$, $\sigma \in \{L,M,R\}$, $\ba \in 0\{0,1\}^\infty$, $\bb \in 1\{0,1\}^\infty$, then $\bu = w\, \sigma(\bv)$ for some $\bv \in \Omega_{\ba,\bb}$, $w \in \{0,1\}^*$.

If $\bu \in \Omega_{\sigma(\ba),\sigma(\bb)} \setminus \{\overline{0}, \overline{1}\}$, $\sigma \,{\in}\, \{L,R\}^*$, $\ba, \bb \,{\in}\, \{0,1\}^\infty$ with $\ba < 0\overline{1}$ and $\bb > 1 \overline{0}$, then $\bu = w\, \sigma(\bv)$ for some $\bv \in \Omega_{\ba,\bb} \setminus \{\overline{0}, \overline{1}\}$, $w \in \{0,1\}^*$.
\end{lemma}

\begin{proof}
If $\sigma = L$, then $\supn(\bu) \le \overline{01}$, thus $\bu = 1^k L(\bv')$ for some $k \ge 0$, $\bv' \in \{0,1\}^\infty$.
If $\sigma = R$, then $\infn(\bu) \ge \overline{10}$  implies that $\bu = 0^k R(\bv')$ for some $k \ge 0$, $\bv' \in \{0,1\}^\infty$.
If $\sigma = M$, then $\supn(\bu) \le 01\overline{10}$ and $\infn(\bu) \ge 10\overline{01}$, thus $\bu = 0^k M(\bv')$ or $\bu = 1^k M(\bv')$ for some $k \ge 0$, $\bv' \in \{0,1\}^\infty$.

To show that $\bv'$ ends with some $\bv \in \Omega_{\ba,\bb}$, assume first that $\bv'$ starts with~$0$.
Then $\sigma(\supn(\bv')) = \supn(\sigma(\bv')) \le \sigma(\ba)$, thus $\supn(\bv') \le \ba$, and $\bv' = \overline{0} \in \Omega_{\ba,\bb}$ or $\bv' = 0^n \bv$ for some $\bv \in 1\{0,1\}^\infty$.
Then $\sigma(\infn(\bv)) = \infn(\sigma(\bv)) \ge \sigma(\bb)$, thus $\infn(\bv) \ge \bb$.
Since $\supn(\bv) \le \supn(\bv')$ (or $\bv = \overline{1}$), we obtain that $\bv \in \Omega_{\ba,\bb}$.
Since the case of $\bv'$ starting with~$1$ is symmetric, this proves the first statement.

Moreover, if $\sigma = L$, then $\ba < 0\overline{1}$, $\bb > 1 \overline{0}$ imply that $L(\ba) < \overline{01}$, $L(\bb) > 1 \overline{0}$, which gives together with $\bu \notin \{\overline{0}, \overline{1}\}$ that $\bv \notin \{\overline{0}, \overline{1}\}$.
In case $\sigma = R$, we also have $\bv \notin \{\overline{0}, \overline{1}\}$ by symmetry.

Let now $\sigma_1 \cdots \sigma_n \in \{L,R\}^*$, $n \ge 1$, $\bu \in \Omega_{\sigma_1 \cdots \sigma_n(\ba),\sigma_1\cdots\sigma_n(\bb)} \setminus \{\overline{0}, \overline{1}\}$, $\ba < 0\overline{1}$, $\bb > 1 \overline{0}$.
We have proved that $\bu \in \{0,1\}^* (\Omega_{\sigma_2\cdots\sigma_n(\ba),\sigma_2\cdots\sigma_n(\bb)} \setminus \{\overline{0}, \overline{1}\})$, and we obtain inductively that $\bu \in \{0,1\}^* (\Omega_{\ba,\bb} \setminus \{\overline{0}, \overline{1}\})$.
\end{proof}

The following lemma shows that the map
\begin{equation*}
\begin{aligned}
s:\ & \{0,1\}^\infty \to \{L,M,R\}^\infty \setminus \{L,M,R\}^*\{L\overline{R}, R\overline{L}\}, \\
& \mathbf{u} \mapsto \bsigma \quad \text{if $\bsigma(\overline{0}) \le \bu \le \bsigma(0\overline{1})$ or $\bsigma(1\overline{0}) \le \bu \le \bsigma(\overline{1})$},
\end{aligned}
\end{equation*}
is well-defined and monotonically increasing on $0\{0,1\}^\infty$ as well as on $1\{0,1\}^\infty$, where sequences in $\{L,M,R\}^\infty$ are ordered lexicographically with $L<M<R$.
Some of the values of $s$ are shown in Figure~\ref{f:IJ}.
Note that $\bsigma(i_1i_2\cdots)$ depends only on~$i_1$ when the length of $\sigma_1 \cdots \sigma_n(i_1)$ is unbounded.
In particular, we have $\bsigma(\overline{0}) \,{=}\, \bsigma(0\overline{1})$ and $\bsigma(1\overline{0}) \,{=}\, \bsigma(\overline{1})$ for primitive~$\bsigma$, $\sigma\overline{R}(\overline{0}) \,{=}\, \sigma\overline{R}(0\overline{1})$ and  $\sigma\overline{L}(1\overline{0}) \,{=}\, \sigma\overline{L}(\overline{1})$ for all $\sigma \in \{L,M,R\}^*$, and
\begin{equation} \label{e:LR}
\begin{aligned}
\overline{L}(\overline{0}) & = \overline{0}, & \overline{L}(\overline{1}) = \overline{L}(1\overline{0}) & = 1\overline{0}, & \overline{L}(0\overline{1}) & = 01\overline{0}, \\
\overline{R}(\overline{1}) & = \overline{1}, & \overline{R}(\overline{0}) = \overline{R}(0\overline{1}) & = 0\overline{1}, & \overline{R}(1\overline{0}) & = 10\overline{1}.
\end{aligned}
\end{equation}

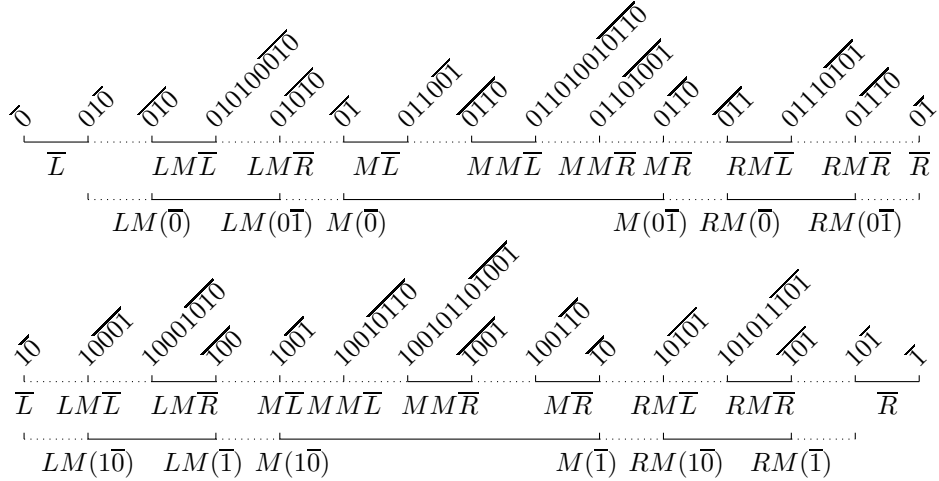
\begin{figure}[ht]
\centerline{\begin{tikzpicture}[scale=.85]
\draw(-1,.1)--(-1,0)node[above right,rotate=45]{$\overline{0}$};
\draw(0,.1)--(0,0)node[above right,rotate=45]{$01\overline{0}$};
\draw(1,.1)--(1,0)node[above right,rotate=45]{$\overline{010}$};
\draw(2,.1)--(2,0)node[above right,rotate=45]{$010100\overline{010}$};
\draw(3,.1)--(3,0)node[above right,rotate=45]{$01\overline{010}$};
\draw(4,.1)--(4,0)node[above right,rotate=45]{$\overline{01}$};
\draw(5,.1)--(5,0)node[above right,rotate=45]{$0110\overline{01}$};
\draw(6,.1)--(6,0)node[above right,rotate=45]{$\overline{0110}$};
\draw(7,.1)--(7,0)node[above right,rotate=45]{$01101001\overline{0110}$};
\draw(8,.1)--(8,0)node[above right,rotate=45]{$0110\overline{1001}$};
\draw(9,.1)--(9,0)node[above right,rotate=45]{$01\overline{10}$};
\draw(10,.1)--(10,0)node[above right,rotate=45]{$\overline{011}$};
\draw(11,.1)--(11,0)node[above right,rotate=45]{$01110\overline{101}$};
\draw(12,.1)--(12,0)node[above right,rotate=45]{$01\overline{110}$};
\draw(13,.1)--(13,0)node[above right,rotate=45]{$0\overline{1}$};
\draw[dotted](0,0)--(1,0) (2,0)--(4,0) (5,0)--(6,0) (7,0)--(10,0) (11,0)--(13,0);
\draw(-1,0)--(0,0) (1,0)--(2,0) (4,0)--(5,0)  (6,0)--(7,0) (10,0)--(11,0);
\node[below] at (-.5,0){$\overline{L}$};
\node[below] at (1.5,0){$LM\overline{L}$};
\node[below] at (3,0){$LM\overline{R}$};
\node[below] at (4.5,0){$M\overline{L}$};
\node[below] at (6.5,0){$MM\overline{L}$};
\node[below] at (7.95,0){$MM\overline{R}$};
\node[below] at (9.05,0){$M\overline{R}$};
\node[below] at (10.5,0){$RM\overline{L}$};
\node[below] at (12,0) {$RM\overline{R}$};
\node[below] at (13,0){$\overline{R}$};
\draw[dotted](0,-.9)--(1,-.9) (3,-.9)--(4,-.9) (9,-.9)--(10,-.9) (12,-.9)--(13,-.9);
\draw(0,-.9)--(0,-.8) (1,-.8)--(1,-.9)--(3,-.9)--(3,-.8) (4,-.8)--(4,-.9)--(9,-.9)--(9,-.8) (10,-.8)--(10,-.9)--(12,-.9)--(12,-.8) (13,-.8)--(13,-.9);
\node[below] at (1,-.9){$LM(\overline{0})$};
\node[below] at (2.8,-.9){$LM(0\overline{1})$};
\node[below] at (4.2,-.9){$M(\overline{0})$};
\node[below] at (8.8,-.9){$M(0\overline{1})$};
\node[below] at (10.2,-.9){$RM(\overline{0})$};
\node[below] at (12,-.9){$RM(0\overline{1})$};
\begin{scope}[shift={(-1,-3.75)}]
\draw(0,.1)--(0,0)node[above right,rotate=45]{$1\overline{0}$};
\draw(1,.1)--(1,0)node[above right,rotate=45]{$10\overline{001}$};
\draw(2,.1)--(2,0)node[above right,rotate=45]{$10001\overline{010}$};
\draw(3,.1)--(3,0)node[above right,rotate=45]{$\overline{100}$};
\draw(4,.1)--(4,0)node[above right,rotate=45]{$10\overline{01}$};
\draw(5,.1)--(5,0)node[above right,rotate=45]{$1001\overline{0110}$};
\draw(6,.1)--(6,0)node[above right,rotate=45]{$10010110\overline{1001}$};
\draw(7,.1)--(7,0)node[above right,rotate=45]{$\overline{1001}$};
\draw(8,.1)--(8,0)node[above right,rotate=45]{$1001\overline{10}$};
\draw(9,.1)--(9,0)node[above right,rotate=45]{$\overline{10}$};
\draw(10,.1)--(10,0)node[above right,rotate=45]{$10\overline{101}$};
\draw(11,.1)--(11,0)node[above right,rotate=45]{$101011\overline{101}$};
\draw(12,.1)--(12,0)node[above right,rotate=45]{$\overline{101}$};
\draw(13,.1)--(13,0)node[above right,rotate=45]{$10\overline{1}$};
\draw(14,.1)--(14,0)node[above right,rotate=45]{$\overline{1}$};
\draw[dotted](0,0)--(2,0) (3,0)--(6,0) (7,0)--(8,0) (9,0)--(11,0) (12,0)--(13,0);
\draw(2,0)--(3,0) (6,0)--(7,0)  (8,0)--(9,0) (11,0)--(12,0) (13,0)--(14,0);
\node[below] at (0,0){$\overline{L}$};
\node[below] at (1,0){$LM\overline{L}$};
\node[below] at (2.5,0){$LM\overline{R}$};
\node[below] at (4,0){$M\overline{L}$};
\node[below] at (5,0){$MM\overline{L}$};
\node[below] at (6.5,0){$MM\overline{R}$};
\node[below] at (8.5,0){$M\overline{R}$};
\node[below] at (10,0){$RM\overline{L}$};
\node[below] at (11.5,0) {$RM\overline{R}$};
\node[below] at (13.5,0){$\overline{R}$};
\draw[dotted](0,-.9)--(1,-.9) (3,-.9)--(4,-.9) (9,-.9)--(10,-.9) (12,-.9)--(13,-.9);
\draw(0,-.8)--(0,-.9) (1,-.8)--(1,-.9)--(3,-.9)--(3,-.8) (4,-.8)--(4,-.9)--(9,-.9)--(9,-.8) (10,-.8)--(10,-.9)--(12,-.9)--(12,-.8) (13,-.8)--(13,-.9);
\node[below] at (1,-.9){$LM(1\overline{0})$};
\node[below] at (2.8,-.9){$LM(\overline{1})$};
\node[below] at (4.2,-.9){$M(1\overline{0})$};
\node[below] at (8.8,-.9){$M(\overline{1})$};
\node[below] at (10.2,-.9){$RM(1\overline{0})$};
\node[below] at (12,-.9){$RM(\overline{1})$};
\end{scope}
\end{tikzpicture}}
\caption{Some values of~$s$ and some intervals $[\sigma(\overline{0}), \sigma(0\overline{1})]$, $[\sigma(1\overline{0}), \sigma(\overline{1})]$, $\sigma \in \{L,R\}^*M$.} \label{f:IJ}
\end{figure}

\begin{lemma} \label{l:partition}
We have the partitions (with lexicographic intervals)
\begin{equation*}
\begin{aligned}
{}[\overline{0}, 0\overline{1}] \ & = \hspace{-4em} \bigcup_{\bsigma\in\{L,M,R\}^\infty \setminus \{L,M,R\}^*\{L\overline{R},R\overline{L}\}} \hspace{-4em} [\bsigma(\overline{0}), \bsigma(0\overline{1})], & (01\overline{0}, 0\overline{1}) & = \hspace{-1em} \bigcup_{\sigma\in\{L,R\}^*M} \hspace{-1em} [\sigma(\overline{0}), \sigma(0\overline{1})] \ \cup \hspace{-2.5em} \bigcup_{\bsigma\in\{L,R\}^\infty\,\text{primitive}} \hspace{-2.5em} \{\bsigma(\overline{0})\}, \\
{}[1\overline{0}, \overline{1}] \ & = \hspace{-4em} \bigcup_{\bsigma\in\{L,M,R\}^\infty \setminus \{L,M,R\}^*\{L\overline{R},R\overline{L}\}} \hspace{-4em} [\bsigma(1\overline{0}), \bsigma(\overline{1})], & (1\overline{0}, 10\overline{1}) & = \hspace{-1em} \bigcup_{\sigma\in\{L,R\}^*M} \hspace{-1em} [\sigma(1\overline{0}), \sigma(\overline{1})] \ \cup \hspace{-2.5em} \bigcup_{\bsigma\in\{L,R\}^\infty\,\text{primitive}} \hspace{-2.5em} \{\bsigma(\overline{1})\}. \end{aligned}
\end{equation*}
For $\bsigma, \btau \in \{L,M,R\}^\infty \setminus \{L,M,R\}^*\{L\overline{R},R\overline{L}\}$ with $\bsigma < \btau$, we have $\bsigma(0\overline{1}) < \btau(\overline{0})$ and $\bsigma(\overline{1}) < \btau(1\overline{0})$.
\end{lemma}

\begin{proof}
Since $L$, $M$ and~$R$ are strictly monotonically increasing on $\{0,1\}^\infty$, we have, for any $\sigma \in \{L,M,R\}^*$, the partition of the open interval
\begin{equation} \label{e:partJ}
\begin{aligned}
\hspace{-2em}
\big(\sigma(01\overline{0}), \sigma(0\overline{1})\big) & = \underbrace{\big(\sigma(01\overline{0}), \sigma(\overline{01})\big)}_{(\sigma L(01\overline{0}), \sigma L(0\overline{1}))} \, \cup \, \underbrace{\big[\sigma(\overline{01}), \sigma(0110\overline{01})\big]}_{[\sigma M\overline{L}(\overline{0}), \sigma M\overline{L}(0\overline{1})]} \\
& \hspace{-2em} \cup \, \underbrace{\big(\sigma(0110\overline{01}), \sigma(01\overline{10})\big)}_{(\sigma M(01\overline{0}), \sigma M(0\overline{1}))} \, \cup \hspace{-.5em} \underbrace{\big\{\sigma(01\overline{10})\big\}}_{[\sigma M\overline{R}(\overline{0}), \sigma M\overline{R}(0\overline{1})]} \hspace{-.5em} \cup \, \underbrace{\big(\sigma(01\overline{10}), \sigma(0\overline{1})\big)}_{(\sigma R(01\overline{0}), \sigma R(0\overline{1}))}.
\end{aligned}
\end{equation}
Starting from $\sigma = \mathrm{id}$ and iterating this partition for $\sigma \in \{L,M,R\}^*$, we get that
\begin{equation*}
(01\overline{0}, 0\overline{1}) = \hspace{-1em} \bigcup_{\sigma\in\{L,M,R\}^*} \hspace{-1em} \big([\sigma M\overline{L}(\overline{0}), \sigma M\overline{L}(0\overline{1})] \ \cup \ [\sigma M\overline{R}(\overline{0}), \sigma M\overline{R}(0\overline{1})]\big) \ \cup \hspace{-2.5em} \bigcup_{\bsigma\in\{L,M,R\}^\infty\,\text{primitive}} \hspace{-2.5em} \{\bsigma(\overline{0})\}.
\end{equation*}
Here, we have used that all sequences outside the union over $\{L,M,R\}^*$ are in $\bigcap_{n\ge1} (\sigma_1\cdots\sigma_n(01\overline{0}), \sigma_1\cdots\sigma_n(0\overline{1}))$ for some $\bsigma = (\sigma_n)_{n\ge1} \in \{L,M,R\}^\infty$; this intersection is empty when $\bsigma$ ends with $\overline{L}$ or $\overline{R}$ because the intervals are open and converge to their left or right endpoint, and it consists of $\bsigma(\overline{0})$ when $\bsigma$ is primitive.
Since $[\overline{L}(\overline{0}), \overline{L}(0\overline{1})] = [\overline{0}, 01\overline{0}]$, $[\overline{R}(\overline{0}), \overline{R}(0\overline{1})] = \{0\overline{1}\}$, and
\begin{equation*}
\begin{aligned}
& \{L,M,R\}^\infty \setminus \{L,M,R\}^* \{L\overline{R},R\overline{L}\} \\
& \quad = \{\bsigma \in \{L,M,R\}^\infty \,:\, \bsigma\ \text{primitive}\} \cup \{L,M,R\}^* M \{\overline{L},\overline{R}\} \cup \{\overline{L},\overline{R}\},
\end{aligned}
\end{equation*}
this shows that the intervals $[\bsigma(\overline{0}), \bsigma(0\overline{1})]$ form a partition of $[\overline{0}, 0\overline{1}]$, and that $\bsigma(0\overline{1}) < \btau(\overline{0})$ for all $\bsigma, \btau \in \{L,M,R\}^\infty \setminus \{L,M,R\}^*\{L\overline{R},R\overline{L}\}$ with $\bsigma < \btau$.

From~\eqref{e:partJ}, we also see that $[\sigma_1\cdots\sigma_n(\overline{0}), \sigma_1\cdots\sigma_n(0\overline{1})]$, $\sigma_1\cdots\sigma_n \in \{L,R\}^*M$, is the union over all $[\btau(\overline{0}), \btau(0\overline{1})]$, $\btau \in \{L,M,R\}^\infty \setminus \{L,M,R\}^*\{L\overline{R},R\overline{L}\}$ with $\tau_1 \cdots \tau_n = \sigma_1 \cdots \sigma_n$, which proves the partition of $(01\overline{0}, 0\overline{1})$.

The proofs for sequences starting with~$1$ are symmetric, by exchanging $0$ and~$1$, $L$ and~$R$, as well as the left and right endpoints of the intervals.
\end{proof}

Using the map~$s$, we can write (\ref{i:lex3})--(\ref{i:lex5}) of Theorem~\ref{t:lex} in a simpler way.

\begin{proposition} \label{p:lex}
Let $\ba \in 0\{0,1\}^\infty$, $\bb \in 1\{0,1\}^\infty$. Then we have the following.
\begin{enumerate}[\upshape (i')]
\itemsep.5ex
\setcounter{enumi}{2}
\item \label{i:plex3}
$\Omega_{\ba,\bb}$ is uncountable with positive entropy if and only if $s(\ba) > s(\bb)$.
\item \label{i:plex4}
$\Omega_{\ba,\bb}$ is uncountable with zero entropy if and only if $s(\ba) = s(\bb)$ is primitive.
\item \label{i:plex5}
$\Omega_{\ba,\bb}$ is countable if and only if $s(\ba) < s(\bb)$ or $s(\ba) = s(\bb)$ ends with $\overline{L}$ or~$\overline{R}$.
\end{enumerate}
\end{proposition}

\begin{proof}[Proof of Theorem~\ref{t:LM}]
For all $\bu \in \{0,1\}^\infty$, $\ba \le \bu \le \bb$ is equivalent to $0\bu \le 0\bb$ and $1\bu \ge 1\ba$, thus $\bu \in \Sigma_{\ba,\bb}$ implies that $\bu$, $0\bu$ and $1\bu$ are in $\Omega_{0\bb,1\ba}$.
Moreover, $0\bu \in \Omega_{0\bb,1\ba}$ implies that $00\bu \le 0\bu \le \ba$, thus $00\bu \in \Omega_{0\bb,1\ba}$, similarly $1\bu \in \Omega_{0\bb,1\ba}$ gives that $11\bu \in \Omega_{0\bb,1\ba}$.
On the other hand, $01\bu \in \Omega_{0\bb,1\ba}$ implies that $1\bu \in \Sigma_{\ba,\bb}$.
Indeed, for $1\bu = i_1i_2\cdots$, we see that $\ba \le i_ni_{n+1}\cdots \le \bb$ for all $n \ge 1$ by induction on~$n$.
This holds for $n=1$ because $01\bu \ge 0\bb$ implies that $\ba \le 1\bu \le \bb$;
if it holds for $n$, then $i_n = 1$ implies that $\ba \le i_{n+1} i_{n+2} \cdots \le i_n i_{n+1} \cdots \le \bb$ since $i_n i_{n+1}\cdots \ge 1\ba$, and $i_n = 0$ implies that $\ba \le i_n i_{n+1}\cdots \le i_{n+1} i_{n+2} \cdots \le \bb$ since $i_n i_{n+1}\cdots \le 0\bb$, thus it holds for $n{+}1$. 
Similarly, $10\bu \in \Omega_{0\bb,1\ba}$ implies that $0\bu \in \Sigma_{\ba,\bb}$.
This shows that $\Omega_{0\bb,1\ba} = 0^* \Sigma_{\ba,\bb} \cup 1^* \Sigma_{\ba,\bb} \cup \{\overline{0}, \overline{1}\}$, hence $\Sigma_{\ba,\bb} = \emptyset$ if and only if $\Omega_{0\bb,1\ba} = \{\overline{0}, \overline{1}\}$, and the countability of $\Sigma_{\ba,\bb}$ and $\Omega_{0\bb,1\ba}$ are equivalent.
Since
\[
A_n(\Sigma_{\ba,\bb}) \le  A_n(\Omega_{0\bb,1\ba}) \le 2\sum_{k=0}^n A_k(\Sigma_{\ba,\bb}) \le 2(n{+}1) A_n(\Sigma_{\ba,\bb}),
\]
$\Sigma_{\ba,\bb}$ and $\Omega_{0\bb,1\ba}$ have the same entropy; see also \cite[Lemma~2.5]{KomKonLi2017}.
\end{proof}

\begin{proof}[Proof of Theorem~\ref{t:lex} and Proposition~\ref{p:lex}]
We first prove (\ref{i:lex1}) and (\ref{i:lex2}).
Let $\sigma \in \{L,R\}^*$.
Then $\supn(\sigma M(\overline{0})) = \sigma M(\overline{0})$, $\infn(\sigma M(\overline{0})) \,{=}\, \infn(\sigma M(\overline{1})) \,{=}\, \sigma M(\overline{1})$ by \eqref{e:infsup}, thus $\sigma M(\overline{0}) \in \Omega_{\sigma M(\overline{0}),\sigma M(\overline{1})} \subseteq \Omega_{\ba,\bb}$ for all $\ba \ge \sigma M(\overline{0})$, $\bb \le \sigma M(\overline{1})$.
On the other hand, by Lemma~\ref{l:Omegasigma}, each $\bu \in \Omega_{\sigma M(\overline{0}),\sigma M(1\overline{0})} \setminus \{\overline{0}, \overline{1}\}$ ends with $\sigma(\bv)$ for some $\bv \in \Omega_{M(\overline{0}),M(1\overline{0})} \setminus \{\overline{0}, \overline{1}\}$ and thus with $\sigma M(\overline{0})$ since $\Omega_{\overline{0},1\overline{0}} = 1^* \{\overline{0}, \overline{1}\}$.
Therefore, we have $\Omega_{\ba,\bb} = \{\overline{0}, \overline{1}\}$ for all $\ba < \sigma M(\overline{0})$, $\bb \ge \sigma M(1\overline{0})$.
By symmetry, $\Omega_{\ba,\bb}$ is also trivial for all $\ba \le \sigma M(0\overline{1})$, $\bb > \sigma M(\overline{1})$.

If $\bsigma \in \{L,R\}^\infty$ is primitive, then $\supn(\bsigma(\overline{0})) = \bsigma(\overline{0})$, $\infn(\bsigma(\overline{0})) = \infn(\bsigma(\overline{1})) = \bsigma(\overline{1})$ by \eqref{e:infsup} and since $X_{\bsigma(\overline{0})} = X_{\bsigma(\overline{1})}$ by \cite[Theorem~5.2]{BerDel2014}.
This implies that $\bsigma(\overline{0}) \in \Omega_{\bsigma(\overline{0}),\bsigma(\overline{1})} \subseteq \Omega_{\ba,\bb}$ for all $\ba \ge \bsigma(\overline{0})$, $\bb \le \bsigma(\overline{1})$.
Since $0\overline{1} \in \Omega_{0\overline{1},\bb}$ and  $1\overline{0} \in \Omega_{\ba, 1\overline{0}}$ for all $\ba, \bb \in \{0,1\}^\infty$, we have proved the ``if'' parts of (\ref{i:lex1}) and~(\ref{i:lex2}).

For the ``only if'' parts, since $\Omega_{\ba,\bb}$ is either trivial or not, it suffices to show that all $\ba\in 0\{0, 1\}^\infty, \bb \in 1\{0,1\}^\infty$ satisfy the conditions of (\ref{i:lex1}) or (\ref{i:lex2}).
If $\ba = 0\overline{1}$ or $\bb =1\overline{0}$, then we are in case~(\ref{i:lex1}).
For $\ba < 0\overline{1}$, $\bb > 1\overline{0}$, consider $s(\ba) \,{=}\, (\sigma_n)_{n\ge 1} \,{<}\, \overline{R}$, $s(\bb) \,{=}\, (\tau_n)_{n\ge1} \,{>}\, \overline{L}$.
If $\sigma_1 \cdots \sigma_k = \tau_1 \cdots \tau_k \in \{L,R\}^*$ and $\sigma_{k+1} \ge M \ge \tau_{k+1}$ for some $k \ge 0$ (where $\sigma_1 \cdots \sigma_k$ is the identity if $k=0$), then $\ba \ge \sigma_1 \cdots \sigma_k M(\overline{0})$, $\bb \le \sigma_1 \cdots \sigma_k M(\overline{1})$ by Lemma~\ref{l:partition}, and we are in case~(\ref{i:lex1}).
Otherwise, we have $\sigma_1 \cdots \sigma_k = \tau_1 \cdots \tau_k \in \{L,R\}^*$ and $\sigma_{k+1} < \tau_{k+1}$ for some $k \ge 0$, thus $\ba < \sigma_1 \cdots \sigma_k M(\overline{0})$, $\bb \ge \sigma_1 \cdots \sigma_k M(1\overline{0})$ or $\ba \le \sigma_1 \cdots \sigma_n M(0\overline{1})$, $\bb > \sigma_1 \cdots \sigma_n M(\overline{1})$ by Lemma \ref{l:partition}, i.e., we are in case~(\ref{i:lex2}).
This concludes the proof of (\ref{i:lex1}) and~(\ref{i:lex2}).

\smallskip
It remains to prove (\ref{i:lex3})--(\ref{i:lex5}) and (\ref{i:plex3}')--(\ref{i:plex5}').
Let $\sigma \,{\in}\, \{L,M,R\}^*$.
If $\bb < \sigma M(1\overline{0}) = \sigma(10\overline{01})$, then $\bb < \sigma(1\overline{0(01)^k})$ for some $k \ge 0$.
For all $\bu \in \{0(01)^k, 0(01)^{k+1}\}^\infty$, we have $\infn(\sigma(\bu)) \ge \infn(\sigma(1\bu)) = \sigma(\infn(1\bu)) \ge \sigma(1\overline{0(01)^k})$  and $\supn(\sigma(\bu)) = \sigma(\supn(\bu)) < \sigma(\overline{01})$, thus $\sigma(\bu) \in \Omega_{\sigma M(\overline{0}),\bb}$.
Therefore, the entropy of $h(\Omega_{\ba,\bb})$ is positive for all $\ba \,{\ge}\, \sigma M(\overline{0})$, $\bb \,{<}\, \sigma M(1\overline{0})$.
For $\ba > \sigma M(0\overline{1})$, $\bb \le \sigma M(\overline{1})$, we obtain symmetrically that $h(\Omega_{\ba,\bb}) > 0$.
This proves the ``if'' part of~(\ref{i:lex3}).

\smallskip
The set $\Omega_{01\overline{0},1\overline{0}} = \{\overline{0}, \overline{1}\} \cup 1^*0^*1\overline{0}$ is countable and $\bu \in \Omega_{\sigma_1\cdots\sigma_n(01\overline{0}),\sigma_1\cdots\sigma_n(1\overline{0})}$, $\sigma_1 \cdots \sigma_n \in \{L,M,R\}^n$, implies by Lemma~\ref{l:Omegasigma} that $\bu \in \{0,1\}^* \sigma_1\cdots\sigma_n(\Omega_{01\overline{0},1\overline{0}})$ or $\bu \in \{0,1\}^* \sigma_1\cdots\sigma_k(\{\overline{0}, \overline{1}\})$, $0 \le k < n$.
Therefore, $\Omega_{\ba,\bb}$ is countable for all $\ba \,{\le}\, \sigma(01\overline{0})$, $\bb \,{\ge}\, \sigma(1\overline{0})$, $\sigma \,{\in}\, \{L,M,R\}^*$.
The case $\ba \,{\le}\, \sigma(0\overline{1})$, $\bb \,{\ge}\, \sigma(10\overline{1})$ is symmetric, thus the ``if'' part of~(\ref{i:lex5}) holds.

\smallskip
Let now $\bsigma \in \{L,M,R\}^\infty$ be primitive, which implies that $X_{\bsigma(\overline{0})} \subseteq \Omega_{\bsigma(\overline{0}), \bsigma(\overline{1})}$.
If $\bsigma = (\sigma_n)_{n\ge1}$ ends with~$\overline{M}$, then $\bsigma(\overline{0})$ is an image of the Thue--Morse word, thus $X_{\bsigma(\overline{0})}$ is uncountable.
If $\bsigma$ does not end with~$\overline{M}$, then there exists an increasing sequence of integers $(n_k)_{k\ge0}$ with $n_0=0$ such that $\sigma_{n_{k+1}} \ne M$ and $\sigma_{n_k+1} \cdots \sigma_{n_{k+1}}(j)$ contains $01$ and $10$ for all $j \in \{0,1\}$, $k \ge 0$.
Then, for all $i \in \{0,1\}$, $k \ge 0$, there is a word $w_{i,k} \in \{0,1\}^*$ such that both $\sigma_{n_k+1} \cdots \sigma_{n_{k+1}}(0)$ and $\sigma_{n_k+1} \cdots \sigma_{n_{k+1}}(1)$ end with $(1{-}i)\, i\, w_{i,k}$; note that $\sigma_{n_k+1} \cdots \sigma_{n_{k+1}}(0)$ ends with $\sigma_{n_k+1} \cdots \sigma_{n_{k+1}}(1)$ or vice versa because $\sigma_{n_{k+1}} \in \{L,R\}$.
For each sequence $(i_k)_{k\ge1} \in \{0,1\}^\infty$, we have
\begin{equation*}
i_0 w_{i_0,1}\, \sigma_1 \cdots \sigma_{n_1}(i_1 w_{i_1,1}) \, \sigma_1 \cdots \sigma_{n_2}(i_2 w_{i_2,2})\, \cdots \in X_{\bsigma(\overline{0})}
\end{equation*}
because $\sigma_1 \cdots \sigma_{n_{k+1}}(1{-}i_{k+1})$ ends with $\sigma_1 \cdots \sigma_{n_k}((1{-}i_k) i_k w_{i_k,k})$, and different sequences give different elements of $X_{\bsigma(\overline{0})}$ because $\sigma_1 \cdots \sigma_{n_k}(i_k)$ starts with~$i_k$.
Therefore, $\Omega_{\bsigma(\overline{0}),\bsigma(\overline{1})}$ is uncountable.
For all $\ba < \bsigma(\overline{0})$, we have $\ba \le \sigma_1 \cdots \sigma_n(01\overline{0})$ for some $n \ge 0$; since $\bsigma(\overline{1}) \ge \sigma_1 \cdots \sigma_n(1\overline{0})$, we have seen above that $\Omega_{\ba,\bsigma(\overline{1})}$ is countable by~(\ref{i:lex5}).
By Theorem~\ref{t:LM}, this implies that $\Sigma_{\ba,\bb}$ is countable for $0\bb < \bsigma(\overline{0})$, $1\ba = \bsigma(\overline{1})$, thus $h(\Sigma_{\ba,\bb}) = 0$.
Since $h(\Sigma_{\ba,\bb})$ is a continuous function of $\bb$ by \cite[Theorem~4]{LabMor2006}, we also have $h(\Sigma_{\ba,\bb}) = 0$ for $0\bb = \bsigma(\overline{0})$, thus $h(\Omega_{\bsigma(\overline{0}),\bsigma(\overline{1})}) = 0$.
This proves the ``if'' part of~(\ref{i:lex4}).

For the ``only if'' parts of (\ref{i:lex3})--(\ref{i:lex5}), since $\Omega_{\ba,\bb}$ is either countable or uncountable, with zero or positive entropy, it suffices to show that all $\ba, \bb$ satisfy the conditions of (\ref{i:lex3}), (\ref{i:lex4}) or (\ref{i:lex5}).
Let $s(\ba) = (\sigma_n)_{n\ge1}$, $s(\bb) = (\tau_n)_{n\ge1}$.
If $s(\ba) > s(\bb)$, i.e., $\sigma_1 \cdots \sigma_k = \tau_1 \cdots \tau_k$, $\sigma_{k+1} > \tau_{k+1}$ for some $k \ge 0$, then $\ba \ge \sigma_1 \cdots \sigma_k M(\overline{0})$, $\bb < \sigma_1 \cdots \sigma_k M(1\overline{0})$, or $\ba > \sigma_1 \cdots \sigma_k M(0\overline{1})$, $\bb \le \sigma_1 \cdots \sigma_k M(\overline{1})$, thus we are in case~(\ref{i:lex3}).
If $s(\ba) < s(\bb)$, i.e., $\sigma_1 \cdots \sigma_k = \tau_1 \cdots \tau_k$, $\sigma_{k+1} < \tau_{k+1}$ for some $k \ge 0$, then $\ba < \sigma_1 \cdots \sigma_k M(\overline{0}) \le \sigma_1 \cdots \sigma_k M(01\overline{0})$, $\bb \ge \sigma_1 \cdots \sigma_k M(1\overline{0})$, or $\ba \le \sigma_1 \cdots \sigma_k M(0\overline{1})$, $\bb > \sigma_1 \cdots \sigma_k M(\overline{1}) \ge \sigma_1 \cdots \sigma_k M(10\overline{1})$, thus we are in case~(\ref{i:lex5}).
If $s(\ba) = s(\bb) = \sigma \overline{L}$, $\sigma \in \{L,M,R\}^*$, then $\ba \le \sigma\overline{L}(0\overline{1}) = \sigma(01\overline{0})$, $\bb \ge \sigma\overline{L}(1\overline{0}) = \sigma(1\overline{0})$, thus we are in case~(\ref{i:lex5}).
The case $s(\bb) \,{=}\, s(\ba) \,{=}\, \sigma \overline{R}$ is symmetric, and the case of primitive $s(\ba) \,{=}\, s(\bb)$ is~(\ref{i:lex4}).
This concludes the proof of (\ref{i:lex3})--(\ref{i:lex5}) as well as (\ref{i:plex3}')--(\ref{i:plex5}').
\end{proof}

\section{The maps $g,\tilde{g}$} \label{s:g}
Recall that $\ba_{q_0,q_1}$ and $\bb_{q_0,q_1}$ are defined in Section~\ref{s:lex}, and that $g_{\bu}(q_0) > 1$, $\tilde{g}_{\bv}(q_0) > 1$ (when defined) satisfy $f_{\bu}(q_0,g_{\bu}(q_0)) = 0$, $\tilde{f}_{\bv}(q_0,\tilde{g}_{\bv}(q_0)) = 0$, with
\begin{equation*}
f_{\bu}(q_0, q_1) := q_0\, (q_1\, \pi_{q_0,q_1}(\bu) - 1), \quad \tilde{f}_{\bv}(q_0,q_1) := q_1\, (q_0\, \tilde{\pi}_{q_0,q_1}(\bv) - 1).
\end{equation*}
We set $g_{\bu}(q_0) := 1$ when $f_{\bu}(q_0, q_1) = 0$ has no solution $q_1 > 1$.

We are only interested in $\bu \in W$ and $\bv \in \tilde{W}$, with
\begin{align*}
W & := \big\{\bu \in \{0,1\}^\infty \setminus \{0,1\}^* \{\overline{0}, \overline{1}\} \,:\, \supn(\bu) = \bu\big\}, \\
\tilde{W} & := \big\{\bv \in \{0,1\}^\infty \setminus \{0,1\}^* \{\overline{0}, \overline{1}\} \,:\, \infn(\bv) = \bv\big\}.
\end{align*}

\begin{lemma} \label{l:functiong}
Let $\bu \in W$.
Then the following holds.
\begin{enumerate}[\upshape(i)]
\item \label{i:g1}
For $q_0 > 1$, $q_1 \ge 1$, the function $f_{\bu}(q_0, q_1)$ is continuous and strictly decreasing in both variables $q_0,q_1$, and there is a unique $q_0 > 1$ such that $f_{\bu}(q_0,1) = 0$; we denote this $q_0$ by~$q_{\bu}$.
\item \label{i:g3}
The function $g_{\bu}(q_0)$ is continuous on $(1,\infty)$ and strictly decreasing on $(1,q_{\bu})$, with $\lim_{q_0\to 1} g_{\bu}(q_0) = \infty$ and $g_{\bu}(q_0) = 1$ for all $q_0 \ge q_{\bu}$.
\item \label{i:g5}
For $q_0, q_1 > 1$ with $q_0+q_1 \geq q_0q_1$, we have $\ba_{q_0,q_1} > \bu$ if and only if $q_1 > g_{\bu}(q_0)$.
For $q_0 \in (1,q_{\bu})$, we have $\ba_{q_0, g_{\bu}(q_0)} = \bu$.
\end{enumerate}
\end{lemma}

\begin{proof}
(\ref{i:g1}) Writing $\bu = 01i_1i_2\cdots$, we have
\begin{equation*}
f_{\bu}(q_0,q_1) = 1-q_0+\sum_{k=1}^\infty \frac{i_k}{q_{i_1}q_{i_2}\cdots q_{i_k}}.
\end{equation*}
The continuity follows from the local uniform convergence of this series; for $q_1 = 1$, note that the $0$s occur in~$\bu$ with bounded distance since $\supn(\bu) = \bu < 0\overline{1}$.
We have $\frac{\partial}{\partial q_0}f_{\bu}(q_0,q_1) < 0$ and, since $i_1i_2 \cdots \ne \overline{0}$, $\frac{\partial}{\partial q_1}f_{\bu}(q_0,q_1) < 0$ for all $q_0,q_1>1$.
Since $\lim_{q_0\to1} f_{\bu}(q_0,1) = \infty$ and $\lim_{q_0\to\infty} f_{\bu}(q_0,1) = -\infty$, the equation $f_{\bu}(q_0,1) = 0$ has a unique solution $q_0 > 1$.

\medskip\noindent
(\ref{i:g3}) The existence, uniqueness and monotonicity of $g_{\bu}(q_0)$ on $(1,q_{\bu})$ follows from the continuity and monotonicity of $f_{\bu}(q_0,q_1)$, from $\lim_{q_1\to\infty} f_{\bu}(q_0,q_1) = 1 - q_0 < 0$ and $f_{\bu}(q_0,1) > f_{\bu}(q_{\bu},1) = 0$ for all $q_0 \in (1,q_{\bu})$.
The continuity of $g_{\bu}(q_0)$ follows from the monotone version of the implicit function theorem (see e.g.\ \cite[p.~423]{Fic1992}) and $f_{\bu}(q_{\bu},1) = 0$.
If $\bu$ starts with $010^k1$, $k \ge 0$, then $f_{\bu}(q_0,q_1) \ge 1-q_0 + \frac{1}{q_0^kq_1}$ and thus $g_{\bu}(q_0) \ge \frac{1}{q_0^k(q_0-1)}$, which implies that $\lim_{q_0\to1} g_{\bu}(q_0) = \infty$.

\medskip\noindent
(\ref{i:g5}) If $\bu < \ba_{q_0,q_1}$, then $\bu$ is a quasi-greedy $(q_0,q_1)$-expansion by \cite[Theorem~1.2]{KomLuZou2022}, and the monotonicity of quasi-greedy expansions gives $\pi_{q_0,q_1}(\bu) < \frac{1}{q_1}$, which is equivalent to $q_1 > g_{\bu}(q_0)$ by (\ref{i:g1}) and~(\ref{i:g3}).
On the other hand, if $\pi_{q_0,q_1}(\bu) < \frac{1}{q_1}$, then the quasi-greedy $(q_0,q_1)$-expansion of $\pi_{q_0,q_1}(\bu)$ is strictly smaller than~$\ba_{q_0,q_1}$, which implies that $\bu < \ba_{q_0,q_1}$.
If $q_1 = g_{\bu}(q_0)$, i.e., $\pi_{q_0,q_1}(\bu) = \frac{1}{q_1}$, then $\ba_{q_0,q_1} = \bu$.
\end{proof}

\begin{lemma} \label{l:functiongt}
Let $\bv \in \tilde{W}$.
Then the following holds.
\begin{enumerate}[\upshape(i)]
\item \label{i:gt1}
For $q_0 \ge 1$, $q_1 > 1$, the function $\tilde{f}_{\bv}(q_0, q_1)$ is continuous and strictly decreasing in both variables $q_0,q_1$.
\item \label{i:gt3}
The function $\tilde{g}_{\bv}(q_0)$ is continuous and strictly decreasing on $(1,\infty)$, with $1 < \lim_{q_0\to1} \tilde{g}_{\bv}(q_0) < \infty$ and $\lim_{q_0\to \infty} \tilde{g}_{\bv}(q_0) = 1$.
\item \label{i:gt5}
For $q_0, q_1 > 1$ with $q_0+q_1 \geq q_0q_1$, we have $\bb_{q_0,q_1} < \bv$ if and only if $q_1 > \tilde{g}_{\bv}(q_0)$, and $\bb_{q_0,\tilde{g}_{\bv}(q_0)} = \bv$.
\end{enumerate}
\end{lemma}

\begin{proof}
Since $\tilde{f}_{i_1i_2\cdots}(q_0,q_1) = f_{(1-i_1)(1-i_2)\cdots}(q_1,q_0)$, the points (\ref{i:gt1}) and (\ref{i:gt3}) follow from Lemma~\ref{l:functiong} (\ref{i:g1}) and (\ref{i:g3}).
The proof of (\ref{i:gt5}) is similar to that of Lemma~\ref{l:functiong}~(\ref{i:g5}), in particular $\bv > \bb_{q_0,q_1}$ implies that~$\bv$ is a quasi-lazy $(q_0,q_1)$-expansion, thus $\tilde{\pi}_{q_0,q_1}(\bv) < \frac{1}{q_0}$, i.e., $\tilde{f}_{\bv}(q_0,q_1) < 0$.
\end{proof}

Next we study $g_{\bu}, \tilde{g}_{\bv}$ as functions of $\bu,\bv$.

\begin{lemma} \label{l:functiong2}
Let $q_0  > 1$.
Then the map
\begin{equation*}
g_{\cdot}(q_0):\, \{\bu \in W \,:\, q_{\bu} > q_0\} \to \big(1,\tfrac{q_0}{q_0-1}\big), \quad \bu \mapsto g_{\bu}(q_0),
\end{equation*}
is a continuous order-preserving bijection, and the map
\begin{equation*}
\tilde{g}_{\cdot}(q_0):\, \tilde{W} \to \big(1,\tfrac{q_0}{q_0-1}\big), \quad \bu \mapsto \tilde{g}_{\bu}(q_0),
\end{equation*}
is a continuous order-reversing bijection.
\end{lemma}

\begin{proof}
Let $\bu, \bu' \in W$ with $\bu > \bu'$ and $q_{\bu}, q_{\bu'} > q_0$.
By Lemma~\ref{l:functiong}~(\ref{i:g5}), we have $\bu = \ba_{q_0,g_{\bu}(q_0)}$ and thus $g_{\bu}(q_0) > g_{\bu'}(q_0)$, hence $g_{\cdot}(q_0)$ is order-preserving and thus injective.
The map $g_{\cdot}(q_0)$ is surjective because, for each $q_1 \in (1,\frac{q_0}{q_0-1})$, we have $\ba_{q_0,q_1} \in W$ and $g_{\ba_{q_0,q_1}}(q_0) = q_1$.
The continuity of $g_{\cdot}(q_0)$  follows from the monitonicity and surjectivity.
The proof for $\tilde{g}_{\cdot}(q_0)$ runs along the same lines.
\end{proof}

Recall that $\mu_{\bu,\bv}$ is defined as the unique solution $q_0 > 1$ of $g_{\bu}(q_0) = \tilde{g}_{\bv}(q_0)$, if this equation has a unique solution $>1$.
The following proposition shows that $\mu_{\bu,\bv}$ is well defined in all cases that are interesting to us.
Recall that limit words of primitive $\bsigma \in \{L,R\}^\infty$ are Sturmian sequences.

\begin{lemma} \label{l:mu}
Let $\bu, \bv \in \sigma M(\{0,1\}^\infty)$, $\sigma \in \{L,R\}^*$, such that $\bu \in W$, $\bv \in \tilde{W}$, or $\bu = \bsigma(\overline{0})$, $\bv = \bsigma(\overline{1})$ with primitive $\bsigma \in \{L,R\}^\infty$.
Then $g_{\bu}(x) = \tilde{g}_{\bv}(x)$ for a unique $x  > 1$.
We have $g_{\bu}(x) > \tilde{g}_{\bv}(x)$ for all $x \in (1, \mu_{\bu,\bv})$, $g_{\bu}(x) < \tilde{g}_{\bv}(x)$ for all $x > \mu_{\bu,\bv}$.
\end{lemma}

\begin{proof}
Note that $\bsigma(\overline{0}) \in W$, $\bsigma(\overline{1}) \in \tilde{W}$ for all primitive $\bsigma \in \{L,R\}^\infty$ by Lemma~\ref{l:infsup}.
Therefore, by Lemmas~\ref{l:functiong} and~\ref{l:functiongt}, $g_{\bu}(x)$ and $\tilde{g}_{\bv}(x)$ are continuous functions with
\begin{equation*}
\lim_{q_0\to 1} (g_{\bu}(q_0)-\tilde{g}_{\bv}(q_0)) = \infty \quad \text{and} \quad
g_{\bu}(q_0) - \tilde{g}_{\bv}(q_0) < 0 \quad \text{for all}\ q_0 \ge q_{\bu}.
\end{equation*}
Moreover, the functions are differentiable on $(1,q_{\bu})$, with
\begin{equation*}
g_{\bu}'(x) = - \frac{\frac{\partial}{\partial x}f_{\bu}(x, g_{\bu}(x))}{\frac{\partial}{\partial y}f_{\bu}(x, g_{\bu}(x))}, \quad \tilde{g}_{\bv}'(x) = - \frac{\frac{\partial}{\partial x}\tilde{f}_{\bv}(x, \tilde{g}_{\bv}(x))}{\frac{\partial}{\partial y}\tilde{f}_{\bv}(x, \tilde{g}_{\bv}(x))}.
\end{equation*}
Therefore, it suffices to show that $g_{\bu}'(x) < \tilde{g}_{\bv}'(x)$ whenever $g_{\bu}(x)=\tilde{g}_{\bv}(x)$, i.e.,
\begin{equation} \label{e:partial}
\frac{\partial}{\partial x}f_{\bu}(x,y) \frac{\partial}{\partial y}\tilde{f}_{\bv}(x,y) - \frac{\partial}{\partial y}f_{\bu}(x,y) \frac{\partial}{\partial x}\tilde{f}_{\bv}(x,y) > 0
\end{equation}
whenever $f_{\bu}(x,y) = 0 = \tilde{f}_{\bv}(x,y)$, $1 < x < q_{\bu}$, $y > 1$.

Write
\begin{equation*}
\supn(\bu) = 010^{n_1}10^{n_2-n_1}10^{n_3-n_2}\cdots \quad \text{and} \quad \infn(\bv) = 101^{\tilde{n}_1}01^{\tilde{n}_2-\tilde{n}_1}01^{\tilde{n}_3-\tilde{n}_2}\cdots
\end{equation*}
with $0 \le n_1 \le n_2 \le \cdots$, $0 \le \tilde{n}_1 \le \tilde{n}_2 \le \cdots$.
Then we have
\begin{equation*}
f_{\bu}(x,y) = 1-x + \sum_{k=1}^\infty \frac{1}{y^kx^{n_k}},
\quad \tilde{f}_{\bv}(x,y) = 1-y+ \sum_{k=1}^\infty \frac{1}{x^ky^{\tilde{n}_k}},
\end{equation*}
hence \eqref{e:partial} means that
\begin{equation} \label{e:partial2}
\bigg(1 + \sum_{k=1}^\infty \frac{n_k}{y^kx^{n_k+1}}\bigg) \bigg(1 + \sum_{\ell=1}^\infty \frac{\tilde{n}_\ell}{x^\ell y^{\tilde{n}_\ell+1}}\bigg) - \sum_{k=1}^\infty \frac{k}{y^{k+1}x^{n_k}}\, \sum_{\ell=1}^\infty \frac{\ell}{x^{\ell+1}y^{\tilde{n}_\ell}} > 0.
\end{equation}
When $f_{\bu}(x,y) = 0 = \tilde{f}_{\bv}(x,y)$, we have $\frac{1}{x-1} \sum_{k=1}^\infty \frac{1}{y^kx^{n_k}} = 1 = \frac{1}{y-1} \sum_{\ell=1}^\infty \frac{1}{x^\ell y^{\tilde{n}_\ell}}$.
Inserting this into \eqref{e:partial2} and multiplying by $xy$ gives the inequality
\begin{equation} \label{e:partial3}
\sum_{k=1}^\infty \sum_{\ell=1}^\infty\frac{\big(n_k+\frac{x}{x-1}\big) \big(\tilde{n}_\ell+\frac{y}{y-1}\big) - k\,\ell}{x^{n_k+\ell}y^{k+\tilde{n}_\ell}} > 0.
\end{equation}

Let $\mathrm{f}_0, \mathrm{f}_1$ be the frequencies of the letters $0$ and $1$ in $\bu$ and~$\bv$.
These frequencies exist for Sturmian words, and for $\sigma M(\{0,1\}^\infty)$, $\sigma \in \{L,R\}^*$,  we have $\frac{\mathrm{f}_0}{\mathrm{f}_1} = \frac{|\sigma(10)|_0}{|\sigma(10)|_1}$, where $|w|_i$ denotes the number of occurrences of the letter $i$ in~$w$.
We show that
\begin{equation} \label{e:nf}
n_k + 1 \ge \frac{\mathrm{f}_0}{\mathrm{f}_1} k \quad \text{for all}\ k \ge 1.
\end{equation}

If $\bu = \bsigma(\overline{0})$ for primitive $\bsigma \in \{L,R\}^\infty$ or $\bu = \sigma M(\overline{0})$, $\sigma \in \{L,R\}^*$, then $\bu$ is a mechanical word with slope~$\mathrm{f}_1$; see e.g.\ \cite[Sections~2.1.2 and~2.2.2]{Lot2002} and note that $M(\overline{0}) = R(\overline{0})$.
Since $\bu = \supn(\bu)$, we have $\bu = (\lceil m \mathrm{f}_1 \rceil {-} \lceil (m{-}1) \mathrm{f}_1 \rceil)_{m\ge0}$.
There are $k$ ones among the first $n_k{+}k{+}1$ letters of~$\bu$, hence $\sum_{m=0}^{n_k+k} (\lceil m \mathrm{f}_1 \rceil {-} \lceil (m{-}1) \mathrm{f}_1 \rceil) = \lceil (n_k{+}k) \mathrm{f}_1 \rceil = k$, i.e., $\lceil n_k \mathrm{f}_1 -  k \mathrm{f}_0 \rceil = 0$.
Since $n_k$ is the maximal integer with this property, we obtain that $n_k = \lfloor \frac{\mathrm{f}_0}{\mathrm{f}_1} k\rfloor$, thus \eqref{e:nf} holds.

If $\bu = i_0i_1\cdots = \sigma M(j_0j_1\cdots)$, $\sigma \in \{L,R\}^*$, then consider $k = h |\sigma(01)|_1 {+} \ell$, $h \ge 0$, $0 \le \ell <|\sigma(01)|_1$.
Now, there are $\ell$ ones among the first $n_k{+}k{+}1{-}h|\sigma(01)|$ letters of $\sigma M(j_h)$.
If $j_h = 0$, then we get$\lceil (n_k{+}k{-}h|\sigma(01)|) \mathrm{f}_1 \rceil = \ell$, i.e., $\lceil (n_k{+}k) \mathrm{f}_1 \rceil = \ell {+} h\, |\sigma(01)|_1 = k$, thus $n_k = \lfloor \frac{\mathrm{f}_0}{\mathrm{f}_1} k\rfloor$.
If $j_h = 1$, then  $\sigma M(\overline{1}) = (\lfloor m \mathrm{f}_1 \rfloor {-} \lfloor (m{-}1) \mathrm{f}_1 \rfloor)_{m \ge 0}$ implies that $\sum_{m=0}^{n_k+k-h|\sigma(01)|} (\lfloor m \mathrm{f}_1 \rfloor {-} \lfloor (m{-}1) \mathrm{f}_1 \rfloor) = \lfloor (n_k{+}k{-}h|\sigma(01)|) \mathrm{f}_1 \rceil + 1 = \ell$, i.e., $\lfloor n_k \mathrm{f}_1  - k \mathrm{f}_0 \rfloor = -1$, thus $n_k = \lceil \frac{\mathrm{f}_0}{\mathrm{f}_1} k\rceil {-} 1$.
Therefore, \eqref{e:nf} holds also in this case.

By symmetry, we get that $\tilde{n}_\ell + 1 \ge \frac{\mathrm{f}_1}{\mathrm{f}_0} \ell$ for all $\ell \ge 1$, thus
\begin{equation*}
\big(n_k+\tfrac{x}{x-1}\big) \big(\tilde{n}_\ell+\tfrac{y}{y-1}\big) > (n_k+1) (\tilde{n}_\ell+1) \ge k\, \ell.
\end{equation*}
This proves \eqref{e:partial3}, thus \eqref{e:partial2} and \eqref{e:partial}, which concludes the proof of the lemma.
\end{proof}

Before proving the main results, we show that the formulas for $\mathcal{G}(q_0)$ in \eqref{e:ggr} and for $\mathcal{K}(q_0)$ in \eqref{e:klc} cover all $q_0 > 1$.

\begin{lemma} \label{l:cover}
We have the partitions
\begin{align}
(1, \infty) & = \bigcup_{\sigma\in\{L,R\}^*M} [\mu_{\sigma(\overline{0}),\sigma(1\overline{0})}, \mu_{\sigma(0\overline{1}),\sigma(\overline{1})}] \ \cup \hspace{-1em} \bigcup_{\bsigma\in\{L,R\}^\infty\,\text{primitive}} \hspace{-1em} \{\mu_{\bsigma(\overline{0}),\bsigma(\overline{1})}\}, \label{e:coverLR} \\
(1, \infty) & = \bigcup_{\substack{\bsigma\in \{L,M,R\}^*M\{\overline{L},\overline{R}\}\,\text{or} \\ \bsigma\in\{L,M,R\}^\infty\,\text{primitive}}} [\mu_{\bsigma(\overline{0}),\bsigma(1\overline{0})}, \mu_{\bsigma(0\overline{1}),\bsigma(\overline{1})}], \label{e:coverLMR}
\end{align}
with $\mu_{\bsigma(0\overline{1}),\bsigma(\overline{1})} < \mu_{\btau(\overline{0}),\btau(1\overline{0})}$ if $\bsigma < \btau$.
\end{lemma}

\begin{proof}
For $q_0 > 1$, $\bsigma \in \{L,M,R\}^*M\{\overline{L},\overline{R}\}$ or primitive $\bsigma \in \{L,M,R\}^\infty$, we have
\begin{equation*}
\begin{aligned}
& q_0 \in [\mu_{\bsigma(\overline{0}),\bsigma(1\overline{0})}, \mu_{\bsigma(0\overline{1}),\bsigma(\overline{1})}] \ \Longleftrightarrow \ g_{\bsigma(\overline{0})}(q_0) \le \tilde{g}_{\bsigma(1\overline{0})}(q_0),\ g_{\bsigma(0\overline{1})}(q_0) \ge \tilde{g}_{\bsigma(\overline{1})}(q_0) \\
& \qquad \Longleftrightarrow \ [g_{\bsigma(\overline{0})}(q_0), g_{\bsigma(0\overline{1})}(q_0)] \cap [\tilde{g}_{\bsigma(\overline{1})}(q_0), \tilde{g}_{\bsigma(1\overline{0})}(q_0)] \ne \emptyset.
\end{aligned}
\end{equation*}
by Lemma~\ref{l:mu} and since $\tilde{g}_{\bsigma(\overline{1})}(q_0) > 1$.
Lemmas~\ref{l:partition} and~\ref{l:functiong2} give the partitions
\begin{equation*}
\big(1,\tfrac{q_0}{q_0-1}\big) \ = \hspace{-3em} \bigcup_{\substack{\bsigma\in \{L,M,R\}^*M\{\overline{L},\overline{R}\}\,\text{or} \\ \bsigma\in\{L,M,R\}^\infty\,\text{primitive}}} \hspace{-3em} \big([g_{\bsigma(\overline{0})}(q_0), g_{\bsigma(0\overline{1})}(q_0)] \setminus \{1\}\big) \ = \hspace{-3em} \bigcup_{\substack{\bsigma\in \{L,M,R\}^*M\{\overline{L},\overline{R}\}\,\text{or} \\ \bsigma\in\{L,M,R\}^\infty\,\text{primitive}}} \hspace{-3em} [\tilde{g}_{\bsigma(\overline{1})}(q_0), \tilde{g}_{\bsigma(1\overline{0})}(q_0)],
\end{equation*}
with $g_{\bsigma(0\overline{1})}(q_0) < g_{\btau(\overline{0})}(q_0)$ (if $g_{\btau(\overline{0})}(q_0) > 1$) and $\tilde{g}_{\bsigma(\overline{1})}(q_0) > \tilde{g}_{\btau(1\overline{0})}(q_0)$ when $\bsigma < \btau$.
Therefore, there is a unique $\bsigma$ such that the intervals $[g_{\bsigma(\overline{0})}(q_0), g_{\bsigma(0\overline{1})}(q_0)]$ and $[\tilde{g}_{\bsigma(\overline{1})}(q_0), \tilde{g}_{\bsigma(1\overline{0})}(q_0)]$ overlap, hence \eqref{e:coverLMR} is a partition.
If $\bsigma < \btau$, then we have
\begin{equation*}
g_{\bsigma(0\overline{1})}(\mu_{\btau(\overline{0}),\btau(1\overline{0})}) < g_{\btau(\overline{0})}(\mu_{\btau(\overline{0}),\btau(1\overline{0})}) = \tilde{g}_{\btau(1\overline{0})}(\mu_{\btau(\overline{0}),\btau(1\overline{0})}) < \tilde{g}_{\bsigma(\overline{1})}(\mu_{\btau(\overline{0}),\btau(1\overline{0})})
\end{equation*}
by Lemma~\ref{l:functiong2}, thus $\mu_{\bsigma(0\overline{1}),\bsigma(\overline{1})} < \mu_{\btau(\overline{0}),\btau(1\overline{0})}$ by Lemma~\ref{l:mu}.

To see that \eqref{e:coverLR} is a partition, it suffices to note that merging all intervals $[\mu_{\bsigma(\overline{0}),\bsigma(1\overline{0})}, \mu_{\bsigma(0\overline{1}),\bsigma(\overline{1})}]$ in \eqref{e:coverLMR} such that $\bsigma$ starts with the same $\sigma \in \{L,R\}^*M$ gives the interval $[\mu_{\sigma\overline{L}(\overline{0}),\sigma\overline{L}(1\overline{0})}, \mu_{\sigma\overline{R}(0\overline{1}),\sigma\overline{R}(\overline{1})}] = [\mu_{\sigma(\overline{0}),\sigma(1\overline{0})}, \mu_{\sigma(0\overline{1}),\sigma(\overline{1})}]$.
\end{proof}

\section{Proof of the main results} \label{sec:proof-main-results}
In this section, we prove Theorems \ref{t:1}--\ref{t:3}.
We start with Theorem~\ref{t:2}.

\begin{proposition} \label{p:K}
For $q_0 > 1$, we have
\begin{equation} \label{e:Kq0}
\mathcal{K}(q_0) = \begin{cases}\tilde{g}_{\bsigma(\overline{1})}(q_0) & \text{if}\ q_0 \in [\mu_{\bsigma(\overline{0}),\bsigma(\overline{1})}, \mu_{\bsigma(0\overline{1}),\bsigma(\overline{1})}], \\ & \text{$\bsigma \in \{L,M,R\}^*M\overline{L}$ or $\bsigma \in \{L,M,R\}^\infty$ primitive}, \\ g_{\bsigma(\overline{0})}(q_0) & \text{if}\ q_0 \in [\mu_{\bsigma(\overline{0}),\bsigma(1\overline{0})}, \mu_{\bsigma(\overline{0}),\bsigma(\overline{1})}], \\ & \text{$\bsigma \in \{L,M,R\}^*M\overline{R}$ or $\bsigma \in \{L,M,R\}^\infty$ primitive},\end{cases}
\end{equation}
which is equivalent to~\eqref{e:klc},
Moreover, $\mathcal{K}(q_0)$ is the unique $q_1\in (1, \frac{q_0}{q_0-1}\big)$ such that $s(\ba_{q_0,q_1}) = s(\bb_{q_0, q_1})$.
\end{proposition}

\begin{proof}
We show first that $s(\ba_{q_0,q_1}) = s(\bb_{q_0, q_1})$, $q_0 > 1$, $q_1\in (1, \frac{q_0}{q_0-1}\big)$, implies that $q_1 = \mathcal{K}(q_0)$.
Indeed, let $\bsigma = s(\ba_{q_0,q_1}) = s(\bb_{q_0, q_1})$.
For $q'_1 \in \big(q_1,\frac{q_0}{q_0-1}\big)$, we have $\ba_{q_0, q_1'} > \ba_{q_0, q_1}$ and $\bb_{q_0, q_1'} < \bb_{q_0, q_1}$ by Lemma~\ref{l:ab}, thus \mbox{$s(\bb_{q_0,q'_1}) \le \bsigma \le s(\ba_{q_0,q'_1})$}.
At least one of the inequalities is strict since $\bsigma(\overline{0}) \,{=}\, \bsigma(0\overline{1})$ if $\bsigma$ is primitive or ends with~$\overline{R}$, $\bsigma(1\overline{0}) \,{=}\, \bsigma(\overline{1})$ if $\bsigma$ is primitive or ends with~$\overline{L}$, hence $U_{q_0,q'_1}$ is uncountable by Lemma~\ref{l:UV} (iii) and Proposition~\ref{p:lex}.
Similarly, we obtain for $q'_1 \in (1,q_1)$ that $s(\ba_{q_0,q'_1}) < s(\bb_{q_0,q'_1})$ and $U_{q_0,q'_1}$ is countable.
This proves that $q_1 = \mathcal{K}(q_0)$.
In particular, there is at most one $q_1\in (1, \frac{q_0}{q_0-1}\big)$ with $s(\ba_{q_0,q_1}) = s(\bb_{q_0, q_1})$.

Next we show that $s(\ba_{q_0,q_1}) = s(\bb_{q_0, q_1})$ for $q_1$ as in \eqref{e:Kq0}.
Let $\bsigma \in \{L,M,R\}^*M\overline{R}$ or $\bsigma \in \{L,M,R\}^\infty$ primitive, $q_0 \in [\mu_{\bsigma(\overline{0}),\bsigma(1\overline{0})}, \mu_{\bsigma(\overline{0}),\bsigma(\overline{1})}]$, $q_1 = g_{\bsigma(\overline{0})}(q_0)$.
Then $\ba_{q_0, q_1} = \bsigma(\overline{0})$ by  Lemma~\ref{l:functiong}, hence $s(\ba_{q_0, q_1}) = \bsigma$.
Since $q_0 \ge \mu_{\bsigma(\overline{0}),\bsigma(1\overline{0})}$, we have $q_1 \le \tilde g_{\bsigma(1\overline{0})}(q_0)$ by Lemma~\ref{l:mu}, hence $\bb_{q_0,q_1} \ge \bsigma(1\overline{0})$ by Lemmas~\ref{l:ab} and~\ref{l:functiongt}.
Similarly, $q_0 \le \mu_{\bsigma(\overline{0}),\bsigma(\overline{1})}$ implies that $\bb_{q_0,q_1} \le \bsigma(\overline{1})$, thus $s(\bb_{q_0,q_1}) = \bsigma$.
By~symmetry, we obtain that $s(\ba_{q_0,q_1}) = s(\bb_{q_0, q_1})$ for $q_0 \in [\mu_{\bsigma(\overline{0}),\bsigma(\overline{1})}, \mu_{\bsigma(0\overline{1}),\bsigma(\overline{1})}]$, $q_1 = \tilde{g}_{\bsigma(\overline{1})}(q_0)$, $\bsigma \in \{L,M,R\}^*M\overline{L}$ or $\bsigma \in \{L,M,R\}^\infty$ primitive.
Therefore, \eqref{e:Kq0} holds.

By Lemma~\ref{l:cover}, the cases in \eqref{e:Kq0} cover all $q_0 > 1$, thus $\mathcal{K}(q_0)$ is the unique $q_1\in (1, \frac{q_0}{q_0-1}\big)$ such that $s(\ba_{q_0,q_1}) = s(\bb_{q_0, q_1})$.
Using \eqref{e:LR} and that $\bsigma(\overline{0}) = \bsigma(0\overline{1})$, $\bsigma(\overline{1}) = \bsigma(1\overline{0})$ for primitive~$\bsigma$, we obtain that \eqref{e:Kq0} is equivalent to~\eqref{e:klc}.
\end{proof}

\begin{proposition} \label{p:G}
For $q_0 > 1$, $\mathcal{G}(q_0)$ is the unique $q_1 \in (1, \frac{q_0}{q_0-1})$ such that
\begin{equation} \label{e:Gchar}
\begin{aligned}
& \ba_{q_0,q_1} \,{=}\, \sigma(\overline{0}), \sigma(1\overline{0}) \,{\le}\, \bb_{q_0,q_1} \,{\le}\, \sigma(\overline{1}), \, \text{or} \ \sigma(\overline{0}) \,{\le}\, \ba_{q_0,q_1} \,{\le}\, \sigma(1\overline{0}), \bb_{q_0,q_1} \,{=}\, \sigma(\overline{1}), \\
& \sigma \in \{L,R\}^*M, \, \text{or} \ \ba_{q_0,q_1} \,{=}\, \bsigma(\overline{0}), \bb_{q_0,q_1} \,{=}\, \bsigma(\overline{1}),\, \bsigma \in \{L,R\}^\infty\ \text{primitive}.
\end{aligned}
\end{equation}
Moreover, \eqref{e:ggr} holds.
\end{proposition}

\begin{proof}
If $q_0 > 1$, $q_1\in (1, \frac{q_0}{q_0-1}\big)$ satisfy \eqref{e:Gchar}, then $V_{q_0,q_1} \ne \{\overline{0}, \overline{1}\}$ by Theorem~\ref{t:lex}, thus $U_{q_0,q'_1} \ne \{\overline{0}, \overline{1}\}$ for all $q'_1 \in (q_1,\frac{q_0}{q_0-1})$ by Lemma~\ref{l:ab}; for all $q'_1 \in (1,q_1)$, we have, by Lemma~\ref{l:ab}, $\ba_{q_0,q'_1} < \sigma(\overline{0})$ or $\bb_{q_0,q'_1} > \sigma(\overline{1})$, or $\ba_{q_0,q'_1} < \bsigma(\overline{0})$, $\bb_{q_0,q'_1} > \bsigma(\overline{1})$, thus $V_{q_0,q_1} = \{\overline{0}, \overline{1}\}$ by Theorem~\ref{t:lex}.
(In case of primitive $\bsigma$, we use that $s(\ba_{q_0,q'_1}) < \bsigma$, thus $\ba_{q_0,q'_1} \le \sigma_1 \cdots \sigma_n M(0\overline{1})$ and $\bb_{q_0,q'_1} > \sigma_1 \cdots \sigma_n M(\overline{1})$, where $n$ is the largest integer such that $s(\ba_{q_0,q'_1})$ starts with $\sigma_1,\dots,\sigma_n$.)
This means that $q_1 = \mathcal{G}(q_0)$.

Similarly to the proof of Proposition~\ref{p:K}, each $q_1$ as in \eqref{e:ggr} satisfies \eqref{e:Gchar}, thus \eqref{e:ggr} holds.
Since the cases in \eqref{e:ggr} cover all $q_0 > 1$ by Lemma~\ref{l:cover}, $\mathcal{G}(q_0)$ is the unique $q_1\in (1, \frac{q_0}{q_0-1}\big)$ satisfying \eqref{e:Gchar}.
\end{proof}

Next, we prove statement (\ref{i:t11}) of Theorem~\ref{t:1}.

\begin{proposition} \label{p:cont}
The functions $\mathcal{G}$ and $\mathcal{K}$ are continuous, strictly decreasing on $(1,\infty)$, and almost everywhere differentiable.
\end{proposition}

\begin{proof}
For all $\bsigma \in \{L,M,R\}^*M\{\overline{L},\overline{R}\}$, the function~$\mathcal{K}$ is continuous, strictly decreasing and differentiable on $[\mu_{\bsigma(\overline{0}),\bsigma(1\overline{0})}, \mu_{\bsigma(0\overline{1}),\bsigma(\overline{1})}]$ by Proposition~\ref{p:K}, Lemmas~\ref{l:functiong} and~\ref{l:functiongt}, and the proof of Lemma~\ref{l:mu};
here, the properties are to be understood one-sided at the endpoints of the interval.
For left-sided properties at $\mu_{\bsigma(\overline{0}),\bsigma(1\overline{0})}$, $\bsigma \in \{L,M,R\}^*M\{\overline{L},\overline{R}\}$ or $\bsigma \in \{L,M,R\}^\infty$ primitive, consider $q_0 \in [\mu_{\btau(\overline{0}),\btau(1\overline{0})}, \mu_{\btau(0\overline{1}),\btau(\overline{1})}]$ with $\btau < \bsigma$.
We have
\begin{equation*}
\tilde{g}_{\bsigma(1\overline{0})}(q_0) < \tilde{g}_{\btau(\overline{1})}(q_0) \le g_{\btau(0\overline{1})}(q_0) < g_{\bsigma(\overline{0})}(q_0)
\end{equation*}
by Lemmas~\ref{l:functiong2} and~\ref{l:mu}, $\mathcal{K}(q_0) \in \{g_{\btau(0\overline{1})}(q_0), \tilde{g}_{\btau(\overline{1})}(q_0)\}$ by Proposition~\ref{p:K}, thus
\begin{equation*}
\tilde{g}_{\bsigma(1\overline{0})}(q_0) < \mathcal{K}(q_0) < g_{\bsigma(\overline{0})}(q_0) \quad \text{for all}\ q_0 \in (1, \mu_{\bsigma(\overline{0}),\bsigma(1\overline{0})}).
\end{equation*}
Since $\tilde{g}_{\bsigma(1\overline{0})}(\mu_{\bsigma(\overline{0}),\bsigma(1\overline{0})}) =  g_{\bsigma(\overline{0})}(\mu_{\bsigma(\overline{0}),\bsigma(1\overline{0})})$ and $\tilde{g}_{\bsigma(1\overline{0})}$, $g_{\bsigma(\overline{0})}$ are continuous and strictly decreasing, $\mathcal{K}$ is left-sided continuous and strictly decreasing at $\mu_{\bsigma(\overline{0}),\bsigma(1\overline{0})}$.
Symmetrically, we obtain that
\begin{equation*}
g_{\bsigma(0\overline{1})}(q_0) < \mathcal{K}(q_0) < \tilde{g}_{\bsigma(\overline{1})}(q_0) \quad \text{for all}\ q_0 > \mu_{\bsigma(0\overline{1}),\bsigma(\overline{1})},
\end{equation*}
thus $\mathcal{K}$ is right-sided continuous and strictly decreasing at $\mu_{\bsigma(0\overline{1}),\bsigma(\overline{1})}$.
Therefore, $\mathcal{K}$ is continuous and strictly decreasing on $(1,\infty)$.
The almost everywhere differentiability follows from the monotonicity by a theorem of Lebesgue; see e.g.\ \cite[p.~5]{RieSz1955}.

The proof for the function $\mathcal{G}$ runs along the same lines.
\end{proof}

\begin{proposition} \label{p:GK}
The statements (\ref{i:t12})--(\ref{i:14_2}) of Theorem~\ref{t:1} and the statement (\ref{i:t31}) of Theorem~\ref{t:3} are true.
\end{proposition}

\begin{proof}
The functions $\mathcal{G}$ and $\mathcal{K}$ are involutions because of the bijection between $U_{q_0,q_1}$ and $U_{q_1,q_0}$ given in Lemma~\ref{l:01}.

\smallskip
To get an upper bound for~$\mathcal{G}$, let $q_0 \in [\mu_{\sigma(\overline{0}),\sigma(1\overline{0})}, \mu_{\sigma(\overline{0}),\sigma(\overline{1})}]$, $\sigma \in \{L,R\}^*M$.
Then $\sigma(0) = 01w$ and $\sigma(1) = 10w$ for some word~$w$. (This is true for $\sigma = M$; if it holds for $\sigma$, then it also holds for $L\sigma$ and~$R\sigma$.)
This implies that $\sigma(\overline{0}) = 01\bu$ and $\sigma(1\overline{0}) = 10\bu$, with $\bu = \overline{w01}$.
For $q_1 = \mathcal{G}(q_0) = g_{\sigma(\overline{0})}(q_0)$, we get that
\begin{equation} \label{e:u1}
q_0 = q_0 q_1 \pi_{q_0,q_1}(\sigma(\overline{0})) = 1 + \pi_{q_0,q_1}(\bu).
\end{equation}
Since $q_1 = g_{\sigma(\overline{0})}(q_0) \le \tilde{g}_{\sigma(1\overline{0})}(q_0)$, we have $\tilde{f}_{\sigma(1\overline{0})}(q_0,q_1) \ge 0$ by Lemma~\ref{l:functiongt}~(\ref{i:gt1}), thus
\begin{equation} \label{e:u2}
q_1 \le q_0 q_1\, \tilde{\pi}_{q_0,q_1}(\sigma(1\overline{0})) = 1 + \tilde{\pi}_{q_0,q_1}(\bu).
\end{equation}
By \eqref{e:u1}, \eqref{e:u2}, and \eqref{e:q1q0}, we obtain that
\begin{equation*}
2\, (q_0-1) (q_1-1) \le (q_1-1)\, \pi_{q_0,q_1}(\bu) + (q_0-1)\, \tilde{\pi}_{q_0,q_1}(\bu) = 1,
\end{equation*}
with equality if and only if $q_1 = \tilde{g}_{\sigma(1\overline{0})}(q_0)$, i.e., $q_0 = \mu_{\sigma(\overline{0}),\sigma(1\overline{0})}$.
The case $q_0 \in [\mu_{\sigma(\overline{0}),\sigma(\overline{1})}, \mu_{\sigma(0\overline{1}),\sigma(\overline{1})}]$ is symmetric, with $(q_0{-}1) (\mathcal{G}(q_0){-}1) = \frac{1}{2}$ if and only if $q_0 = \mu_{\sigma(0\overline{1}),\sigma(\overline{1})}$.
If $q_0 \,{=}\, \mu_{\bsigma(\overline{0}),\bsigma(\overline{1})}$ for a primitive $\bsigma \,{\in}\, \{L,R\}^\infty$, then $(q_0{-}1) (\mathcal{G}(q_0){-}1) \,{=}\, \frac{1}{2}$ by continuity or by using that $\bsigma(\overline{0}) = 01\bu$, $\bsigma(\overline{1}) = 10\bu$ for some $\bu \in \{0,1\}^\infty$.

\smallskip
Next, we prove lower bounds for~$\mathcal{G}$.
For $q_0 = \mu_{\bsigma(\overline{0}),\bsigma(\overline{1})}$, primitive $\bsigma \in \{L,R\}^\infty$, we have $(q_0{-}1) (\mathcal{G}(q_0){-}1) \,{=}\, \frac{1}{2} \,{>}\, \max\{\frac{1}{q_0+1}, \frac{1}{\mathcal{G}(q_0)+1}\}$.
If $q_0 \,{\in}\, [\mu_{\sigma(\overline{0}),\sigma(1\overline{0})}, \mu_{\sigma(0\overline{1}),\sigma(\overline{1})}]$, $\sigma \in \{L,R\}^*M$, $q_1 \,{=}\, \mathcal{G}(q_0)$, then we have $q_1 \,{\ge}\, g_{\sigma(\overline{0})}(q_0)$ and $q_1 \,{\ge}\, \tilde{g}_{\sigma(\overline{1})}(q_0)$, thus
\begin{equation*}
q_0 \ge q_0 q_1 \pi_{q_0,q_1}(\sigma(\overline{0})) = 1 + \pi_{q_0,q_1}(\overline{w01}), \quad q_1 \ge q_0 q_1\, \tilde{\pi}_{q_0,q_1}(\sigma(\overline{1})) = 1 + \tilde{\pi}_{q_0,q_1}(\overline{w10})
\end{equation*}
for some~$w$.
We have $\overline{w01} \ge \overline{0w1}$, thus $\pi_{q_0,q_1}(\overline{w01}) \ge \pi_{q_0,q_1}(\overline{0w1})$ by Lemma~\ref{l:qgql}; note that $\overline{w01}$ is a quasi-greedy $(q_0,q_1)$-expansion because $\sigma(\overline{0})$ is quasi-greedy, that $\sigma(\overline{0})$ ends with $\sigma(\overline{1})$ because $M(\overline{0}) = \overline{01}$, $M(\overline{1}) = \overline{10}$, thus $\overline{0w1}$ is also quasi-greedy.
Hence, we have $q_0 {-} 1 \ge \pi_{q_0,q_1}(\overline{0w1}) = q_0\, \pi_{q_0,q_1}(\overline{w10})$, thus $(q_0{+}1) (q_0{-}1) (q_1{-}1) \ge 1$.
Now, equality holds if and only if $q_0 = \mu_{\sigma(\overline{0}),\sigma(\overline{1})}$ and $w01 = 0w1$, i.e., $w = 0^k$ for some $k \ge 0$, which means that $\sigma = L^kM$.
By Example~\ref{ex:Lk}, we have $q_0 = \mu_{L^kM(\overline{0}),L^kM(\overline{1})}$ if and only if $q_0^{k+2} = q_0+1$.
Since $\mathcal{G}(q_1) = q_0$ for $q_1 = \mathcal{G}(q_0)$, we also have $(q_1{+}1) (q_0{-}1) (q_1{-}1) \ge 1$, with equality if and only if $q_1^{k+2} = q_1+1$, $k \ge 0$; note that $(q_0{-}1)(q_1^2{-}1) = 1$ means that $q_0 = \frac{q_1^2}{q_1^2-1}$.

\smallskip
For a lower bound on~$\mathcal{K}$, let $q_0 \in [\mu_{\bsigma(\overline{0}),\bsigma(1\overline{0})}, \mu_{\bsigma(\overline{0}),\bsigma(\overline{1})}]$, $\bsigma \in \{L,M,R\}^*M\overline{R}$ or $\bsigma \in \{L,M,R\}^\infty$ primitive, $q_1 = \mathcal{K}(q_0) = g_{\bsigma(\overline{0})}(q_0)$.
For $\sigma \in \{L,M,R\}^* M$, we have $\sigma(0) = 01v$, $\sigma(1) = 10w$ for words $v \ge w$, with $v = w$ if and only if $\sigma \in \{L,R\}^* M$.
This implies that $\bsigma(\overline{0}) = 01\bu$, $\bsigma(\overline{1}) = 10\bv$ with $\bu \ge \bv$, and $\bu = \bv$ if and only if $\bsigma \in \{L,R\}^*M\overline{R}$ or $\bsigma \in \{L,R\}^\infty$.
From $q_1 = g_{\bsigma(\overline{0})}(q_0) \ge \tilde{g}_{\sigma(1\overline{0})}(q_0)$, we get that
\begin{equation*}
q_0 = q_0 q_1 \pi_{q_0,q_1}(\bsigma(\overline{0})) = 1 + \pi_{q_0,q_1}(\bu), \quad q_1 \ge q_0 q_1\, \tilde{\pi}_{q_0,q_1}(\bsigma(\overline{1})) = 1 + \tilde{\pi}_{q_0,q_1}(\bv).
\end{equation*}
Since $\bsigma(\overline{0})$ is a quasi-greedy $(q_0,q_1)$-expansion and $\bsigma(\overline{1}) \in X_{\bsigma(\overline{0})}$, $\bsigma(\overline{1})$ is also quasi-greedy, thus $\pi_{q_0,q_1}(\bu) \ge \pi_{q_0,q_1}(\bv)$ by Lemma~\ref{l:qgql}.
This implies $2 (q_0{-}1) (q_1{-}1) \ge 1$, with equality if and only if $q_0 = \mu_{\bsigma(\overline{0}),\bsigma(\overline{1})}$, $\bsigma \in \{L,R\}^*M\overline{R}$ or $\bsigma \in \{L,R\}^\infty$.
By symmetry, we have $(q_0{-}1) (q_1{-}1) \ge \frac{1}{2}$ for all $q_0 \in [\mu_{\sigma(\overline{0}),\sigma(1\overline{0})}, \mu_{\sigma(01\overline{0}),\sigma(1\overline{0})}]$, $\sigma \in \{L,M,R\}^*M$, with equality if and only if $q_0 = \mu_{\sigma(\overline{0}),\sigma(1\overline{0})}$, $\sigma \in \{L,R\}^*M$.

\smallskip
We have shown that $(q_0{-}1) (\mathcal{G}(q_0){-}1) = \frac{1}{2}$ if and only if $(q_0{-}1) (\mathcal{K}(q_0){-}1) = \frac{1}{2}$, and $(q_0{-}1) (\mathcal{G}(q_0){-}1) < \frac{1}{2} < (q_0{-}1) (\mathcal{K}(q_0){-}1)$ otherwise, thus $\mathcal{G}(q_0) = \mathcal{K}(q_0)$ if and only if $\mathcal{G}(q_0)$ or $\mathcal{K}(q_0)$ equals $\frac{2q_0-1}{2(q_0-1)}$.

\smallskip
The upper bounds for~$\mathcal{K}$ are proved similarly to the lower bounds for~$\mathcal{G}$.
Let $q_1 = \mathcal{K}(q_0)$.
If $(q_0{-}1) (q_1{-}1) = \frac{1}{2}$, then $(q_0{-}1) (q_1{-}1) < \min\{\frac{q_0}{q_0+1}, \frac{q_1}{q_1+1}\}$.
If $q_0 \in [\mu_{\sigma(\overline{0}),\sigma(1\overline{0})}, \mu_{\sigma(0\overline{1}),\sigma(\overline{1})}]$, $\sigma \in \{L,R\}^*M$, then $s(\ba_{q_0,q_1}) = s(\bb_{q_0,q_1})$ starts with~$\sigma$, thus $q_1 \le g_{\sigma(0\overline{1})}(q_0)$ and $q_1 \le \tilde{g}_{\sigma(1\overline{0})}(q_0)$, with at least one of the inequalities being strict.
Since $\sigma(0) = 01w$ and $\sigma(1) = 10w$ for some~$w$, we obtain that
\begin{equation} \label{e:u3}
\hspace{-.4em} q_0 \,{\le}\, q_0 q_1 \pi_{q_0,q_1}(\sigma(0\overline{1})) \,{=}\, 1 {+} \pi_{q_0,q_1}(\overline{w10}),  q_1 \,{\le}\, q_0 q_1 \tilde{\pi}_{q_0,q_1}(\sigma(1\overline{0})) \,{=}\, 1 {+} \tilde{\pi}_{q_0,q_1}(\overline{w01}). \hspace{-.6em}
\end{equation}
We have $\overline{w10} \le \overline{1w0}$, and the quasi-greedy $(q_0,q_1)$-expansion $\sigma(\overline{0})$ ends with both words, thus $\pi_{q_0,q_1}(\overline{w10}) \le \pi_{q_0,q_1}(\overline{1w0})$.
Since $\tilde{\pi}_{q_0,q_1}(\overline{w01}) = q_1\,\tilde{\pi}_{q_0,q_1}(\overline{1w0})$ and one of the inequalities in \eqref{e:u3} is strict, \eqref{e:q1q0} gives $(q_1{+}1) (q_0{-}1)(q_1{-}1) < q_1$.
Since $\mathcal{K}(q_1) = q_0$ for $q_1 = \mathcal{K}(q_0)$, we also have $(q_0{+}1) (q_0{-}1)(q_1{-}1) < q_0$.

To show that the Hausdorff dimension of $E := \{q_0 > 1 \,:\, \mathcal{G}(q_0) = \mathcal{K}(q_0)\}$ is zero, we proceed similarly to \cite[Theorem~3]{BakSte2017}.
We have already shown that
\begin{equation*}
q_0 \in E\ \Longleftrightarrow \ \hat{\ba}_{q_0} := \ba_{q_0,1+\frac{1}{2(q_0-1)}} \in \hspace{-1em} \bigcup_{\sigma\in\{L,R\}^*M} \hspace{-1em} \{\sigma(\overline{0}), \sigma(0\overline{1})\} \ \cup \ \hspace{-2em} \bigcup_{\bsigma\in\{L,R\}^\infty\,\text{primitive}} \hspace{-2em} \{\bsigma(\overline{0})\}.
\end{equation*}
Since these words are of the form $01\bu$ with mechanical words~$\bu$, the number of different $i_1 \cdots i_n \in \{0,1\}^n$ such that $\hat{\ba}_{q_0}$ starts with $01i_1\cdots i_n$ for some $q_0 \in E$ grows polynomially in~$n$; see e.g.\ \cite[Theorem~2.2.36]{Lot2002}.
We show that the size of the interval of numbers $q_0 \in E$ such that $\hat{\ba}_{q_0}$ starts with a given word $01i_1\cdots i_n$ decreases exponentially in~$n$.
However, contrary to \cite{BakSte2017}, this holds only locally. 

Since $T_{q_{i_n},i_n} \cdots T_{q_{i_1},,i_1}(q_0{-}1) \in [0,q_0/q_1]$ if $\ba_{q_0,q_1}$ starts with $01i_1\cdots i_n$, with $T_{q,d}$ as in Section~\ref{s:lex}, we estimate $T_{q'_{i_n},i_n} \cdots T_{q'_{i_1},i_1}(q'_0{-}1) - T_{q_{i_n},i_n} \cdots T_{q_{i_1},i_1}(q_0{-}1)$ for $q_0 < q'_0$, $q_1 = 1{+}\frac{1}{2(q_0-1)}$, $q'_1 = 1{+}\frac{1}{2(q'_0-1)}$. 
Since $q_1 > q'_1$, this is more difficult than in the single base case. 
We have 
\begin{equation*} 
T_{q'_1,1} T_{q'_0,0}^k(y) - T_{q_1,1} T_{q_0,0}^k(x) = q'_0{\!}^k q'_1 y - q_0^k q_1 x = q'_0{\!}^k q'_1\, (y{-}x) + (q'_0{\!}^k q'_1 {-} q_0^k q_1)\, x.
\end{equation*}
The derivative of the function $q_0 \mapsto q_0^k (1{+}\frac{1}{2(q_0-1)})$ is $\frac{q_0^k}{q_0{-}1} (k{-}\frac{k}{2q_0}{-}\frac{1}{2(q_0{-}1)})$, hence $q'_0{\!}^k q'_1 {-} q_0^k q_1 = \frac{q''_0{}^k}{q''_0{-}1} (k{-}\frac{k}{2q''_0}{-}\frac{1}{2(q''_0{-}1)}) (q'_0{-}q_0)$ for some $q''_0 \in [q_0,q'_0]$ by the mean value theorem.
For $0 \le x \le q_0{-}1$, we obtain that
\begin{equation*}
\begin{aligned}
q'_0{\!}^k q'_1 y - q_0^k q_1 x & = q'_0{\!}^k q'_1 (y-x) + \frac{q''_0{}^k}{q''_0-1} \Big(k-\frac{k}{2q''_0}-\frac{1}{2(q''_0-1)}\Big)\, (q'_0-q_0)\, x \\
& \ge q'_0{\!}^k q'_1 (y-x) + \min\Big\{0, q'_0{\!}^k \Big(k - \frac{k}{2q_0} - \frac{1}{2(q_0-1)}\Big) (q'_0-q_0)\Big\}.
\end{aligned}
\end{equation*}

Let now $q_0, q'_0 \in [\mu_{L^kM(0\overline{1}),L^kM(\overline{1})}, \mu_{L^{k-1}M(\overline{0}),L^{k-1}M(1\overline{0})}] \cap E$, $k \ge 1$.
Then $\hat{\ba}_{q_0}$ and~$\hat{\ba}_{q'_0}$ are images of~ $L^kR$ and thus in $01\{0^k1,0^{k+1}1\}^\infty$. 
By Example~\ref{ex:Lk}, we have $\mu_{L^kM(0\overline{1}),L^kM(\overline{1})} = 2^{1/(k+1)} \ge 1{+}\frac{2}{3k+2}$ and $\mu_{L^{k-1}M(\overline{0}),L^{k-1}M(1\overline{0})}$ is a root of $2X^k{-}X^{k-1}{-}2$, thus $\mu_{L^{k-1}M(\overline{0}),L^{k-1}M(1\overline{0})} \le 1{+}\frac{3}{4k+2}$. 
For $1{+}\frac{2}{3k+2} \le q_0 \le q'_0 \le 1{+}\frac{3}{4k+2}$, we have
\begin{equation*}
q'_1 + k -\frac{k}{2q_0} - \frac{1}{2(q_0{-}1)} \ge 1 + \frac{2k{+}1}{3} + k - \frac{k\,(3k{+}2)}{2\,(3k{+}4)} - \frac{3k{+}2}{4} = \frac{(5k{+}4)(3k{+}10)}{12\,(3k{+}4)} > 1. 
\end{equation*}
For $0 \le x \le q_0-1$, $y-x \ge q'_0-q_0$, this implies that
\begin{equation*}
T_{q'_1,1} T_{q'_0,0}^k(y) - T_{q_1,1} T_{q_0,0}^k(x) \ge q'_0{\!}^k \min\Big\{q'_1, \frac{(5k{+}4)(3k{+}10)}{12\,(3k{+}4)}\Big\} (y-x),
\end{equation*}
and, similarly, 
\begin{equation*}
T_{q'_1,1} T_{q'_0,0}^{k+1}(y) - T_{q_1,1} T_{q_0,0}^{k+1}(x) \ge q'_0{\!}^{k+1} \min\Big\{q'_1, \frac{15k^2{+}80k{+}76}{12\,(3k+4)}\Big\} (y-x). 
\end{equation*}
Therefore, we have some $\beta_k>1$, $C_k > 0$ (depending only on~$k$) such that
\begin{equation*}
T_{q'_{i_n},i_n} \cdots T_{q'_{i_1},i_1}(q'_0{-}1) - T_{q_{i_n},i_n} \cdots T_{q_{i_1},i_1}(q_0{-}1) \ge C_k \beta_k^n\, (q'_0-q_0)
\end{equation*}
for all $i_1 \cdots i_n$ at the beginning of some word in $\{0^k1,0^{k+1}1\}^\infty$.
For $\hat{\ba}_{q_0}, \hat{\ba}_{q'_0}$ starting with $i_1 \cdots i_n$, we have $T_{q'_{i_n},i_n} \cdots T_{q'_{i_1},i_1}(q'_0{-}1) - T_{q_{i_n},i_n} \cdots T_{q_{i_1},i_1}(q_0{-}1) \le q'_0/q'_1$, thus the size of the interval of those $q_0$ is bounded by $C'_k \beta_k^{-n}$ for some~$C'_k$. 
This means that $[\mu_{L^kM(0\overline{1}),L^kM(\overline{1})}, \mu_{L^{k-1}M(\overline{0}),L^{k-1}M(1\overline{0})}] \cap E$ is covered, for each~$n$, by a polynomial number of intervals of size $C'_k \beta_k^{-n}$, hence the Hausdorff dimension of this set is zero.
Since $(1,2) \cap E$ is the union over $k\ge 1$ of these sets, $(1,2) \cap E$ also has zero Hausdorff dimension.
Finally, $(3/2,\infty) \cap E$ is the image of $(1,2) \cap E$ by the map $q_0 \mapsto 1{+}\frac{1}{2(q_0-1)}$, which is locally bi-Lipschitz, thus $E$ has zero Hausdorff dimension.
\end{proof}

The following proposition concludes the proof of our main results.

\begin{proposition}\label{p:GL2}
The statements (\ref{i:t14}) and (\ref{i:t15}) of Theorem~\ref{t:1} as well as the statements (\ref{i:t32}) and (\ref{i:t33}) of Theorem~\ref{t:3} are true.
\end{proposition}

\begin{proof}
For all $q_1 > \mathcal{G}(q_0)$, the set $U_{q_0,q_1}$ is infinite by Lemma~\ref{l:UV}~(\ref{i:UU1}) and because $U_{q_0,q_1} = \{0,1\}^\infty$ for $q_1 > \frac{q_0}{q_0-1}$, thus Theorem~\ref{t:1}~(\ref{i:t14}) holds.

\smallskip
Let now $q_1 = \mathcal{G}(q_0)$.
If $q_0 = \mu_{\bsigma(\overline{0}),\bsigma(\overline{1})}$ for some primitive $\bsigma \in \{L,R\}^\infty$, then $s(\ba_{q_0,q_1}) = \bsigma = s(\bb_{q_0,q_1})$, hence $V_{q_0,q_1}$ is uncountable (with zero entropy) by Proposition~\ref{p:lex}, thus $U_{q_0,q_1}$ is also uncountable (with zero entropy) by Lemma~\ref{l:UV}.
If $q_0 \in [\mu_{\sigma(\overline{0}),\sigma(1\overline{0})}, \mu_{\sigma(\overline{0}),\sigma(\overline{1})}]$, $\sigma \in \{L,R\}^*M$, then $\ba_{q_0,q_1} = \sigma(\overline{0})$, hence $s(\ba_{q_0,q_1}) = \sigma \overline{L} \le s(\bb_{q_0,q_1})$; by the proof of Theorem~\ref{t:lex}, each $\bu \in V_{q_0,q_1} \setminus \{\overline{0}, \overline{1}\}$ ends with $\sigma(\overline{0})$ and is therefore not in $U_{q_0,q_1}$.
Similarly, we have $U_{q_0,q_1} = \{\overline{0}, \overline{1}\}$ for $q_0 \in [\mu_{\sigma(\overline{0}),\sigma(\overline{1})}, \mu_{\sigma(0\overline{1}),\sigma(\overline{1})}]$.
We have already seen in Theorem~\ref{t:3}~(\ref{i:t31}) that $\{\mu_{\bsigma(\overline{0}),\bsigma(\overline{1})} \,:\, \bsigma \in \{L,R\}^\infty\, \mbox{primitive}\}$ has zero Hausdorff dimension, thus Theorem~\ref{t:3}~(\ref{i:t32}) holds.

\smallskip
Consider next $q_1 = \mathcal{K}(q_0)$.
If $q_0 \in \{\mu_{\sigma(\overline{0}),\sigma(1\overline{0})}, \mu_{\sigma(\overline{0}),\sigma(\overline{1})}\}$, $\sigma \in \{L,R\}^*M$, then $\mathcal{K}(q_0) = \mathcal{G}(q_0)$, thus $U_{q_0,q_1}$ is trivial by the preceding paragraph.
If $q_0 = \mu_{\bsigma(\overline{0}),\bsigma(\overline{1})}$ for some primitive $\bsigma \in \{L,M,R\}^\infty$, then $U_{q_0,q_1}$ is uncountable with zero entropy.
In all other cases, we have $s(\ba_{q_0,q_1}) = s(\bb_{q_0,q_1}) \in \{L,M,R\}^* \{\overline{L}, \overline{R}\}$ and $\mathcal{K}(q_0) > \mathcal{G}(q_0)$, thus $U_{q_0,q_1}$ is countably infinite by Proposition~\ref{p:lex} and the preceding paragraph.

Finally, let $q_1 > \mathcal{K}(q_0)$.
Then $s(\ba_{q_0,q_1}) > s(\bb_{q_0,q_1})$ by Proposition~\ref{p:K}, thus $U_{q_0, q_1}$ has positive entropy by Proposition~\ref{p:lex} and Lemma~\ref{l:UV}~(\ref{i:UU3}).
It remains to show that the Hausdorff dimension of $\pi_{q_0, q_1}(U_{q_0, q_1})$ is positive.
From the proof of Theorem~\ref{t:lex}, we see that $Y := \{\sigma(0(01)^k), \sigma(0(01)^{k+1})\}^\infty \subset U_{q_0, q_1}$ for some $\sigma \in \{L,M,R\}^*$, $k \ge 0$.
Then $\pi_{q_0, q_1}(Y)$ is the self-similar set generated by
\[
y_0(x) := r_0\, x + \pi_{q_0, q_1}(\sigma(0(01)^k)\overline{0}), \quad y_1(x) := r_1\, x + \pi_{q_0, q_1}(\sigma(0(01)^{k+1})\overline{0}),
\]
with $r_0 = q_0^{-|\sigma(0(01)^k)|_0} q_1^{-|\sigma(0(01)^k)|_1}$, $r_1 = q_0^{-|\sigma(0(01)^{k+1})|_0} q_1^{-|\sigma(0(01)^{k+1})|_1}$.
Since the elements of $Y$ are unique $(q_0,q_1)$-expansions, the iterated function system $\{y_0, y_1\}$ satisfies the Open Set Condition (OSC); see e.g.\ \cite[(9.12)]{Fal2003} for the definition of the OSC.
By applying \cite[Theorem~9.3]{Fal2003}, the Hausdorff dimension of $\pi_{q_0, q_1}(Y)$
is $\lambda>0$, where $\lambda$ satisfies $r_0^\lambda + r_1^\lambda = 1$.
\end{proof}

\section{Open problems} \label{sec:op}
We end this paper by formulating some open problems:
\begin{enumerate}[\upshape(i)]
\itemsep1ex
\item \label{op:1}
What is the growth rate of the number of possible prefixes of length $n$ of Thue--Morse--Sturmian words? 
Is it polynomial as for Sturmian words \cite[Theorem~2.2.36]{Lot2002}?
Using the proof of Proposition~\ref{p:GK}, this would imply that $U_{q_0,\mathcal{K}(q_0)}$ is countably infinite for all $q_0 > 1$ except a set of zero Hausdorff dimension, confirming the conjecture after Theorem~\ref{t:3}.
\item 
Is it possible to give a formula for the Hausdorff dimension of $\pi_{q_0, q_1}(U_{q_0, q_1})$ (in terms of the topological entropy $h(U_{q_0, q_1})$)?  
Are these functions continuous in $q_0$ and $q_1$?
\item
For fixed $q_0 > 1$, what are the maximal intervals (entropy plateaus) such that $h(U_{q_0, q_1})$ is constant?
We know from Theorems~\ref{t:1} that the first entropy plateau is $(1, \mathcal{K}(q_0)]$.
\item
For alphabet-systems $\mathcal{S} = \{(d_0,q_0),(d_1,q_1),\ldots, (d_m, q_m)\}$ with $m \ge 2$, what can be said about critical values?
\end{enumerate}

\subsection*{Acknowledgements}
The authors thank Ai-Hua Fan for valuable discussions on Bowen's entropy.


\begin{thebibliography}{SK}
\bibitem{Alc2014}
R. Alcaraz Barrera, \emph{Topological and ergodic properties of symmetric sub-shifts}, Discrete Contin. Dyn. Syst. {\bf 34} (2014), 4459–-4486.

\bibitem{AlcBarBakKon2019}
R. Alcaraz Barrera, S. Baker, D. Kong,
\emph{Entropy, topological transitivity, and dimensional properties of unique q-expansions},
Trans. Amer. Math. Soc. {\bf 317} (2019), 3209-3258.

\bibitem{All2017}
P. Allaart,
\emph{On univoque and strongly univoque sets},
Adv. Math. {\bf 308} (2017), 575--598.

\bibitem{AllBakKon2019}
P. Allaart, S. Baker, D. Kong,
\emph{Bifurcation sets arising from non-integer base expansions},
J. Fractal Geom.  {\bf 6} (2019), 301--341.

\bibitem{AllKon2019}
P. Allaart, D. Kong,
\emph{On the continuity of the Hausdorff dimension of the univoque set},
Adv. Math. {\bf 354} (2019), 106729, 24 pp.

\bibitem{AllKon2021a}
P. Allaart,  D. Kong,
\emph{Relative bifurcation sets and the local dimension of univoque bases},
Ergodic Theory Dynam. Systems {\bf 41} (2021) 2241--2273.

\bibitem{AllKon2021b}
P. Allaart, D. Kong,
\emph{On the smallest base in which a number has a unique expansion},
Trans. Amer. Math. Soc. {\bf 374} (2021), 6201--6249.

\bibitem{Bak2014}
S. Baker,
\emph{Generalized golden ratios over integer alphabets},
Integers {\bf 14} (2014), A15, 28 pp.

\bibitem{BakKon2020}
S. Baker, D. Kong, \emph{Two bifurcation sets arising from the beta transformation with a hole at 0}, Indag. Math. (N.S.) {\bf 31} (2020), 436--449.

\bibitem{BakSte2017}
S. Baker, W. Steiner,
\emph{On the regularity of the generalised golden ratio function},
Bull. Lond. Math. Soc. {\bf 49} (2017), no. 1, 58--70.

\bibitem{BarSteVin2014} 
M. Barnsley, W. Steiner, A. Vince, 
\emph{Critical itineraries of maps with constant slope and one discontinuity},
Math. Proc. Cambridge Philos. Soc. {\bf 157} (2014), no. 3, 547–565. 

\bibitem{BerDel2014}
V. Berth\'{e}, V. Delecroix,
\emph{Beyond substitutive dynamical systems: $S$-adic expansions},
Numeration and substitution 2012, 81--123, RIMS K\^oky\^uroku Bessatsu, B46, Res. Inst. Math. Sci. (RIMS), Kyoto, 2014.

\bibitem{Bow1973}
R. Bowen,
\emph{Topological entropy for noncompact sets},
Trans. Amer. Math. Soc. {\bf 184} (1973), 125--136.

\bibitem{Cla2016}
L. Clark, \emph{The $\beta$-transformation with a hole}, Discrete Contin. Dyn. Syst. {\bf36} (2016), 1249--1269.

\bibitem{Dev2009}
M. de Vries,
\emph{On the number of unique expansions in non-integer bases},
Topology Appl. {\bf 156} (2009), 652--657.

\bibitem{DevKom2009}
M. de Vries,  V. Komornik,
\emph{Unique expansions of real numbers},
Adv.  Math. {\bf 221} (2009), 390--427.

\bibitem{DevKomLor2022}
M. de Vries,  V. Komornik,  P. Loreti,
\emph{Topology of univoque sets in real base expansions},
Topology Appl. {\bf 312} (2022), 108085, 36 pp.

\bibitem{EggVanEyn1966}
L. C. Eggan, C. L. Vanden Eynden,
\emph{ ``Decimal'' expansions to nonintegral bases},
Amer. Math. Monthly {\bf 73} (1966), 576--582.

\bibitem{ErdHorJoo1991}
P. Erd\H{o}s, M. Horv\'{a}th,  I. Jo\'{o},
\emph{On the uniqueness of the expansions $1=\sum q^{-n_i}$},
Acta. Math. Hungar. {\bf 58} (1991), 333--342.

\bibitem{ErdJooKom1990}
P. Erd\H{o}s,  I. Jo\'{o},  V. Komornik,
\emph{Characterization of the unique expansions $1=\sum_{i=1}^\infty q^{-n_i}$ and related problems},
Bull. Soc.  Math. France {\bf 118} (1990), 377--390.

\bibitem{Fal2003}
K. Falconer,
\emph{Fractal Geometry. Mathematical Foundations and Applications}, second edition.
John Wiley \& Sons, Chichester, 2003.

\bibitem{Fic1992}
G. M. Fichtenholz, \emph{Differential- und Integralrechnung}, I, Twelfth edition, Johann Ambrosius Barth Verlag GmbH, Leipzig, 1992. 

\bibitem{GleSid2001}
P. Glendinning, N. Sidorov,
\emph{Unique representations of real numbers in non-integer bases},
Math. Res. Letters {\bf 8} (2001), 535--543.

\bibitem{GleSid2015}
P. Glendinning, N. Sidorov, 
\emph{The doubling map with asymmetrical holes}, 
Ergodic Theory Dynam. Systems {\bf 35} (2015), 1208--1228.

\bibitem{KalKonLanLi2020} 
C. Kalle, D. Kong, N. Langeveld,  W. Li, 
\emph{ The $\beta$-transformation with a hole at $0$}, 
Ergodic Theory Dynam. Systems {\bf 40} (2020), 2482--2514.

\bibitem{KalKonLiLu2019}
C. Kalle, D. Kong,  W. Li, F. L\"{u},
\emph {On the bifurcation set of unique expansions},
Acta Arith. {\bf 188} (2019), 367--399.

\bibitem{KomKonLi2017}
V. Komornik, D. Kong, W. Li,
\emph{Hausdorff dimension of univoque sets and devil's staircase},
Adv. Math. {\bf 305} (2017), 165--196.

\bibitem{KomLaiPed2011}
V. Komornik, A. C. Lai, M. Pedicini.
\emph{Generalized golden ratios of ternary alphabets},
J. Eur. Math. Soc. (JEMS) {\bf 13} (2011), 1113--1146.

\bibitem{KomLor1998}
V. Komornik, P.  Loreti,
\emph{Unique developments in non-integer bases},
Amer. Math.  Monthly {\bf 105} (1998), 636--639.

\bibitem{KomLuZou2022}
V. Komornik, J. Lu, Y. Zou,
\emph{Expansions in multiple bases over general alphabets},
Acta Math. Hungar. {\bf 166} (2022), 481--506.

\bibitem{KomPed2017}
V. Komornik, M. Pedicini,
\emph{Critical bases for ternary alphabets},
Acta Math. Hungar. {\bf 152} (2017), 25--57.

\bibitem{KonLiLuWanXu2020}
D. Kong, W. Li, F. L\"{u}, Z. Wang, J. Xu,
\emph{Univoque bases of real numbers: local dimension, devil's staircase and isolated points},
Adv. in Appl. Math. {\bf 121} (2020), 102103, 31 pp.

\bibitem{LabMor2006}
R. Labarca, C. G. Moreira,
\emph{Essential dynamics for Lorenz maps on the real line and the lexicographical world},
Ann. Inst. H. Poincar\'{e} Anal. Non Lin\'eaire {\bf 23} (2006), 683--694.

\bibitem{Li2021}
Y. Li,
\emph{Expansions in multiple bases},
Acta. Math. Hungar. {\bf 163} (2021), 576--600.

\bibitem{Lot2002}
M. Lothaire,
\emph{Algebraic combinatorics on words}.
Encyclopedia of Mathematics and its Applications, 90. Cambridge University Press, Cambridge, 2002.

\bibitem{Neu2021}
J. Neunh\"{a}userer,
\emph{Non-uniform expansions of real numbers},
Mediterr. J. Math. {\bf 18} (2021), 70, 8 pp.

\bibitem{Neu2024}
J. Neunh\"{a}userer,
\emph{Expansions of real numbers with respect to two integer bases},
Rocky Mountain J. Math., to appear.

\bibitem{Par1960}
W. Parry,
\emph{On the $\beta$-expansion of real numbers},
Acta. Math. Hungar. {\bf 11} (1960), 401--416.

\bibitem{Ped2005}
M. Pedicini,
\emph{Greedy expansions and sets with deleted digits},
Theoret. Comput. Sci. {\bf 332} (2005), 313--336.

\bibitem{Ren1957}
A. R\'enyi,
\emph{Representations for real numbers and their ergodic properties},
Acta. Math. Hungar. {\bf 8} (1957), 477--493.

\bibitem{RieSz1955}
F. Riesz, B. Sz.-Nagy, \emph{Functional analysis},  Frederick Ungar Publishing Co., New York, 1955.

\bibitem{Sid2003}
N. Sidorov,
\emph{Almost every number has a continuum of beta-expansions},
Amer. Math. Monthly {\bf 110} (2003), 838--842.

 \bibitem{Sid2014}
N. Sidorov, \emph{Supercritical holes for the doubling map}, Acta Math. Hungar. {\bf 143} (2014), 298--312.

\bibitem{Ste2020}
W. Steiner,
\emph{Thue-Morse-Sturmian words and critical bases for ternary alphabets},
Bull. Soc. Math. France {\bf 148} (2020), 597--611.

\bibitem{ZouLiLuKom2021}
Y. Zou, J. Li, J. Lu, V. Komornik,
\emph{Univoque graphs for non-integer base expansions},
Sci. China Math. {\bf 64} (2021), 2667--2702.
\end{thebibliography}
\end{document}